\documentclass[aos,preprint]{imsart}
\setattribute{journal}{name}{}
\RequirePackage[OT1]{fontenc}
\RequirePackage{amsthm,amsmath}
\RequirePackage[round]{natbib}
\RequirePackage[colorlinks,citecolor=blue,urlcolor=blue]{hyperref}
\newtheorem{theorem}{Theorem}[section]
\newtheorem{lemma}[theorem]{Lemma}
\newtheorem{remark}[theorem]{Remark}

\numberwithin{equation}{section}

\newcommand{\be}{\begin{equation}}
\newcommand{\ee}{\end{equation}}
\newcommand{\bqa}{\begin{eqnarray}}
\newcommand{\eqa}{\end{eqnarray}}
\newcommand{\bqn}{\begin{eqnarray*}}
\newcommand{\eqn}{\end{eqnarray*}}
\newcommand{\bdes}{\begin{description}}
\newcommand{\edes}{\end{description}}
\newcommand{\bitem}{\begin{itemize}}
\newcommand{\eitem}{\end{itemize}}
\newcommand{\bnum}{\begin{enumerate}}
\newcommand{\enum}{\end{enumerate}}
\newcommand{\bsubn}{\begin{subnumcases}}
\newcommand{\esubn}{\end{subnumcases}}

\newcommand{\ga}{\alpha}

\newcommand{\gl}{\lambda}

\newcommand{\gD}{\Delta}

\newcommand{\bfell}{{\mbox{\boldmath$\ell$}}}

\newcommand{\bgS}{{\mbox{\boldmath$\Sigma$}}}
\newcommand{\bga}{{\mbox{\boldmath$\alpha$}}}

\newcommand{\rtr}{{\rm tr}}

\newcommand{\non}{\nonumber\\}

\newcommand{\bbA}{{\bf A}}
\newcommand{\bba}{{\bf a}}
\newcommand{\bbB}{{\bf B}}

\newcommand{\bbC}{{\bf C}}

\newcommand{\bbe}{{\bf e}}
\newcommand{\bbE}{{\bf E}}

\newcommand{\bbI}{{\bf I}}

\newcommand{\bbj}{{\bf j}}
\newcommand{\bbJ}{{\bf J}}

\newcommand{\bbM}{{\bf M}}

\newcommand{\bbQ}{{\bf Q}}

\newcommand{\bbP}{{\bf P}}

\newcommand{\bbu}{{\bf u}}

\newcommand{\bbv}{{\bf v}}

\newcommand{\bbX}{{\bf X}}

\newcommand{\bbx}{{\bf x}}

\newcommand{\bbY}{{\bf Y}}

\newcommand{\bay}{\begin{array}}
\newcommand{\eay}{\end{array}}

\newcommand{\E}{{\mathbb{E}}}

\newcommand{\tr}{{\rm{tr}}}
\newcommand{\asto}{\stackrel{a.s.}{\to}}

\usepackage{mathrsfs,color,amsfonts,graphicx,subfigure,bm}
\usepackage{multirow}

\begin{document}

\begin{frontmatter}
\title{
Strong Consistency of the AIC, BIC, $C_p$ and
KOO Methods in High-Dimensional Multivariate Linear Regression
}
\runtitle{Strong consistency of  AIC, BIC, $C_p$ and
KOO methods}

\begin{aug}
\author{\fnms{Zhidong} \snm{Bai}\thanksref{t1,m1}\ead[label=e1]{baizd@nenu.edu.cn}},
\author{\fnms{Yasunori} \snm{Fujikoshi}\thanksref{t2,m2}\ead[label=e2]{fujikoshi\_y@yahoo.co.jp}}
\and
\author{\fnms{Jiang} \snm{Hu}\thanksref{t3,m1}
\ead[label=e3]{huj156@nenu.edu.cn}}

\thankstext{t1}{Supported by NSFC 11571067 and 11471140.}
\thankstext{t2}{Supported by NSFC 11771073.}
\thankstext{t3}{Supported by the Ministry of Education, Science, Sports, and Culture, a Grant-in-Aid for Scientific Research (C), \#16K00047,  2016-2018.}
\runauthor{Z. D. Bai et al.}

\affiliation{Northeast Normal University\thanksmark{m1} and Hiroshima University\thanksmark{m2}}

\address{KLASMOE and School of Mathematics \& Statistics\\ Northeast Normal University, Changchun, 
China.\\
\printead{e1};~~\printead*{e3}
\phantom{E-mail:\ }}

\address{Department of Mathematics\\
 Graduate School of Science\\
  Hiroshima University, Hiroshima, Japan.\\
\printead{e2}}
\end{aug}

\begin{abstract}
Variable selection is essential for improving inference and interpretation in  multivariate linear regression.  
Although a number of alternative regressor selection criteria have been suggested, the most prominent and widely used are the Akaike information criterion (AIC), Bayesian information criterion (BIC), Mallow's $C_p$, and their modifications. However, for the data with high dimensionality in both responses and covariates, experience has shown that the performance of these classical criteria is not always satisfactory. 
In the present article,  we begin by presenting the necessary and sufficient conditions (NSC) for the strong consistency of the high-dimensional AIC, BIC, and $C_p$, based on which we can identify some reasons for their poor performance. Specifically, we show that under certain mild high-dimensional conditions, if the BIC is strongly consistent, then the AIC is strongly consistent, but not vice versa. This result contradicts the classical understanding. 
 In addition, we consider some NSC for the strong consistency of the high-dimensional kick-one-out (KOO) methods introduced by \cite{ZhaoK86D} and \cite{NishiiB88S}. 
Furthermore, we propose two general methods based on the KOO methods  and prove their strong consistency. The proposed general methods remove  the penalties while simultaneously reducing the conditions for the dimensions and sizes of the regressors. Simulation studies and real data analysis support our  conclusions and show that the convergence rates of the two proposed general KOO methods are much faster than those of the original ones.
\end{abstract}



\end{frontmatter}
\section{ Introduction}

In multivariate statistical analysis, the most general and favorable model to investigate the relationship between a  predictor matrix  $\bbX$ and a response matrix  $\bbY$ is the multivariate linear regression (MLR) model.  More specifically,
let 
\begin{align}\label{fullmodel}
	\bbY=\bbX\Theta+\bbE,
\end{align}
where $\bbY=(y_{ij}): n\times p$ (the responses), $\bbX=(\bbx_1,\dots,\bbx_k): n\times k$ (the predictors), $\Theta=(\bm\theta_1, \dots,\bm\theta_k)': k\times p$ (the regression coefficients), and $\bbE=(\bbe_1,\dots,\bbe_p)=(e_{ij}): n\times p$ (the random errors).
The goal in MLR analysis is to estimate the regression coefficients $\Theta$. The estimates should be such that the
estimated regression plane explains the variation in the
values of the responses with great accuracy.
The classical  linear least-squares solution is to estimate the matrix of regression coefficients $\hat\Theta$ by
\begin{align*}
	\hat\Theta=(\bbX'\bbX)^{-1}\bbX'\bbY.
\end{align*}
However, model \eqref{fullmodel} (referred to hereinafter as the full model) is not always a good model because some of the predictors may be uncorrelated with the responses, i.e., the corresponding rows of $\Theta$ are zeros.  
The Akaike information criterion (AIC), Bayesian information criterion (BIC), and Mallows's $C_p$  are among the most popular and versatile strategies for model selection among the predictors. 

Let $\bbj$ be a subset of ${\bm \omega}=\{1,2,\cdots,k\}$ and $\bbX_{\bbj}=(\bbx_j, j\in\bbj)$ and $\Theta_{\bbj}=(\bm\theta_j, j\in\bbj)'$.  Denote model $\bbj$ by
 \be\label{eq1}
 M_{\bbj}:\ \ \bbY=\bbX_{\bbj}\Theta_{\bbj}+\bbE.
 \ee
 Akaike's seminal paper \citep{Akaike73I} proposed the use of the Kullback-Leibler distance as a fundamental basis for model selection known as the AIC, which is defined as follows:
\begin{equation}\label{AIC}
  A_{\bbj}=n\log(|\widehat \bgS_{\bbj}|)+2[k_{\bbj}p+\frac12p(p+1)]+np(\log(2\pi)+1),
\end{equation}
where
\begin{align}\label{nhatS}
  n\widehat\bgS_{\bbj}=\bbY'\bbQ_{\bbj}\bbY,~~\bbQ_{\bbj}=\bbI_n-\bbP_{\bbj},~~
\bbP_{\bbj}=\bbX_{\bbj}(\bbX_{\bbj}'\bbX_{\bbj})^{-1}\bbX_{\bbj}',
\end{align}
and $k_{\bbj}$ is the cardinality of subset $\bbj$. 
Note that $\bbP_{\bbj}$ is orthogonal projection of rank $k_{\bbj}$ onto the subspace spanned by $\bbX_{\bbj}$ and  $\bbQ_{\bbj}$ is the orthogonal projection of rank $n-k_{\bbj}$   onto the
orthogonal complement subspace panned by $\bbX_{\bbj}$. 
The BIC, which is also known as the Schwarz criterion, was proposed by \cite{Schwarz78E} in the form of a penalized log-likelihood function, in which the penalty is equal to the logarithm of the sample size times the number of estimated parameters in the model, i.e., 
 \begin{equation}\label{BIC}
  B_{\bbj}=n\log(|\widehat \bgS_{\bbj}|)+\log(n)[k_{\bbj}p+\frac12p(p+1)]+np(\log(2\pi)+1).
\end{equation}
A criterion with behavior similar to that of the AIC for variable selection in regression models is Mallows's $C_p$ proposed by \cite{Mallows73S}, which is defined as follows:
\begin{equation}\label{Cp}
  C_{\bbj}=(n-k)\rtr(\widehat\bgS_{\bm \omega}^{-1}\widehat\bgS_{\bbj})+2pk_{\bbj}.
\end{equation}
  Refer to \citep{Fujikoshi83C,SparksC83M,NishiiB88S} for additional details of formulas \eqref{AIC}, \eqref{BIC}, and \eqref{Cp}.
Then, the AIC, BIC, and $C_p$ rules are used to select
 \begin{align}\label{hatABC}
 	 \hat \bbj_A=\arg \min A_{\bbj},\quad \hat \bbj_B=\arg \min B_{\bbj}\quad\mbox{and}\quad
 \hat\bbj_C= \arg\min C_{\bbj},
 \end{align}
 respectively. 
 
 If the data are generated from a  model (referred to hereinafter as a true model) that is one of the candidate models, then we would apply a model selection method to identify the true model. Then, some optimality, such as consistency, is desirable for model selection. A model selection method is weakly consistent if, with probability tending to one, the selection method is able to select the true model from the candidate models. Strong consistency means that the true model tends to almost surely be selected. Strong consistency implies weak consistency but not vice versa. Thus, strong consistency can provide a deeper understanding of the selection methods. 
 Under a large-sample asymptotic framework, i.e., dimension $p$ is fixed and $n$ tends to infinity, the AIC and $C_p$ are not  consistent \citep{Fujikoshi85S,FujikoshiV79E}, but the BIC is strongly consistent \citep{NishiiB88S}. 
 However, in recent years, statisticians have increasingly noticed that these properties cannot be adapted to high-dimensional data. In particular, experience has shown that the classical model selection criteria tend to select more variables than necessary when $k$ and $p$ are large. For the case in which $k$ is fixed, $p$ is large but smaller than $n$, and $p/n\to c\in [0,1)$, which is referred to as a large-sample and large-dimensional asymptotic framework, the BIC has been shown to be not consistent, but the AIC and $C_p$ are weakly consistent under certain conditions (see, e.g., \citep{FujikoshiS14C,YanagiharaW15C,Yanagihara15C}).  

To clarify the model selection methods, in the present paper, we focus on the strongly consistent properties 
under a large-model, large-sample, and large-dimensional (LLL) asymptotic framework, i.e., $
\min\{k,p,n\}$ tends to infinity for the case in which $ p/n\to c\in(0,1)$, $ k/n\to\ga\in(0,1-c)$, and the true model size is fixed.
We do not intend to judge the advantages and disadvantages of the existing selection methods in the MLR model beyond their consistency properties in the present paper, because these advantages and disadvantages depend on their intended applications and on the nature of the data. Our goal is to explain the theoretical insights of the classical selection methods and modified methods under an LLL framework, and we hope that this article will stimulate further research that will provide an even clearer understanding of high-dimensional variable selection. 
Here, we refer to three recent reviews comparing the variable selection methods \citep{AnzanelloF14R,BleiK17V,HeinzeW18V}. In addition, note that in the present paper, we assume that $n -k > p$. A number of studies have examined sparse and penalized methods for high-dimensional data for which this condition is not satisfied, see, e.g., \citep{LiN15M,ZouH05R}.

We now describe the four main contributions of the present paper. 
\begin{itemize}
	\item First, in Section \ref{ABC}, we present the necessary and sufficient conditions (NSC) for the strong consistency of variable selection methods based on the AIC, BIC, and $C_p$ under an LLL asymptotic framework, including the ranges of $c$ and $\alpha$, the moment condition of random errors (our results do not require the normality condition), and the convergence rate of the noncentrality matrix. Specifically, on the basis of these results, we conclude that under an LLL asymptotic framework, if the BIC is strongly consistent, then the AIC is strongly consistent, but not vise versa, which contradicts the classical understanding. 
	\item Second, in Section \ref{KOO}, we examine the strongly consistent properties of the kick-one-out (KOO) methods based on the AIC, BIC, and $C_p$ under an LLL asymptotic framework, which were introduced by \cite{ZhaoK86D}  and \cite{NishiiB88S} and followed by \cite{FujikoshiS18C}. The KOO methods, which were proposed for the computation problem in the classical AIC, BIC, and $C_p$, reduce the number of computational statistics from  $2^k-1$ to $k$. In addition, \cite{NishiiB88S} showed that under a large-sample asymptotic framework, the KOO methods share the same conditions and strong consistency of the classical  AIC and BIC. However, in the present paper, we find that under an LLL  asymptotic framework, the KOO methods have higher costs for dimension conditions than do the AIC, BIC, and $C_p$ for strong consistency. 
\item Third, on the basis of the KOO methods, in Section \ref{GKOO}, we propose two general KOO methods that not only remove  the penalty terms but also reduce the conditions for the dimensions and sizes of the predictors. Furthermore, the sufficient condition is given for their strong consistency.  The proposed general KOO methods have considerable advantages, such as simplicity of expression, ease of computation, limited restrictions, and fast convergence. 
\item Fourth, random matrix theory (RMT) is introduced to model selection methods in high-dimensional MLR.  The new theoretical results and the concepts behind their proofs are applicable to numerous other model selection methods, such as the modified AIC and $C_p$ \citep{FujikoshiS97M,Bozdogan87M}. Furthermore, the technical tool developed in the present paper is applicable to future research, e.g., the 
growth curve model \citep{EnomotoS15C, FujikoshiE13H},
multiple discriminant analysis \citep{Fujikoshi83C,FujikoshiS16H},
principal component analysis \citep{FujikoshiS16S,BaiC18C}, and canonical correlation analysis \citep{NishiiB88S,BaoH18C}.
\end{itemize}

The remainder of this paper is organized as follows. In Section 2, we present the necessary notation and assumptions for the MLR model under an LLL asymptotic framework. The main results on the strong consistency of the AIC, BIC, $C_p$, KOO, and general KOO methods are stated in Section 3. We  present some simulation studies and  two real data examples  in Section 4 and  Section 5 respectively to illustrate the performance of our results. 
The main theorems are proven in Sections 6 and 7, and other potential applications are briefly discussed in Section 8. Finally, additional technical results are present in the Appendix.

\section{Notation and Assumptions} 
In this section, we introduce the notation and assumptions required for our main results to hold. 
Recall the  MLR model (\ref{fullmodel})
\bqn
M:\ \ \bbY=\bbX\Theta+\bbE,
\eqn
$\bbj$ is a subset of ${\bm \omega}=\{1,2,\cdots,k\}$, $k_{\bbj}$ is  the cardinality of  set $
\bbj$, $\bbX_{\bbj}=\{\bbx_j, j\in\bbj\}$,  and $\Theta_{\bbj}=\{\theta_j, j\in\bbj\}$.
Model $\bbj$  is denoted by
 \be\label{eq1}
 M_{\bbj}:\ \ \bbY=\bbX_{\bbj}\Theta_{\bbj}+\bbE.
 \ee
 Denote the true model as $\bbj_*$, and
 $$M_{\bbj_*}:\ \ \bbY=\bbX_{\bbj_*}\Theta_{\bbj_*}+\bbE.$$ 
We first suppose that the model and the random errors satisfy the following conditions: 
\begin{itemize}
\item[(A1):] The true model $\bbj_*$ is  a subset of  set  ${\bm \omega}$, and $k_*:=k_{\bbj_*}$ is fixed.
\item[(A2):]  $\{e_{ij}\}$  are independent and identically distributed (i.i.d.) with zero means, unit 
 variances, and  finite fourth moments.  
 \end{itemize}
Note that in a typical multivariate regression model, the rows of $\bbE$ are assumed to be i.i.d. from a $p$-variate distribution with zero mean and covariance matrix $\Sigma$. Since we consider the distribution of a statistic invariant under the transformation $\bbY$ $\to$ $\bbY \Sigma^{-1/2}$, without loss of generality, we may assume that $\Sigma=\bbI_p$ by replacing $\Theta$ with $\Theta \Sigma^{-1/2}$. Moreover, the finite fourth moments condition is required only for the technical proof; we believe finite second moments are sufficient.

We assume that 
\begin{itemize}
\item[(A3):] $\bbX'\bbX$ is positive  definite.
 \end{itemize}
 Note that if assumption (A3) is satisfied, then for any $\bbj\subset\bm\omega$, $\bbX_{\bbj}'\bbX_{\bbj}$ is invertible because $\bbX_{\bbj}'\bbX_{\bbj}$ is a principal submatrix of $\bbX'\bbX$.  

 In the present paper, we focus primarily on an LLL asymptotic framework, which is specified as follows.
 \begin{itemize}
\item[(A4):]  Assume that as $\{k,p,n\}\to\infty$, $ c_n:=p/n\to c\in(0,1)$ and $\alpha_k:= k/n\to\ga\in(0,1-c)$.
\end{itemize}
We assume that $c$ and $\alpha$ are larger than 0 because we can take  $\alpha$ and $c$ as two unknown parameters, the consistent estimators of which are  $\alpha_k$ and $c_n$, respectively, no information for the convergence of $\{k,p,n\}$ exists for any dataset at hand, and
 $\alpha_k$ and $c_n$ are always positive. 
  Therefore,  the assumption that  $c$ and $\alpha$ are positive is reasonable. In addition, if the model size $k$  is bigger than the sample size $n$,  one can use the screening methods to reduce it to a relatively large scale that satisfies assumption (A4), e.g., sure independence screening method based on the distance correlation \citep{LiZ12F}, interaction pursuit via distance correlation \citep{KongL17I}. For the screening methods, we refer to \citep{FanL08S,FanL10S} and their citations for more details. 
Here one should notice that not all the existing variable screening methods can perform well with multiple responses. 

 Next, we present additional notation that is frequently used herein.
 Denote
 \bqn
 \bbJ_+=\{\bbj: \bbj\supset \bbj_*\},~
 \bbJ_-=\{\bbj: \bbj \not\supset \bbj_*\}~\mbox{and}~
 \bbJ=\bbJ_-\cup\bbj_*\cup\bbJ_+.
 \eqn
In the following, we use the terms overspecified model and underspecified model to indicate whether a model $ \bbj$ includes the true model (i.e., $\bbj\in  \bbJ_+$) or not (i.e., $\bbj\in  \bbJ_-$).
 If $\bbj$ is an overspecified model and $k_{\bbj}-k_{\bbj_*}=m>0$, then subsets of models $\bbj=\bbj_0\supset\bbj_{-1}\supset\cdots\supset\bbj_{-m}=\bbj_*$ exist such that each consecutive pair decreases one and only one index,
i.e., $\bbj_t=\bbj_0\backslash\{j_{1},\cdots,j_{-t}\}$ for $t=0,-1,\dots, -m$, which means that $j_
{-t}$ is in $\bbj_{t+1}$ but not in $\bbj_t$.
If  $\bbj$ is an underspecified model, denote $\bbj_-=\bbj\cap \bbj_*$, $\bbj_+=\bbj\cap \bbj_*^c$ and write the elements in $\bbj_*\cap\bbj_-^c$ as $i_1,\cdots, i_s$ and the elements in $\bbj_+$ as $j_1,\cdots,j_m$. Define the model index set $\bbj_t=\bbj\cup\{i_{t+1},\cdots,i_{s}\}$ for $t=0,1,\cdots, s$ with convention that $\bbj_{s}=\bbj$, which also indicates that $i_
{t}$ is in $\bbj_{t+1}$ but not in $\bbj_t$. Moreover, we can define $\bbj_t=\bbj_0\backslash\{j_{t},\cdots,j_{-t}\}$ for $t=0,-1,\dots, -m$, but we should note that in this case, $\bbj_0 \neq  \bbj$. The positive subscript of $\bbj$ indicates the addition of indexes, and, correspondingly, the negative subscript of $\bbj$ indicates the removal of indexes. 
In addition, we denote $\bba_t=\bbQ_{\bbj_{_t}}\bbx_{i_t}/\|\bbQ_{\bbj_t}\bbx_{i_t}\|$ for $t>0$ and  $\bba_t=\bbQ_{\bbj_t}\bbx_{j_{-t}}/\|\bbQ_{\bbj_t}\bbx_{j_{-t}}\|$ for $t<0$.
Thus, for any integer $t$, the following two equations are straightforward: 
\begin{align}\label{Pat}
  \bbP_{\bbj_{t+1}}=\bbP_{\bbj_t}+\bba_t\bba_t'\end{align}
and
\begin{align}\label{Qat}
\bbQ_{\bbj_{t+1}}=\bbQ_{\bbj_t}-\bba_t\bba_t'.
\end{align}
We hereinafter denote the spectral norm for a matrix by $\|\cdot\|$. 
For the underspecified model, our results require another assumption. Denote 
\begin{align*}
  \Phi=\frac1n\Theta'_{\bbj_*}\bbX_{\bbj_*}'\bbX_{\bbj_*}\Theta_{\bbj_*} ~\mbox{and}~  \Phi_\bbj=\frac1n\Theta'_{\bbj_*}\bbX_{\bbj_*}'\bbQ_{\bbj}\bbX_{\bbj_*}\Theta_{\bbj_*},
\end{align*}
we assume that
\begin{itemize}
\item[(A5):]   $\|\Phi\|$ is bounded uniformly in $n$.  
 \end{itemize} 
   Note that if assumption (A5) is satisfied, then for any $\bbj\subset\bm\omega$, $\|\Phi_\bbj\|\leq \|\Phi\|$  is bounded uniformly. On the other hand, if assumption (A5) does not hold, then the following assumption is considered:
  \begin{itemize}
  \item[(A5'):]   For any $\bbj\in \bbJ_-$, as $\{k,p,n\}\to\infty$, $\|\Phi_\bbj\|\to\infty$.
 \end{itemize} 
 Notice that if (A5) does not hold, there may also exist some
 $\bbj\in \bbJ_-$ such that $\|\Phi_\bbj\|$ is bounded uniformly, which does not satisfy (A5'). But, this case can be considered by combining the results under the assumptions (A5) and (A5'). Thus, in this paper we omit the detail results in the special situation.  
Throughout the present paper,  we use $o_{a.s}(1)$ to denote almost surely scalar negligible entries.
  
 \section{Main results}
 
 In this section, we present the main results of the present paper, including the strong consistency of the AIC, BIC,  $C_p$, KOO methods and general KOO methods.
First, we present some notation which will be used in the following frequently. For $\bbj\in \bbJ_-$ with  $k_{\bbj_+}=m\geq0$ and $k_{\bbj_*\cap\bbj_-^c}=s>0$,  we denote   
\begin{align*}
  \tau_{n\bbj}&:=(1-\alpha_m)^{s-p}|(1-\alpha_m)\bbI+\Phi_{\bbj}|\to\tau_{\bbj}\leq\infty\\
  \kappa_{n\bbj}&:=\tr(\Phi_{\bbj})\to\kappa_{\bbj}\leq\infty.
  \end{align*}
  We hereinafter denote $\alpha_m:=m/n$, and due to assumption (A4), if  $m\leq k$,  then we know that $\lim\alpha_m\leq\alpha$.
  It is obvious that if assumption (A5) holds, $\tau_{\bbj}$ and $\kappa_{\bbj}$ are bounded for all $\bbj\in \bbJ_-$. 
  For  simplicity, we denote $\tau_{\bbj}^0:=\lim_{p,n}|\bbI+\Phi_{\bbj}|$, which equals $\tau_{\bbj}$ when  $\alpha_m\to0$.  
 

\subsection{Strong consistency of the AIC, BIC, and $C_p$}\label{ABC}
Now, we are in position to present our main results concerning the strong consistency of the AIC, BIC, and $C_p$.
Let 
  \begin{align*}
  \phi(\alpha,c)&:= 2c\ga+\log\left(\frac{(1-c)^{1-c}(1-\ga)^{1-\ga}}{(1-c-\ga)^{1-c-\ga}}\right)\\
  \psi(\ga,c)&:=\frac{c(\ga-1)}{1-\ga-c}+2c,
\end{align*}
then we have the following theorems. 
\begin{theorem}\label{th1}
Suppose that assumptions {\rm (A1) through (A5)} hold.  
\begin{itemize}
	\item[{\rm (1)}] For $\phi(\alpha,c)>0$,  
	\begin{itemize}
	\item  if for any $\bbj\in \bbJ_-$ with $k_{\bbj_+}=m\geq0$ and        $k_{\bbj_*\cap\bbj_-^c}=s>0$, such that $\log(\tau_{\bbj})>(s-m)(\log(1-c)+2c)$, then the variable selection method based on the AIC is strongly consistent;
	\item   if there exists some $\bbj\in \bbJ_-$ with $k_{\bbj_+}=m\geq0$ and        $k_{\bbj_*\cap\bbj_-^c}=s>0$,  such that $\log(\tau_{\bbj})< (s-m)(\log(1-c)+2c)$,  then the variable selection method based on the AIC is almost surely under-specified. 
	\end{itemize}
	If $\phi(\alpha,c)< 0$, then the variable selection method based on the AIC is almost surely over-specified. 

  \item[{\rm (2)}] The variable selection method based on the BIC is almost surely under-specified. 
  \item[{\rm (3)}] For $\psi(\alpha,c)>0$,
  \begin{itemize}
  \item if for any $\bbj\in \bbJ_-$ with $k_{\bbj_+}=m\geq0$ and         $k_{\bbj_*\cap\bbj_-^c}=s>0$, such that $\kappa_{\bbj}>(s-m)\psi(\alpha,c)(1-\alpha-c)/(1-\alpha)$, then  the variable selection method based on $C_p$ is strongly consistent;
  \item if there exists some $\bbj\in \bbJ_-$ with $k_{\bbj_+}=m\geq0$ and        $k_{\bbj_*\cap\bbj_-^c}=s>0$,  such that $\kappa_{\bbj}<(s-m)\psi(\alpha,c)(1-\alpha-c)/(1-\alpha)$, then the variable selection method based on $C_p$ is almost surely under-specified.
  \end{itemize}
  If $\psi(\alpha,c)< 0$,  then the variable selection method based on $C_p$ is almost surely over-specified. 

\end{itemize}
\end{theorem}
The proof of this theorem is  presented in Section \ref{proofths}. If assumption (A5) does not hold, we instead consider assumption (A5') and have the following theorem.
\begin{theorem}\label{th2}
	Suppose that assumptions {\rm (A1) through (A4)} and {\rm (A5')} hold.  
	\begin{itemize}
	\item[{\rm (1)}] If $\phi(\alpha,c)>0$, then the variable selection method based on the AIC is strongly consistent; otherwise,  if $\phi(\alpha,c)<0$, then  the variable selection method based on the AIC is almost surely over-specified.
  \item[{\rm (2)}]
  For any $\bbj\in \bbJ_-$ satisfying $k_{\bbj_+}=m\geq0$,        $k_{\bbj_*\cap\bbj_-^c}=s>0$ and $m-s<0$, if $\lim_{n,p}
  \Big(\log(\tau_{n\bbj})-c(s-m)\log(n)\Big)>(s-m)\log(1-c),$ then 
  the variable selection method based on the BIC is strongly consistent.
	If for some $\bbj\in \bbJ_-$ satisfying  $m-s<0$,  $\lim_{n,p}
  \Big(\log(\tau_{n\bbj})-c(s-m)\log(n)\Big)<(s-m)\log(1-c),$ then
  the variable selection method based on the BIC is almost surely under-specified.   
  \item[{\rm (3)}]  If $\psi(\alpha,c)>0$, then the variable selection method based on the $C_p$ is strongly consistent; otherwise,  if $\psi(\alpha,c)>0$, then the variable selection method based on the $C_p$ is almost surely over-specified.
\end{itemize}
\end{theorem}
The proof of this theorem is also presented in Section \ref{proofths}.
\begin{remark}
	In the two theorems,  $\alpha_k$ and $c_n$, respectively, are typically used instead of $\alpha$ and $c$ for application, because for a real dataset, we do not have information regarding their limits. 
\end{remark}
\begin{remark}
	3D plots are presented in Figure \ref{fig1} to illustrate the ranges of $\alpha$ and $c$ such that $\phi(\alpha,c)>0$ and  $\psi(\alpha,c)>0 $. This figure shows that large $\alpha$ and $c$ both result in overestimation of the true model. Moreover, \cite{FujikoshiS14C,YanagiharaW15C} proved that for the fixed-$k$ case, the consistency ranges of  $c$ for the AIC and $C_p$ are $[0, 0.797)$ and $[0,1/2)$, respectively, which coincide with our results in Lemma \ref{le1} when $\alpha_k\to0$. 
\end{remark}
\begin{remark}
	Combining Theorems \ref{th1} and \ref{th2}, under an LLL asymptotic framework,  if the BIC is strongly consistent, then the AIC is strongly consistent but not vice versa. This result contradicts the classical understanding that under a large-sample asymptotic framework, the AIC and $C_p$ are not consistent, but the BIC is strongly consistent. 
	\end{remark}
\begin{figure}[h]
	\subfigure{
		\includegraphics[width=6cm,height=5cm]{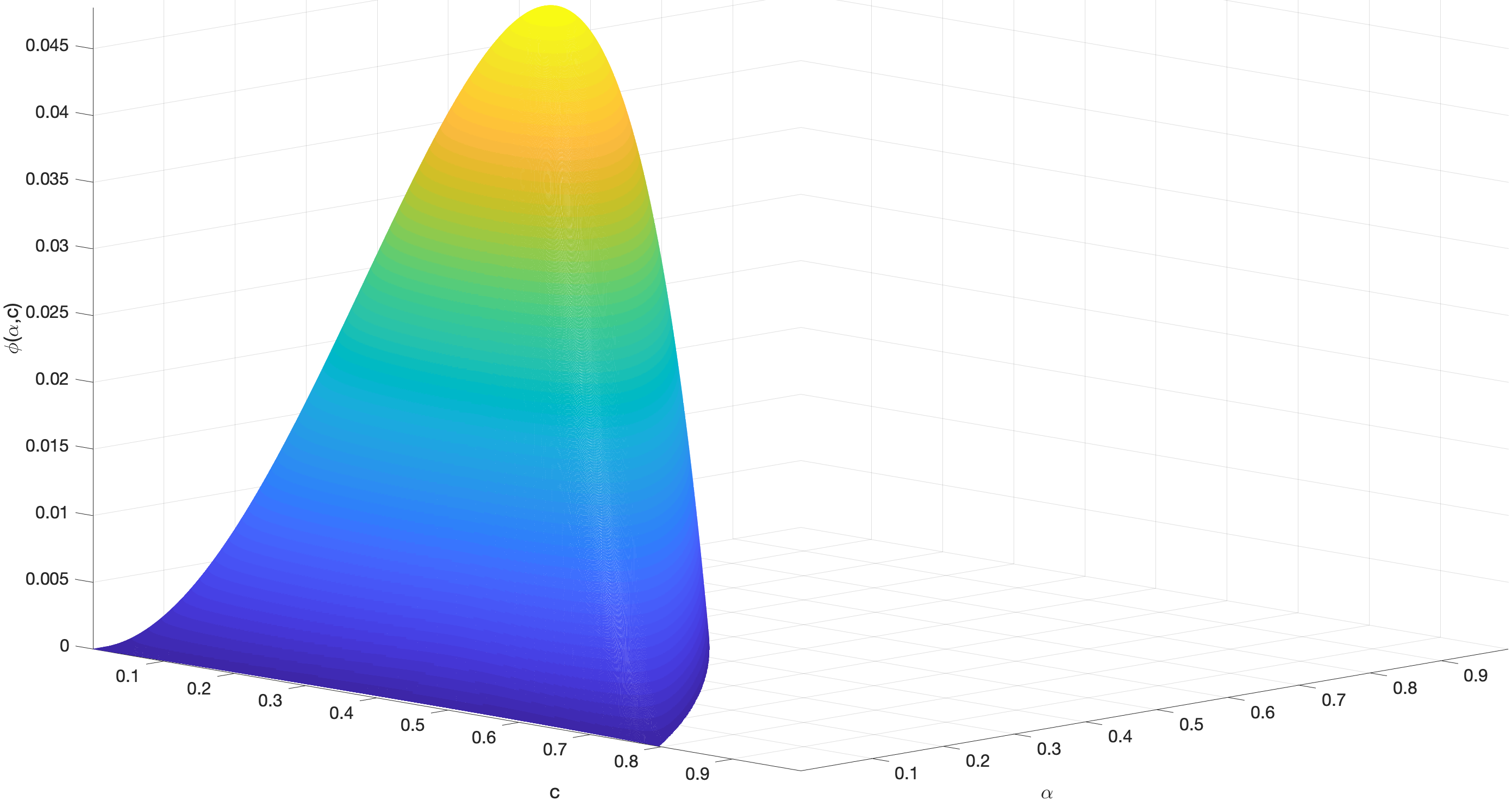}}
\subfigure{
		\includegraphics[width=6cm,height=5cm]{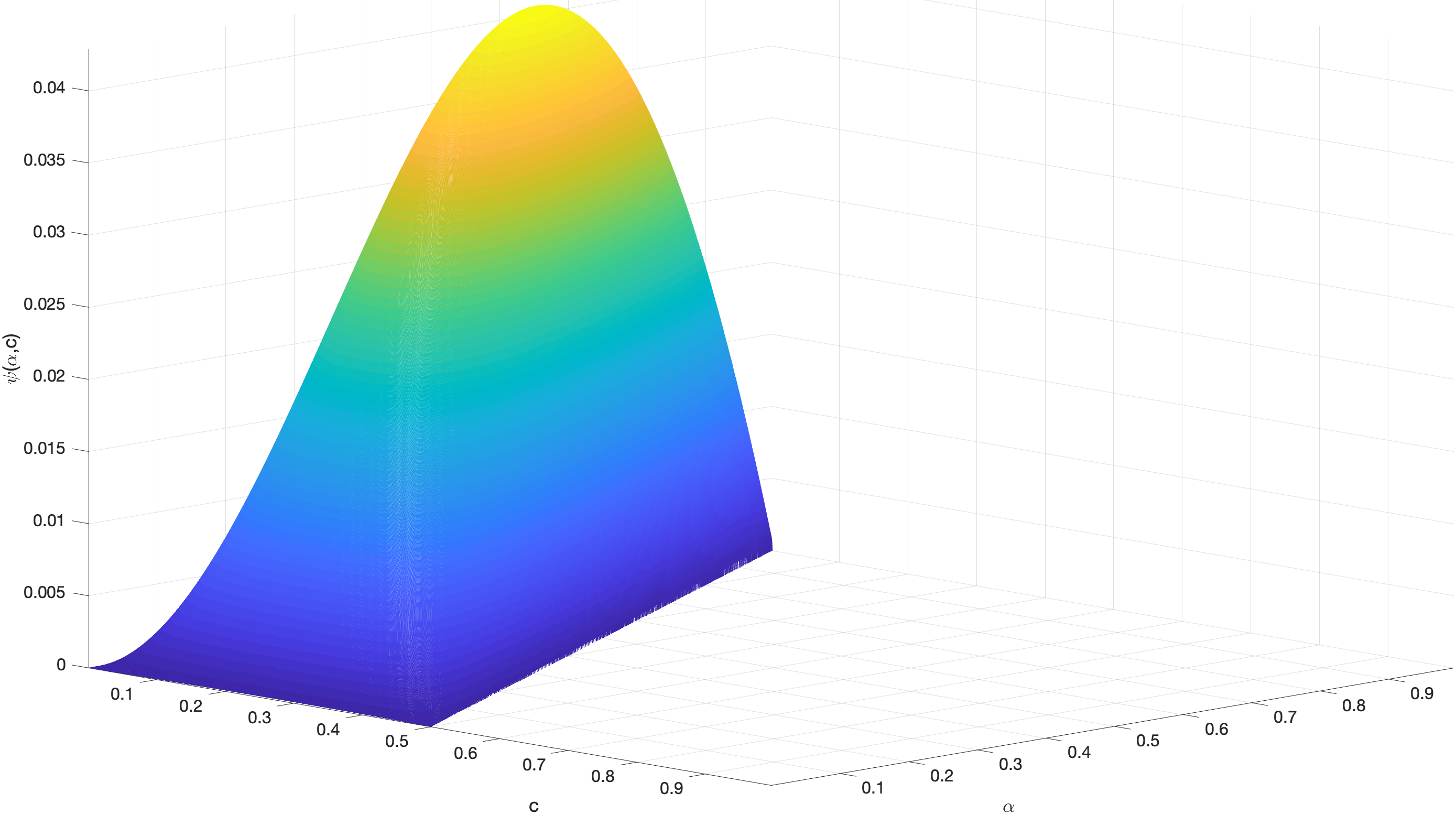}}
		\subfigure{
			\includegraphics[width=6cm,height=2.5cm]{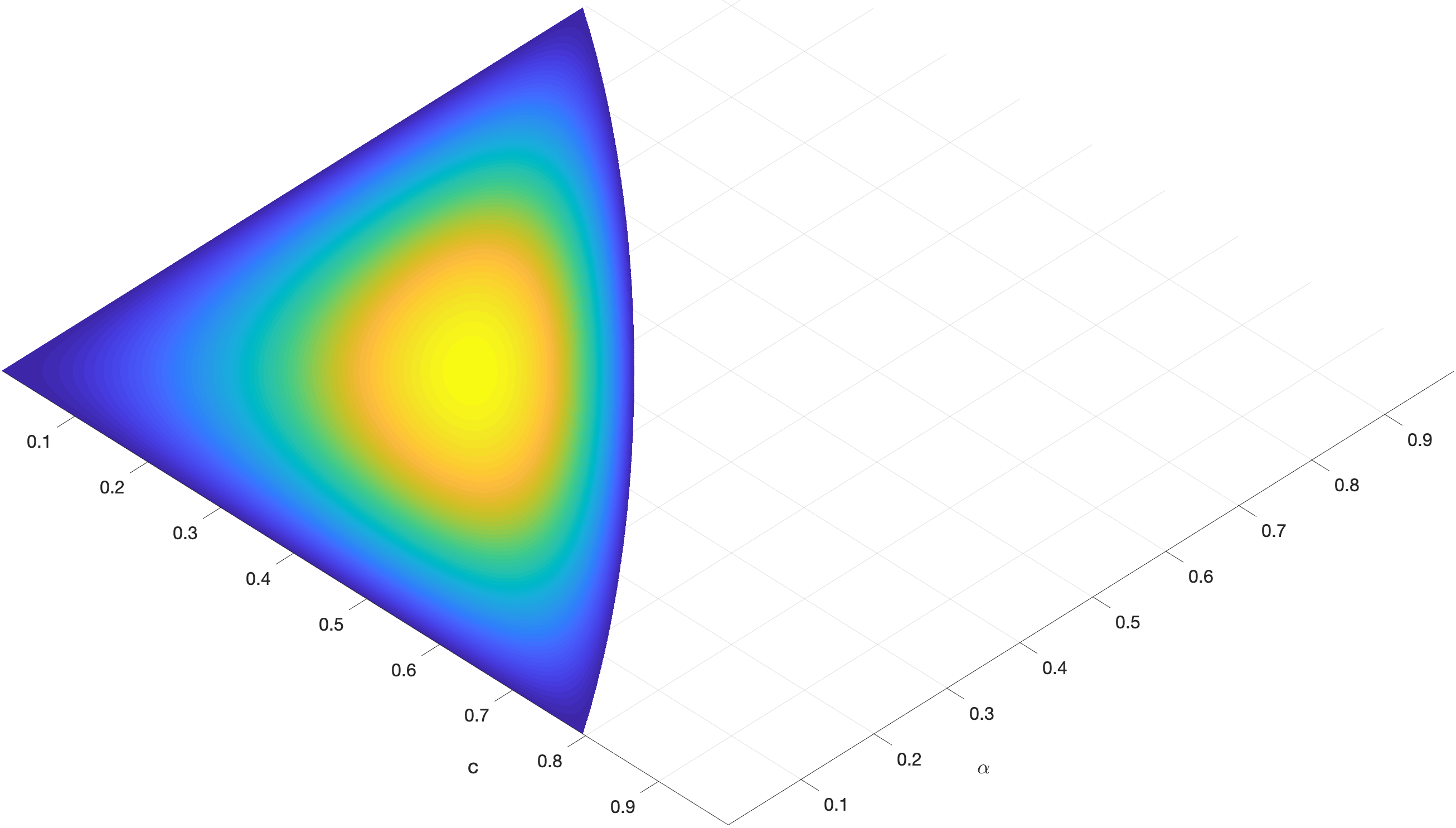}}
	\subfigure{
			\includegraphics[width=6cm,height=2.5cm]{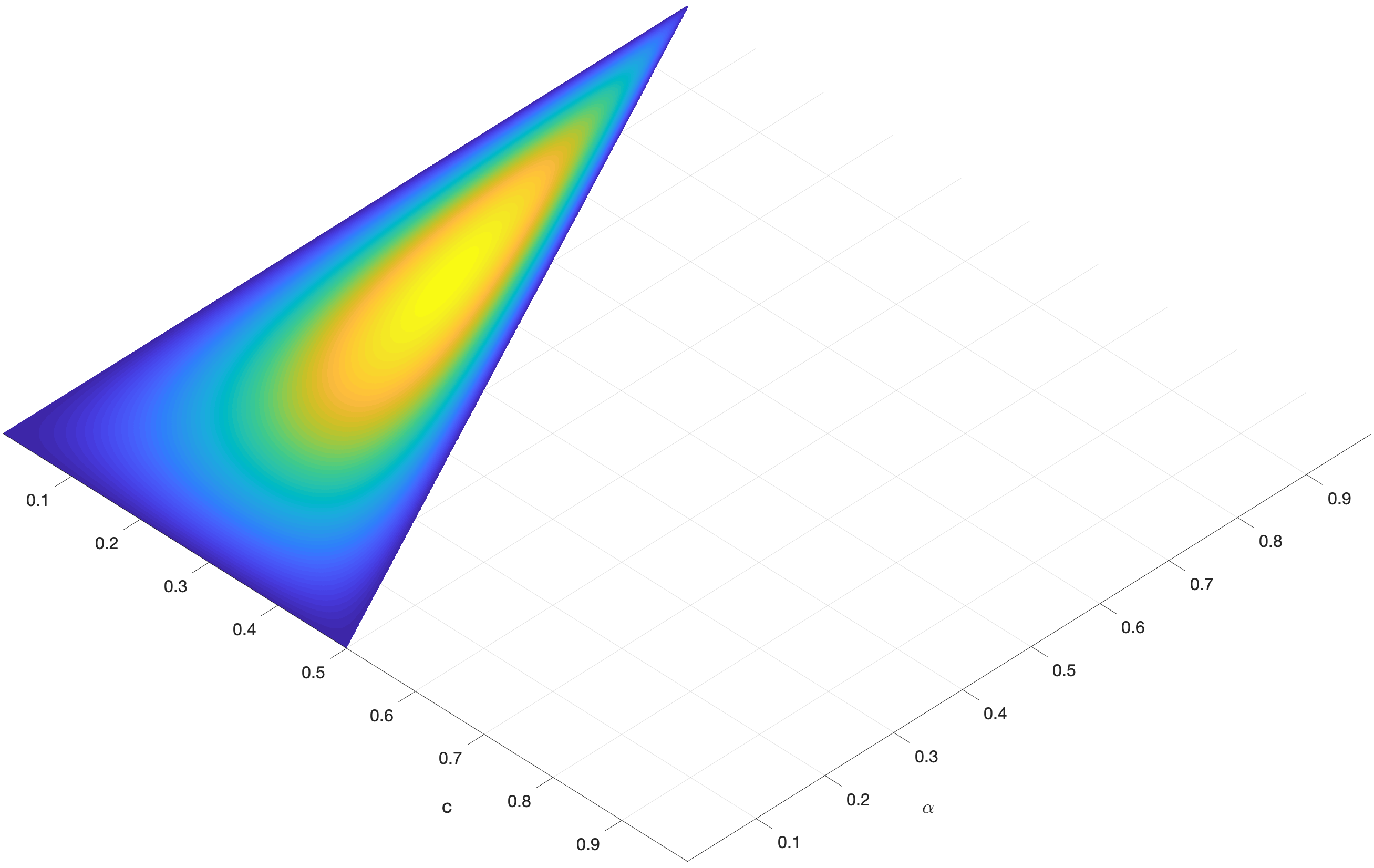}}
			\caption{3D plots for $\phi(\alpha,c)>0$ and $\psi(\alpha,c)>0$. The left two figures are a wireframe mesh and a contour plot for $\phi(\alpha,c)>0$. The right two figures are  a wireframe mesh and a contour plot for $\psi(\alpha,c)>0$.}
			\label{fig1} 
\end{figure}

\subsection{KOO methods based on the AIC, BIC, and $C_p$}\label{KOO}
The AIC, BIC, and $C_p$ become computationally complex as $k$ becomes large because we must compute a minimum of $2^k-1$ statistics. An alternate procedure, which was introduced by \cite{ZhaoK86D}  and \cite{NishiiB88S} and implemented by \cite{FujikoshiS18C}, is available to avoid this problem.
  In the following, we examine the performance of this procedure under an LLL framework. Denote 
\begin{align*}
  \tilde{A}_{j}&:=\frac1{n}(A_{{\bm \omega} \backslash j}-A_{{\bm \omega}})=\log(|\widehat \bgS_{{\bm \omega} \backslash j}| )-\log(|\widehat \bgS_{{\bm \omega}}|) -2c_n,\\
    \tilde{B}_{j}&:=\frac1{n}(B_{{\bm \omega} \backslash j}-B_{{\bm \omega}})=\log(|\widehat \bgS_{{\bm \omega} \backslash j}|) -\log(|\widehat \bgS_{{\bm \omega}}|) -\log(n)c_n,\\
      \tilde{C}_{j}&:=\frac1{n}(C_{{\bm \omega} \backslash j}-C_{{\bm \omega}})=(1-\alpha_k)\rtr(\widehat\bgS_{{\bm \omega}}^{-1}\widehat\bgS_{{\bm \omega} \backslash j})-(n-k+2)c_n.
\end{align*}
Choose the model
\begin{gather*}
  \tilde \bbj_A=\{j\in{\bm \omega}|\tilde{A}_{j}>0\},~~~
    \tilde \bbj_B=\{j\in{\bm \omega}|\tilde{B}_{j}>0\}\\
      \tilde \bbj_C=\{j\in{\bm \omega}|\tilde{C}_{j}>0\}.
\end{gather*}
These methods are based on the comparison of two models, models $M_{\bm\omega \backslash j}$ and $M_{
\bm\omega}$; therefore, selection methods
 $\tilde{\bbj}_A$,  $\tilde{\bbj}_B$ and $ \tilde{\bbj}_C$
are referred to as kick-one-out (KOO) methods based on the AIC, BIC and $C_p$, respectively.

Note that the $-2 \log$ likelihood  ratio  statistic for testing
$\bm\theta_j=\bf 0$ under normality can be expressed as
$$
n\left\{ \log( |\hat{\Sigma}_{\bm\omega}|)-\log (|\hat{\Sigma}_{\bm\omega\slash j}|)\right\}.
$$
Similarly, $(n-k)\tr (\hat{\Sigma}_{
\bm\omega}^{-1}\hat{\Sigma}_{\bm\omega\backslash j})$
is the Lawley-Hotelling trace statistic for testing $\bm\theta_j=\bf 0$. Here, $\tilde{A}_j$ $(\tilde{B}_j, ~\tilde{C}_j)$ is regarded as a measure that expresses the degree of contribution of $\bf x_j$ based on $A_j$ $ (B_j,~ C_p)$. As such, the KOO methods may also be referred to as test-based methods, as in Fujikoshi and Sakurai (2018).
Therefore, we have the following theorem for the KOO methods.
\begin{theorem}\label{newthm} 
	Suppose assumptions {\rm (A1) through (A5)} hold.
	\begin{itemize}
		\item[{\rm (1)}] If  for any $j\in \bbj_{*}$, $\log(\tau_{{\bm \omega}\backslash j})>\log(1-\alpha-c)+2c$, then
		\begin{itemize}
		\item if $\log(\frac{1-\alpha}{1-\alpha-c})<2c$, the KOO method based on the AIC is strongly consistent;
		\item if $\log(\frac{1-\alpha}{1-\alpha-c})>2c$, the KOO method based on the AIC is almost surely overspecified.  
			\end{itemize}
		 If  for some $j\in \bbj_{*}$, $\log(\tau_{{\bm \omega}\backslash j})<\log(1-\alpha-c)+2c$, then the KOO method based on the AIC is almost surely underspecified.  
		  \item[{\rm (2)}] The KOO method based on the BIC is almost surely underspecified. 
		  
		   \item[{\rm (3)}] If  for any $j\in \bbj_{*}$, $\kappa_{{\bm \omega}\backslash j}>\frac{c(1-\alpha-2c)}{1-\alpha}$, then 
  \begin{itemize}
  \item if $(1-\alpha)<2(1-\alpha-c)$, the KOO method based on the $C_p$  is strongly consistent.
  \item if $(1-\alpha)>2(1-\alpha-c)$, the KOO method based on $C_p$ is almost surely overspecified. 
  \end{itemize}
   If  for some $j\in \bbj_{*}$,  $\kappa_{{\bm \omega}\backslash j}<\frac{c(1-\alpha-2c)}{1-\alpha}$, then the KOO method based on the $C_p$ is almost surely underspecified.  
  \end{itemize}
  Suppose that assumptions {\rm (A1) through (A4)} and {\rm (A5')} hold.  
	\begin{itemize}
	\item[{\rm (4)}] If $\log(\frac{1-\alpha}{1-\alpha-c})<2c$, the KOO method based on the AIC is strongly consistent. Otherwise, if $\log(\frac{1-\alpha}{1-\alpha-c})>2c$, the KOO method based on the AIC is almost surely overspecified. 
  \item[{\rm (5)}] If  for any $j\in \bbj_{*}$, $\lim_{p,n}[\log(\tau_{n{\bm \omega}\backslash j})-\log(n)c]>\log(1-\alpha-c),$ then 
		\begin{itemize}
		\item if $\log(\frac{1-\alpha}{1-\alpha-c})<2c$, the KOO method based on the BIC is strongly consistent;
		\item if $\log(\frac{1-\alpha}{1-\alpha-c})>2c$, the KOO method based on the BIC is almost surely overspecified.  
			\end{itemize}
			If  for some $j\in \bbj_{*}$, $\lim_{p,n}[\log(\tau_{n{\bm \omega}\backslash j})-\log(n)c]<\log(1-\alpha-c),$ then the KOO method based on the BIC is almost surely underspecified.  
  \item[{\rm (6)}]  If $(1-\alpha)<2(1-\alpha-c)$, then the KOO method  based on the $C_p$ is strongly consistent. Otherwise, if $(1-\alpha)>2(1-\alpha-c)$, the KOO method based on the AIC is almost surely overspecified. 
\end{itemize}
\end{theorem}
The proof of this theorem is presented in Section \ref{proofths}.
\begin{remark}
When the dimension $p$ and model size $k$ are fixed but the sample size $n\to\infty$, the asymptotic performance of the KOO methods and the classical AIC and BIC procedures are the same, as described by \cite{NishiiB88S}. However, according to Theorems \ref{th1} and  \ref{newthm}, 
when $p$ and $k$ are large, the conditions for the KOO methods based on the AIC, BIC, and $C_p$ are  stronger than those based on the classical AIC, BIC, and $C_p$. The reason is that   the KOO methods are compared with the full model, whereas the classical AIC, BIC, and $C_p$ are compared with the true model. When the full model size is large and the true model size is small, the methods have different properties.

\end{remark}
\subsection{General KOO methods}\label{GKOO}
In the classical AIC, BIC, and $C_p$, including the KOO methods based on the AIC, BIC, and $C_p$, the penalty terms are important and modified by many researchers. For the classical information criteria under a large-sample asymptotic framework, \cite{NishiiB88S} proved that the strong consistency must be on the order of the penalty  larger than $O(\log\log n)$ and smaller than $O(n)$, which coincides with the fact that the AIC is not consistent and the BIC is strongly consistent. 
However, on the basis of the above results, under an LLL framework, a large penalty may cause incorrect selection  (actually, a constant penalty is sufficient for strong consistency), and the ranges of $\alpha$ and $c$ may be crucial for the consistency.  Thus, we
 consider a new criterion that is independent of the penalty and
 reduces the conditions for $\alpha$ and $c$. Therefore, in this subsection, we propose two general KOO methods based on the likelihood  ratio  statistic and the Lawley-Hotelling trace statistic.

From  the  proof of Theorem \ref{newthm}, we find  
that the differences in the limits of $\tilde{A}_{j}$ for $j$ that exist and do not exist in the true model $\bbj_*$  are the two terms $\log(\tau_{{\bm \omega}\backslash j})$ and $\log(1-\alpha)$ (see \eqref{tildeA1} and \eqref{tildeA2}). That means if  $\min_{j\in\bbj_*}\{\tau_{{\bm \omega}\backslash j}\}>1-\alpha$, for large enough $\{k,p,n\}$, the $k$ values of $\tilde{A}_{j}$ should be separated by a critical point and all the $\{\tilde{A}_{j}|j\in\bbj_*\}$ are the outliers because $k_*$ is fixed. The $k$ values of $\tilde{C}_{j}$ are analogous.
We provide a numerical example in  Figure \ref{koo_fig_1} to illustrate the performances of  $\{\tilde{A}_{j}\}$ and $\{\tilde{C}_{j}\}$ more clearly. In this case, the use of $Z_1$  and $Z_2$ as the critical points is not reasonable apparently. 
\begin{figure}[htbp!]
\subfigure[Histogram of $\breve{A}_{j}$]{
		\includegraphics[width=6cm,height=3.5cm]{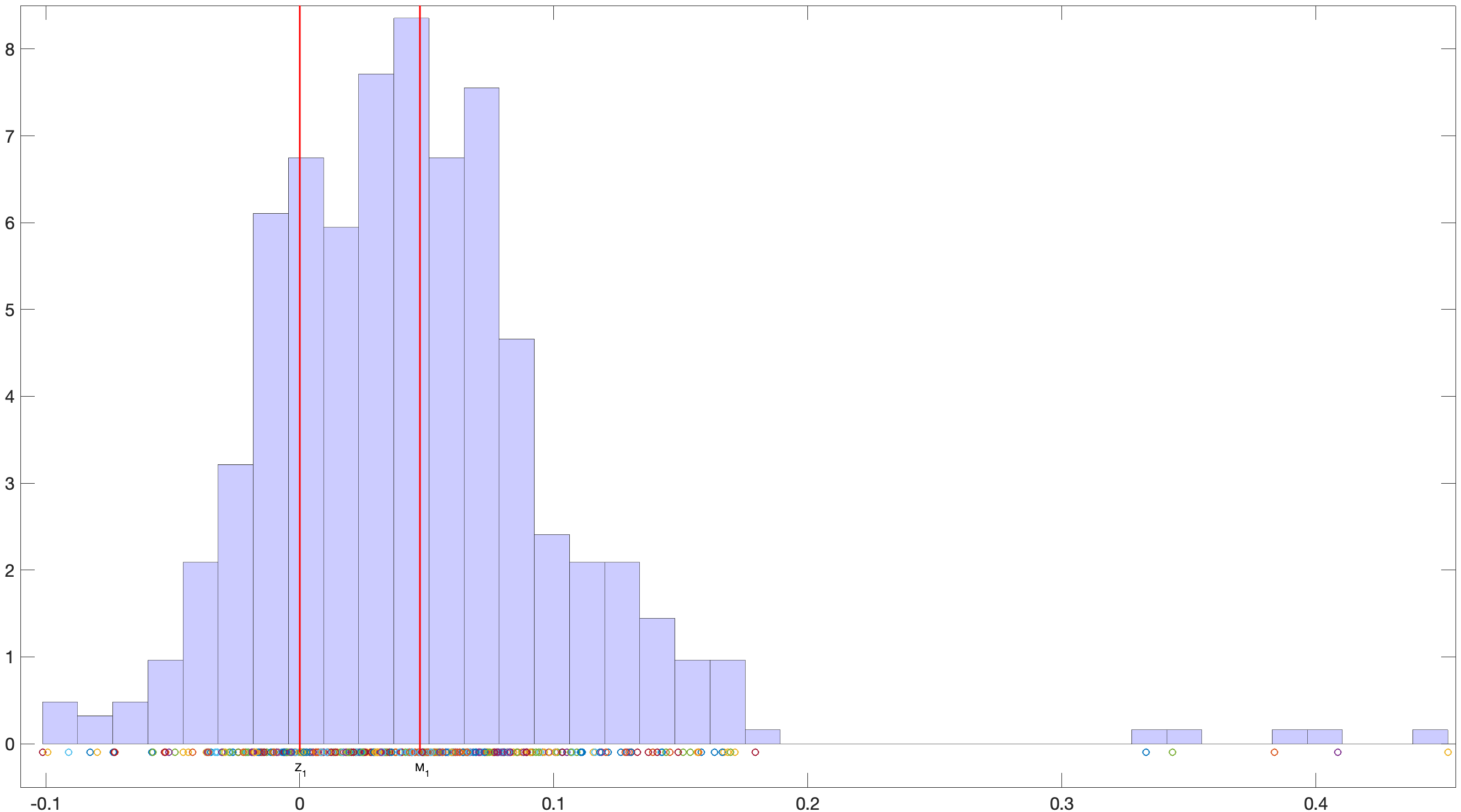}}
		\subfigure[Histogram of $\breve{C}_{j}$]{
		\includegraphics[width=6cm,height=3.5cm]{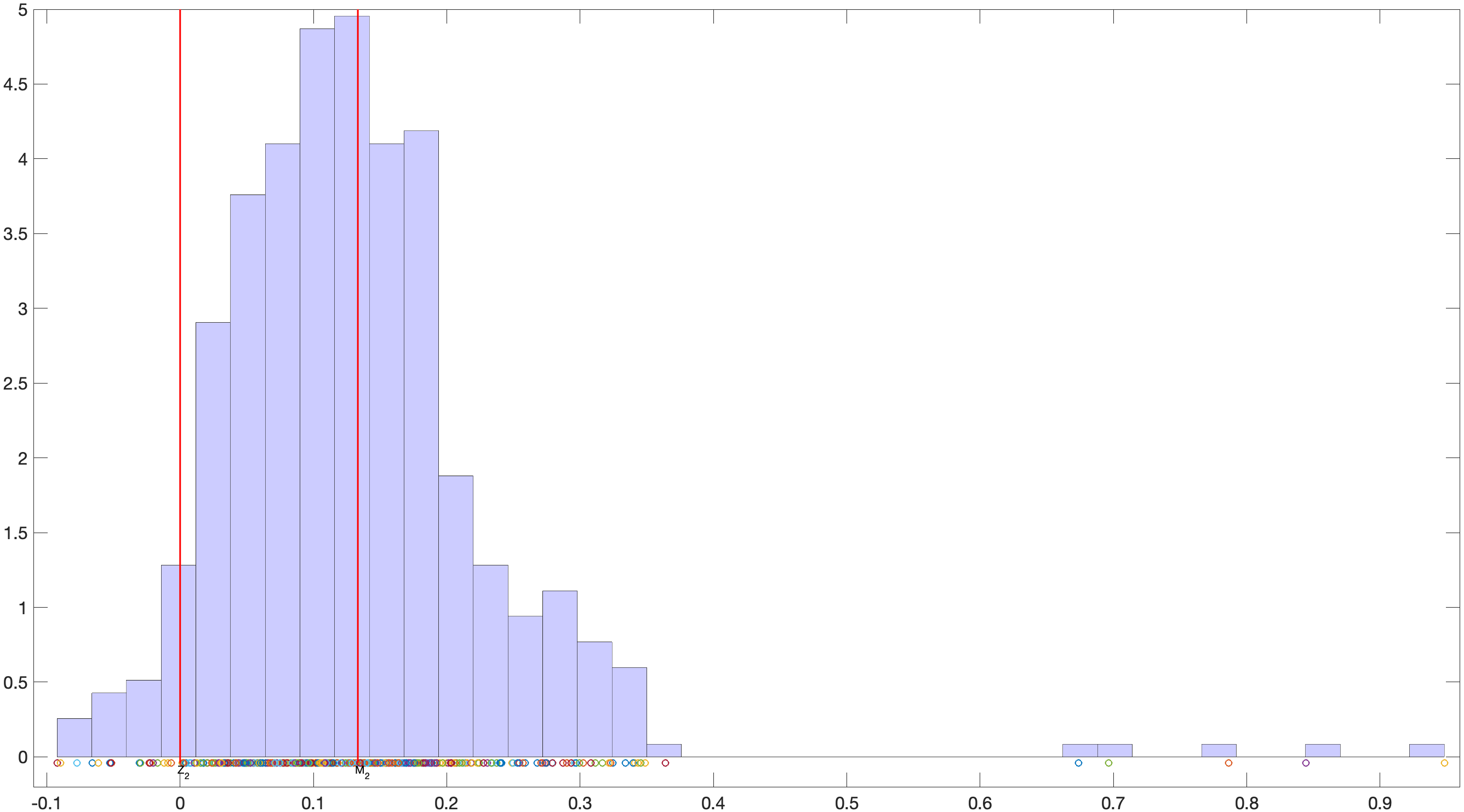}}
		\caption{We chose a Gaussian sample with $p=600$, $n=1500$, $k=450$ and $k_*=5$.  Hence, $c=0.4$ and $\alpha=0.3$.  The  histograms represent  the distributions of the $k$ values of $\breve{A}_{j}$ and $\breve{C}_{j}$ respectively. Here   $Z_1=0$ and $Z_2=0$ represent the critical points of the KOO method  based on AIC and $C_p$ respectively.    $M_1=\log((1-\alpha)/(1-\alpha-c))$ and $M_2=p+c/(1-\alpha-c)$ are the limits of $\breve{A}_{j}$ and $\breve{C}_{j}$ respectively when  $j$ does not lie in the true model.}
		\label{koo_fig_1}  
\end{figure}
Thus, on the basis of these properties, we propose the following methods which are called the general KOO methods in the sequel.  Denote  
\begin{align*}
\breve{A}_{j}:=\log(|\widehat \bgS_{{\bm \omega} \backslash j}|) -\log(|\widehat \bgS_{{\bm \omega}}|) ~~\mbox{and}~~\breve{C}_{j}:=\tr(\widehat \bgS_{{\bm \omega} \backslash j}\widehat \bgS_{{\bm \omega}}^{-1}),
\end{align*}
and  a fixed value  $\vartheta\in(0,\min_{j\in\bbj_*}\{\kappa_{{\bm \omega}\backslash j}\})$, 
choose the model
\begin{gather*}
  \breve \bbj_A=\{j\in{\bm \omega}|\breve{A}_{j}>\log(\frac{1-\alpha+\vartheta}{1-\alpha-c})\},~~~
      \breve \bbj_C=\{j\in{\bm \omega}|\breve{C}_{j}>\frac{\vartheta+c}{1-\alpha-c}+p\}.
\end{gather*}
Then, we have the following theorem.
\begin{theorem}\label{gen_kooth}
	Suppose assumptions {\rm (A1) through (A4)} hold and that for any $j\in \bbj_{*}$, $\kappa_{{\bm \omega} \backslash j}>0$. 	
	Then, for any fixed value  $\vartheta\in(0,\min_{j\in\bbj_*}\{\kappa_{{\bm \omega}\backslash j}\})$, the general KOO methods are strongly consistent, i.e.,
	\begin{align*}
    \lim_{n,p\to\infty}\breve{\bbj}_A\asto\bbj_*~~~\mbox{and} ~~~\lim_{n,p\to\infty}\breve{\bbj}_C\asto\bbj_*.
\end{align*}
	\end{theorem}
	The proof of this theorem is presented in Section \ref{proofths}.

\begin{remark}
	Note that the condition in this theorem is much weaker than that in the AIC, BIC, and $C_p$ and in the KOO methods based on the AIC, BIC, and $C_p$. In addition, although $\kappa_{{\bm \omega} \backslash j}$ is not estimable for $j\in \bbj_{*}$, as the general KOO methods are essentially to detect   the univariate outliers, thus there are many well developed methods can be used to determine the value of  $\vartheta$ for application, such as  the  Standard Deviation (SD) method,  Z-Score method, Tukey's  method,  Median Absolute Deviation method and so on.  Here we refer to \citep{BarnettL94O} for more details.
	\end{remark}

\section{Simulation studies}

In this section, we numerically examine the validity of our claims. More precisely, we attempt to examine the consistency properties of the KOO methods and the general KOO methods based on the AIC, BIC, and  $C_p$ in an LLL framework with different settings. The classical AIC, BIC, and  $C_p$ procedures are not considered herein because of their computational challenges. We  conduct a number of simulation studies to examine the effects of assumptions (A2), (A3),  (A5)  and  (A5') on the consistency of estimators $\tilde \bbj_A$, $\tilde \bbj_B$, $\tilde \bbj_C$, $\breve\bbj_A$, and $\breve\bbj_C$. Moreover, we are interested in gaining insight into the rate of convergence.

We consider the following  two settings:
\begin{itemize}
	\item[]Setting I: Fix $k_*=5$,  $p/n=\{0.2, 0.4, 0.6\}$ and $k/n=\{0.1,0.2\}$ with several different values of $n$. 
	Set  $\bbX=(x_{ij})_{n\times k}$, $\Theta_{\bbj_*}=\sqrt{n}{\bf 1}_{5}{\bm\theta}_*$ and $\Theta=(\Theta_{\bbj_*},\bf{0})$, where $\{x_{ij}\}$ are i.i.d. generated from the continuous uniform distributions $U(1,5)$,  ${\bf 1}_{5}$ is a five-dimensional vector of ones and ${\bm\theta}_*=((-0.5)^0,\dots,(-0.5)^{p-1})$. 
		\item[]Setting II: This setting is the same as  Setting I, except $\Theta_{\bbj_*}=n{\bf 1}_{5}{\bm\theta}_*.$
\end{itemize}
We consider three cases for the distribution of $\bbE$: (i) a standard normal distribution; (ii) a standardized $t$ distribution with three degrees of freedom, i.e., $e_{ij} \sim t_3/\sqrt{Var(t_3)}$; and (iii) a standardized chi-square distribution with two degrees of freedom, i.e., $e_{ij} \sim \chi^2_2/\sqrt{Var(\chi^2_2)}$. Here we use the 2 SD method to choose the critical points in the general KOO methods, that are
	\begin{gather*}
  \breve \bbj_A=\{j\in{\bm \omega}|\breve{A}_{j}>\log(\frac{1-\alpha}{1-\alpha-c})+2sd_A\}
  \end{gather*}
  and
  	\begin{gather*}
      \breve \bbj_C=\{j\in{\bm \omega}|\breve{C}_{j}>\frac{c}{1-\alpha-c}+p+2sd_C\},
\end{gather*}
where $sd_A$ and $sd_C$ are the sample standard deviations of $\{\breve{A}_{j}\}$ and $\{\breve{C}_{j}\}$ respectively.

We highlight some salient features of our settings and the distributions. 
For Setting I, convergent values in the conditions for consistency are presented  in Table \ref{tab1} by simulation. From these values and Theorem \ref{newthm}, we know that $\tilde \bbj_{A}$  is strongly consistent in cases where $\{\alpha=0.1, c=0.2, 0.4,0.6\}$ and $\{\alpha=0.2, c=0.2\}$. 
$\tilde \bbj_{C}$  is strongly consistent  in cases where $\{\alpha=0.1,0.2, c=0.2\}$, and in other cases,  $\tilde \bbj_{A}$, $\tilde \bbj_{B}$ and $\tilde \bbj_{C}$ are inconsistent. 
For Setting II, $\log(\tau_{{\bm \omega}\backslash \{1\}})=\log(n)+O(1)$ and $\kappa_{{\bm \omega}\backslash \{1\}}=O(n)$, which satisfy assumption (A5'). Under this setting, whether $\tilde \bbj_{B}$ is strongly consistent depends on the values of $c$ and $\alpha$.
  \begin{table}[htbp]\small\center
	\begin{tabular}{ccccccccccccc}
		\hline
		\multirow{2}{2em}&\multicolumn{4}{c}{$c=.2$} &\multicolumn{4}{c}{$c=.4$} &\multicolumn{4}{c}{$c=.6$} \\ \cline{2-13}
		&$V_1$&$V_2$&$V_3$&$V_4$&$V_1$&$V_2$&$V_3$&$V_4$&$V_1$&$V_2$&$V_3$&$V_4$ \\ \hline
		$\alpha=.1$    &.15&.50&.87&1.49&.21&.10& .81 &1.56& .10&-.30&.92 &1.80 \\ \hline
		$\alpha=.2$     &.11&.40&.91&1.32&.11& 0 &.92&1.43& -.19&-.40&1.21 &1.72 \\ \hline	
	\end{tabular}
	  \caption{Values  of  $V_1:=2c-\log(\frac{1-\alpha}{1-\alpha-c})$,  $V_2:=2(1-\alpha-c)-(1-\alpha)$, $V_3:=\log(\tau_{{\bm \omega}\backslash \{1\}})-\log(1-\alpha-c)-2c$, and $V_4:=\kappa_{{\bm \omega}\backslash \{1\}}-\frac{c(1-\alpha-2c)}{1-\alpha}$.}\label{tab1}
\end{table} 

To illustrate the performance of these estimators, the selection percentages of belonging to  $\bbJ_-$, $\{\bbj_*\}$ and $\bbJ_+$  were computed by Monte Carlo simulations with 1,000 repetitions. We first considered the standard  normal distribution case. Since the sum of the three selection percentages is 1, for the sake of clarity of the plots, we display only the selection percentages of belonging to $\bbJ_-$ and $\{\bbj_*\}$ (see Figure \ref{fig_normal_1}). Moreover, in some cases,  when the selection percentages of belonging to $\bbJ_+$ are close to 1, the selected model sizes are indicators of the consistency of the estimators, as presented in Figure \ref{fig_normal_2}. 
On the basis of these results, we have the following conclusions: (1) Under Setting I, the performances of the general KOO methods $\breve\bbj_A$ and $\breve\bbj_C$ are much better than those of the KOO methods $\tilde \bbj_A$, $\tilde \bbj_B$, and $\tilde \bbj_C$, and the sufficient conditions for the consistency of the KOO methods are satisfied. The convergence of $\tilde \bbj_C$ is the fastest among the five estimators. (2) If $c$ is large or close to the boundary of the sufficient conditions for consistency, i.e., $V_1$ and $V_2$ are small, then the convergence rate of the selection probabilities is slow, i.e., only sufficiently large samples can guarantee their selection accuracy. However,  despite the low selection accuracy in this case, these methods usually overestimate the true model, and the selection sizes are also under control. An overspecified model is more acceptable than an underspecified model.
(3) 
The  KOO method based on the BIC performs the best among the three methods under Setting II, and when its   sufficient conditions for consistency are satisfied. The reasons is that  overestimating the true model by the BIC is difficult compared to overestimating the true model by the other criteria.

\begin{figure}[htbp!]\subfigure[Setting I]{
		\includegraphics[width=12.2cm,height=6cm]{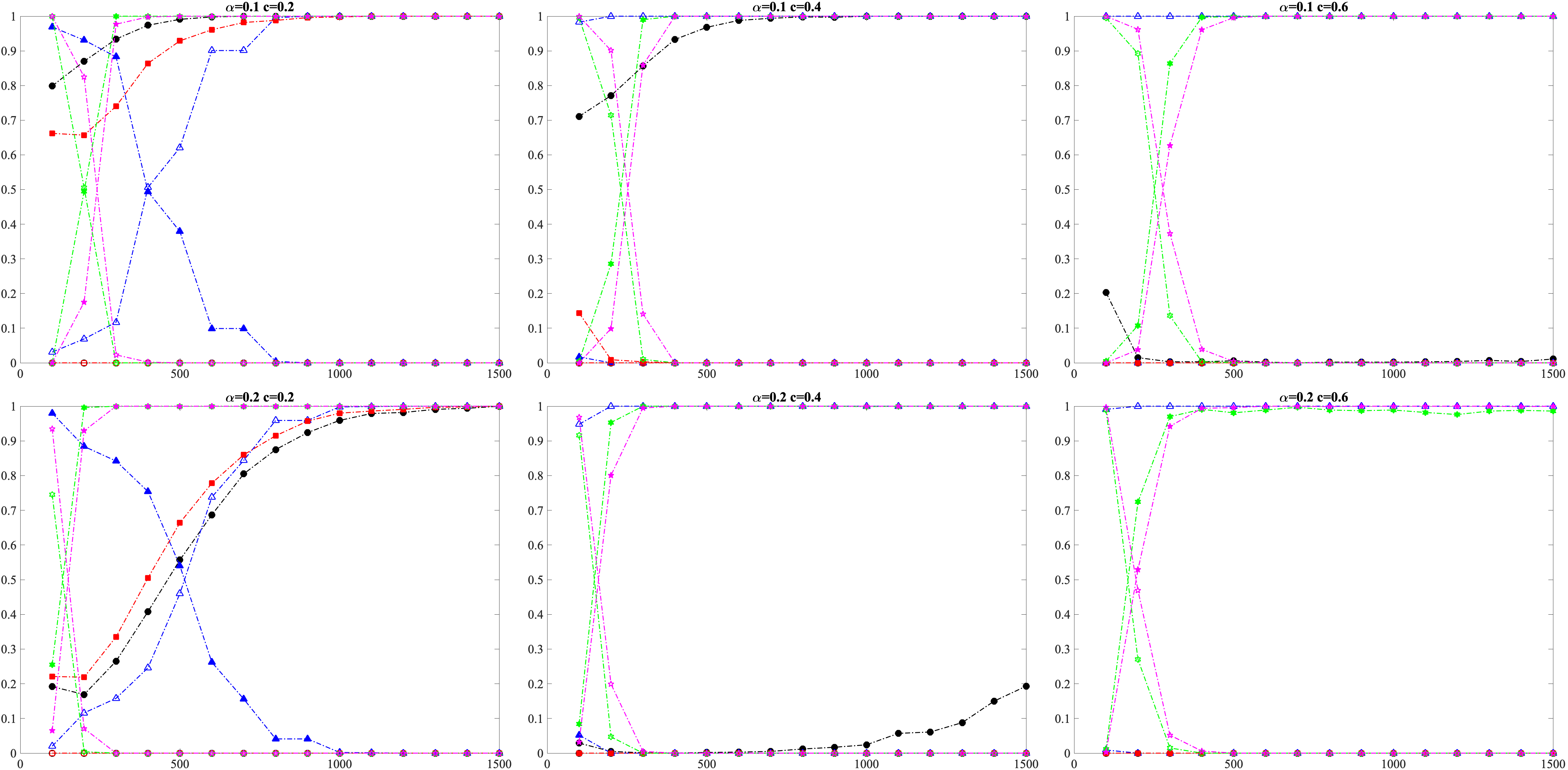}}
		\subfigure[Setting II]{
		\includegraphics[width=12.2cm,height=6cm]{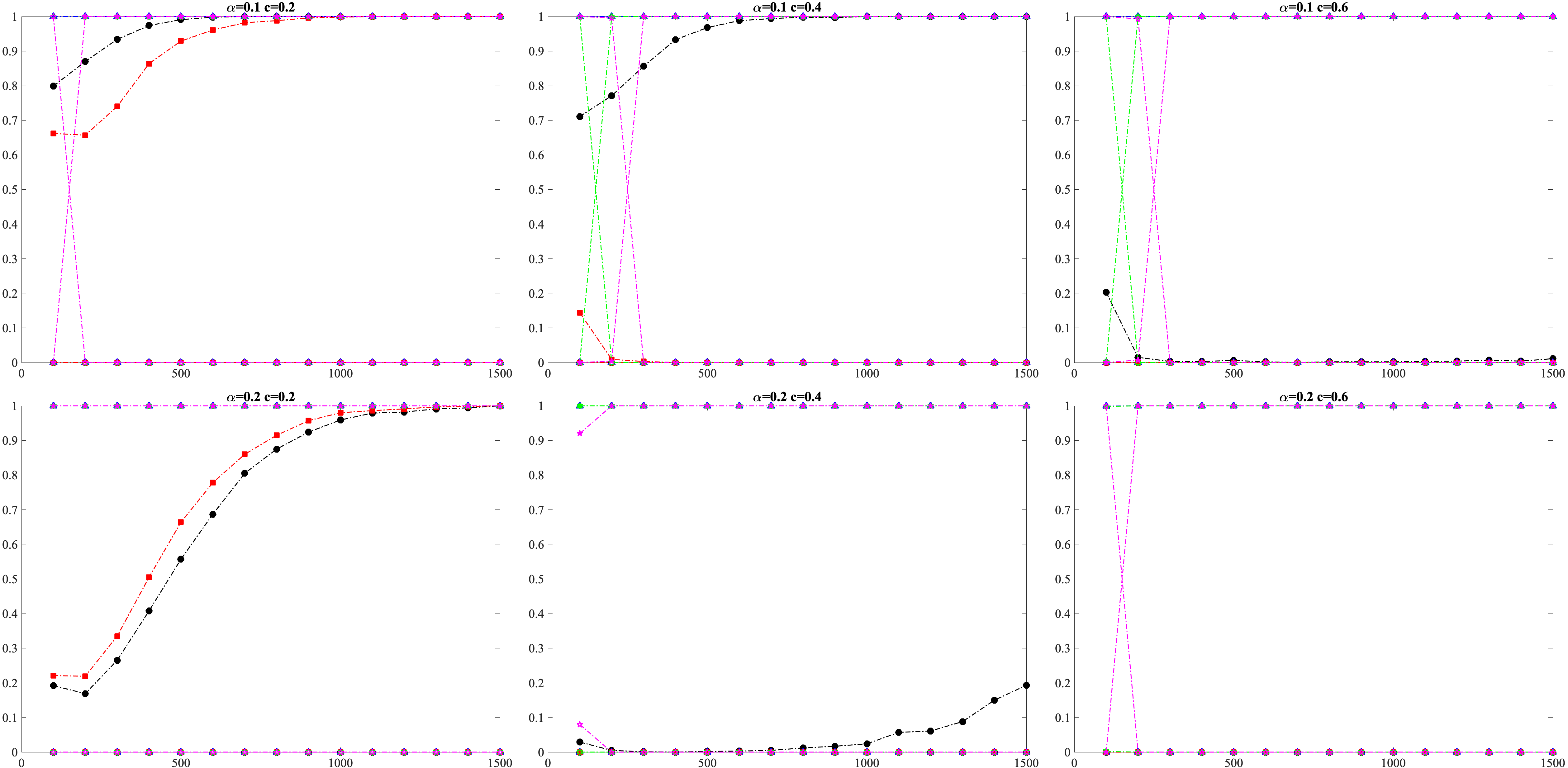}}
		\caption{Selection percentages under AIC, BIC and $C_p$ for Settings I and II with a standard normal distribution. The horizontal axes represent the sample size $n$, and the vertical axes represent the selection percentages. Black solid circles, blue solid triangles, red  solid squares, green solid hexagrams and magenta solid pentagrams
		denote the selection percentages of $\tilde\bbj_A=\bbj_*$,  $\tilde\bbj_B=\bbj_*$, $\tilde\bbj_C=\bbj_*$, $\breve\bbj_A=\bbj_*$ and $\breve\bbj_C=\bbj_*$, respectively. Correspondingly, black  circles, blue  triangles, red squares, green hexagrams, and magenta  pentagrams
		denote the selection percentages of $\tilde\bbj_A\in\bbJ_-$,  $\tilde\bbj_B\in\bbJ_-$, $\tilde\bbj_C\in\bbJ_-$, $\breve\bbj_A\in\bbJ_-$ and $\breve\bbj_C\in\bbJ_-$, respectively. 
}
		\label{fig_normal_1} 
\end{figure}

\begin{figure}[htbp!]
\subfigure[Setting I]{
		\includegraphics[width=6.1cm,height=6cm]{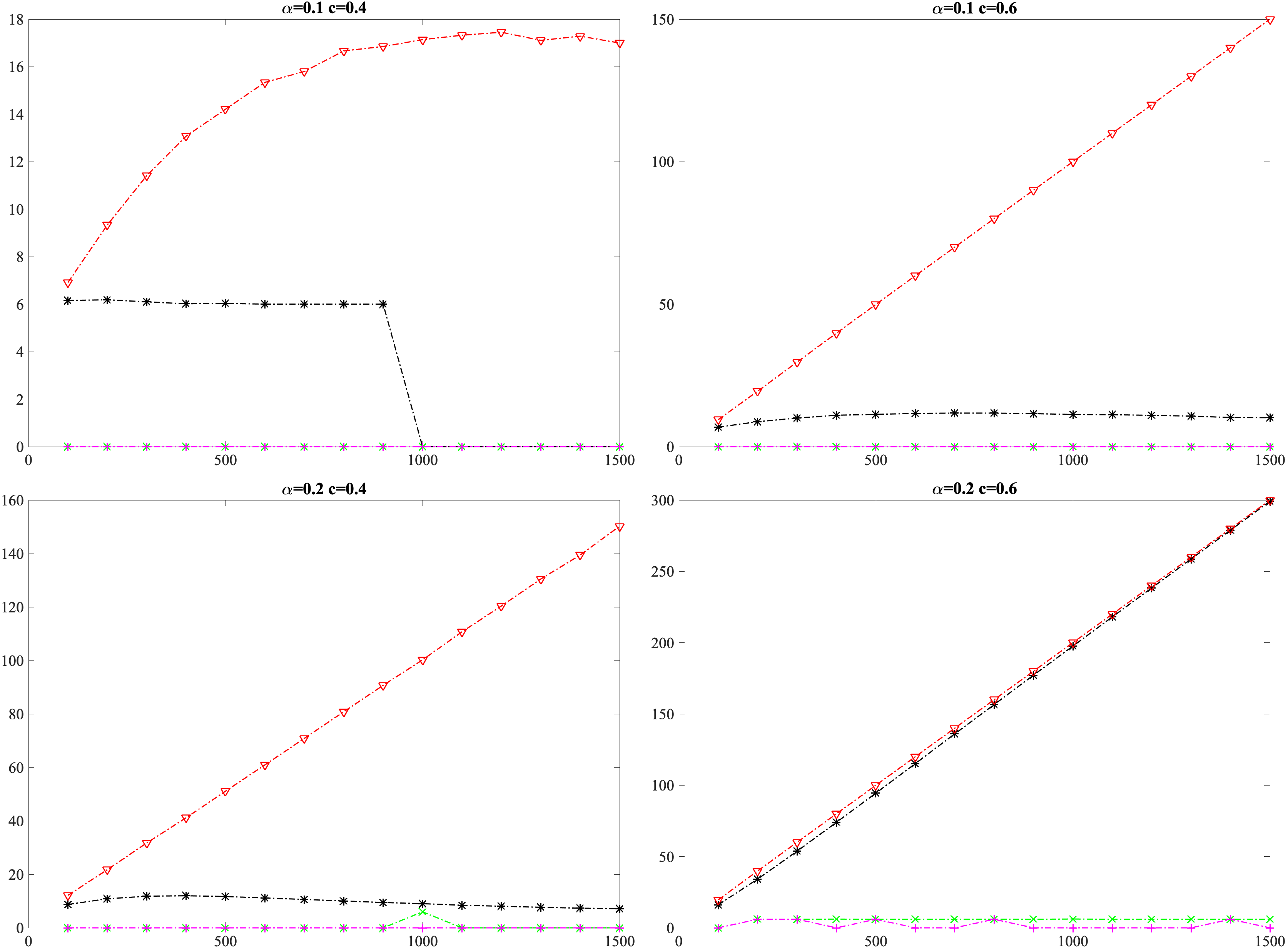}}
		\subfigure[Setting II]{
		\includegraphics[width=6.1cm,height=6cm]{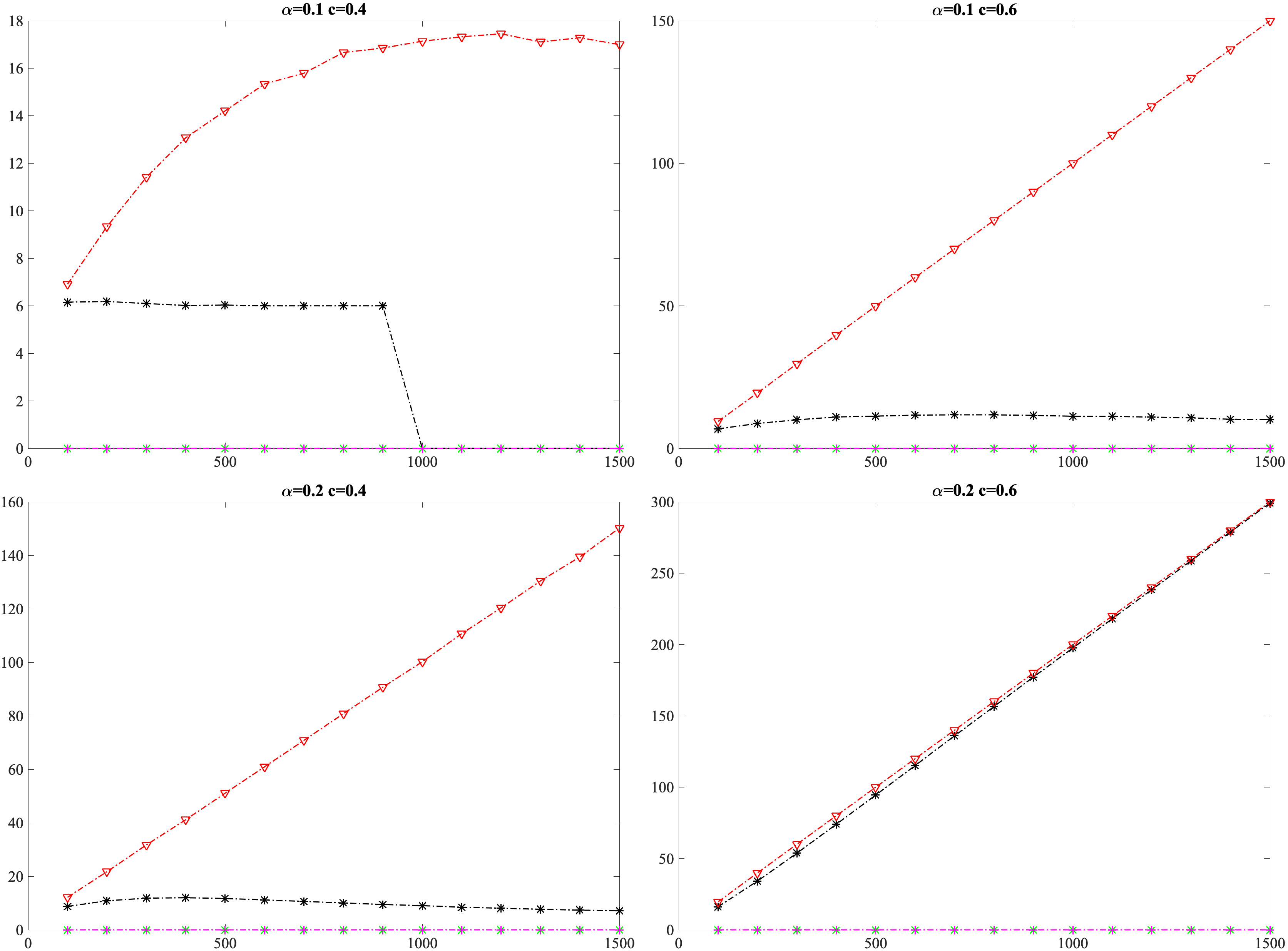}}
		\caption{Overspecified model sizes of  AIC and $C_p$ for Settings I and II with a standard normal distribution. The horizontal axes represent the sample size $n$, and the vertical axes represent the model size. 	Black  asterisks,  red   right-pointing triangles, green  crosses and magenta  plus signs
		denote the average sizes  of $\tilde\bbj_A\in\bbJ_+$,  $\tilde\bbj_C\in\bbJ_+$, $\breve\bbj_A\in\bbJ_+$ and $\breve\bbj_C\in\bbJ_+$, respectively.  In this figure, we let $0/0=0$.
		}
		\label{fig_normal_2}  
\end{figure}

The results under non-normal distributions are similar to those under normal distributions. Please see Figures \ref{fig_t3_1} through \ref{fig_chi2_2} in the Appendix, which verify our conclusion  that the consistency of these estimators is independent of the distribution.

\section{Real data analysis}
For illustration, we apply the proposed methods to two real data examples. The first one is chemometrics data, which can be taken from \cite{SkagerbergM92M} (we replaced the value 19203 with 1.9203 in the the 37th observation). This data set has been studied by \cite{BreimanF97P, SimilaT07I}. In this data set, 
There are $n=56$ observations  with $k=22$ predictor variables and $p=6$ responses. The data are taken from a simulation of a low density tubular polyethylene reactor and to study the relationship between polymer properties and the process. The predictor variables consist of 20 temperatures measured at equal distances along the  low density  polyethylene reactor section together with the wall temperature of the reactor and the solvent feed rate. The responses are the output characteristics of the polymers including two  molecular weights, two branching frequencies and  the contents of two groups. Similar to \cite{BreimanF97P}, we log-transformed the responses values, because they are highly skewed to the right. By calculation, we get $\phi(\alpha_n,c_n)=0.0266$,  $\psi(\alpha_n,c_n)=0.3444$,  $2c-\log(\frac{1-\alpha_n}{1-\alpha_n-c_n})=0.0201 $ and $2(1-\alpha_n-c_n)- (1-\alpha_n)=0.392$, which are all bigger than 0. Thus the AIC, BIC, $C_p$,    KOO methods  and  general KOO methods are all applicable. We list the result in Table \ref{tab_real1}. For this dataset,  we use  the  1 SD method to choose the critical points in the general KOO methods, because there exist large extreme values in $\{\tilde{A}_{j}\}$ and $\{\tilde{C}_{j}\}$.  We also present the scatterplots of $\{\tilde{A}_{j}\}$ and $\{\tilde{C}_{j}\}$ in Figure \ref{real_fig1}, and that seems the results obtained by the general KOO methods $\breve  \bbj_{A}$ and $\breve  \bbj_{C}$ are more reasonable than the KOO methods 		$\tilde \bbj_A$, 		$\tilde \bbj_B$ and 		$\tilde \bbj_C$.
  \begin{table}[htbp]\center
	\begin{tabular}{c|c}
		\hline
		Methods   &  Results\\ \hline
		$\hat \bbj_A$    & 1,	2,	3,	4,	5,	7,	9,	10,	11,	12,	15,	16,	17,	19,	20,	21,	22\\ 
		$\hat \bbj_B$   &2,	3,	4,	5,	7,	11,	16,	17,	20,	21,	22\\
		$\hat \bbj_C$   &1,	3,	4,	5,	7,	9,	11,	12,	15,	16,	17,	19,	20,	21,	22\\
		$\tilde \bbj_A$ & 1,	2,	3,	4,	5,	7,	8,	9,	10,	11,	12,	15,	16,	17,	19,	20,	21,	22  \\
		$\tilde \bbj_B$&3,	4,	5,	7,	11,	21,	22\\
		$\tilde \bbj_C$&2,	3,	4,	5,	7,	9,	11	,12,	21,	22\\
		$\breve  \bbj_{A}$&3,	4,	22\\
		$\breve \bbj_{C}$&3,	4,	21,	22\\
		\hline	
	\end{tabular}
	  \caption{Selection results under the AIC, BIC, $C_p$,    KOO methods  and  general KOO methods for the chemometrics dataset. Labels 1-20 are  temperatures,  21 is the wall temperature
and 22 is the solvent feed rate.}\label{tab_real1}
\end{table} 
\begin{figure}[htbp!]
\subfigure[Scatterplot of $\breve{A}_{j}$]{
		\includegraphics[width=6cm,height=4.5cm]{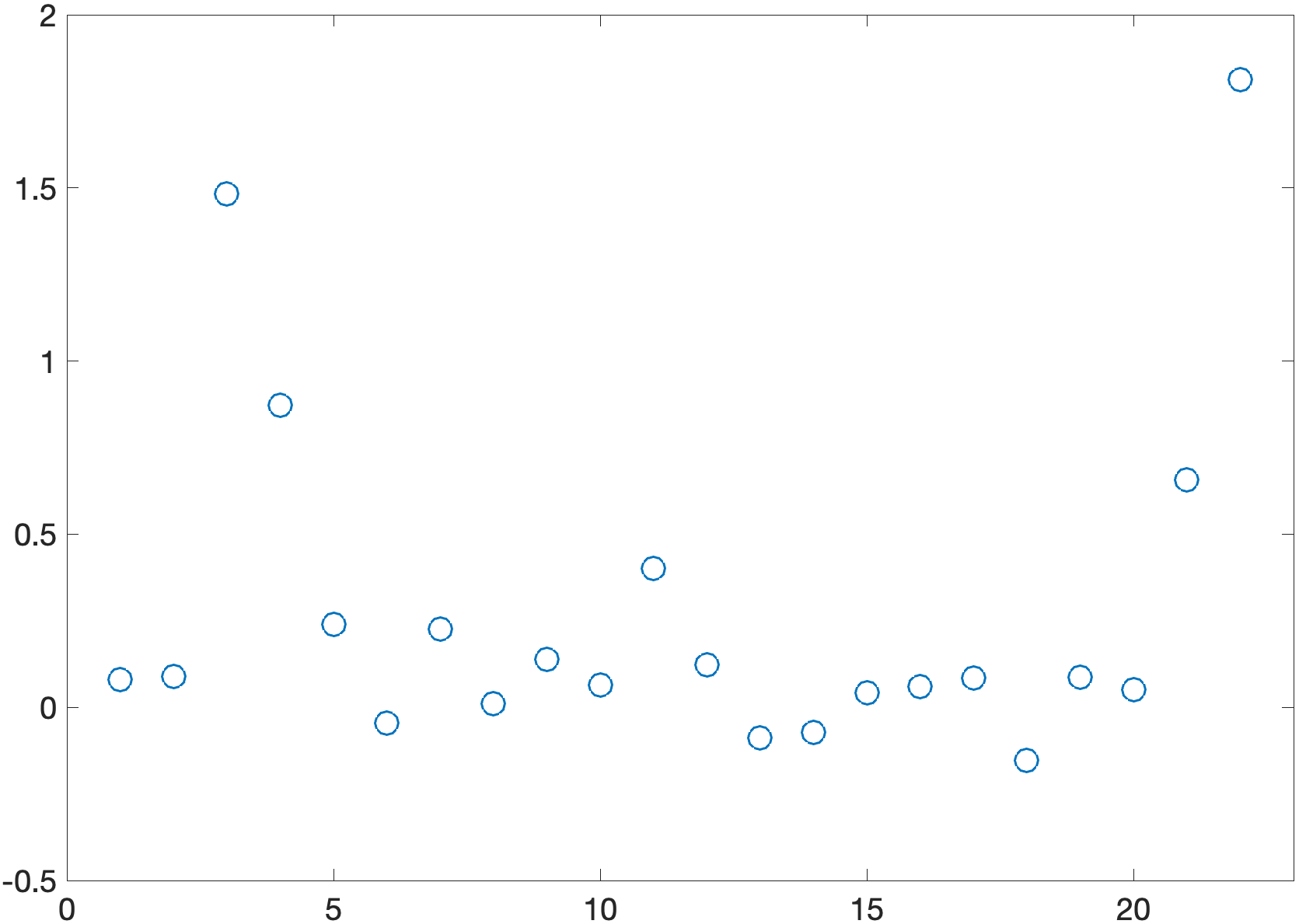}}
		\subfigure[Scatterplotof $\breve{C}_{j}$]{
		\includegraphics[width=6cm,height=4.5cm]{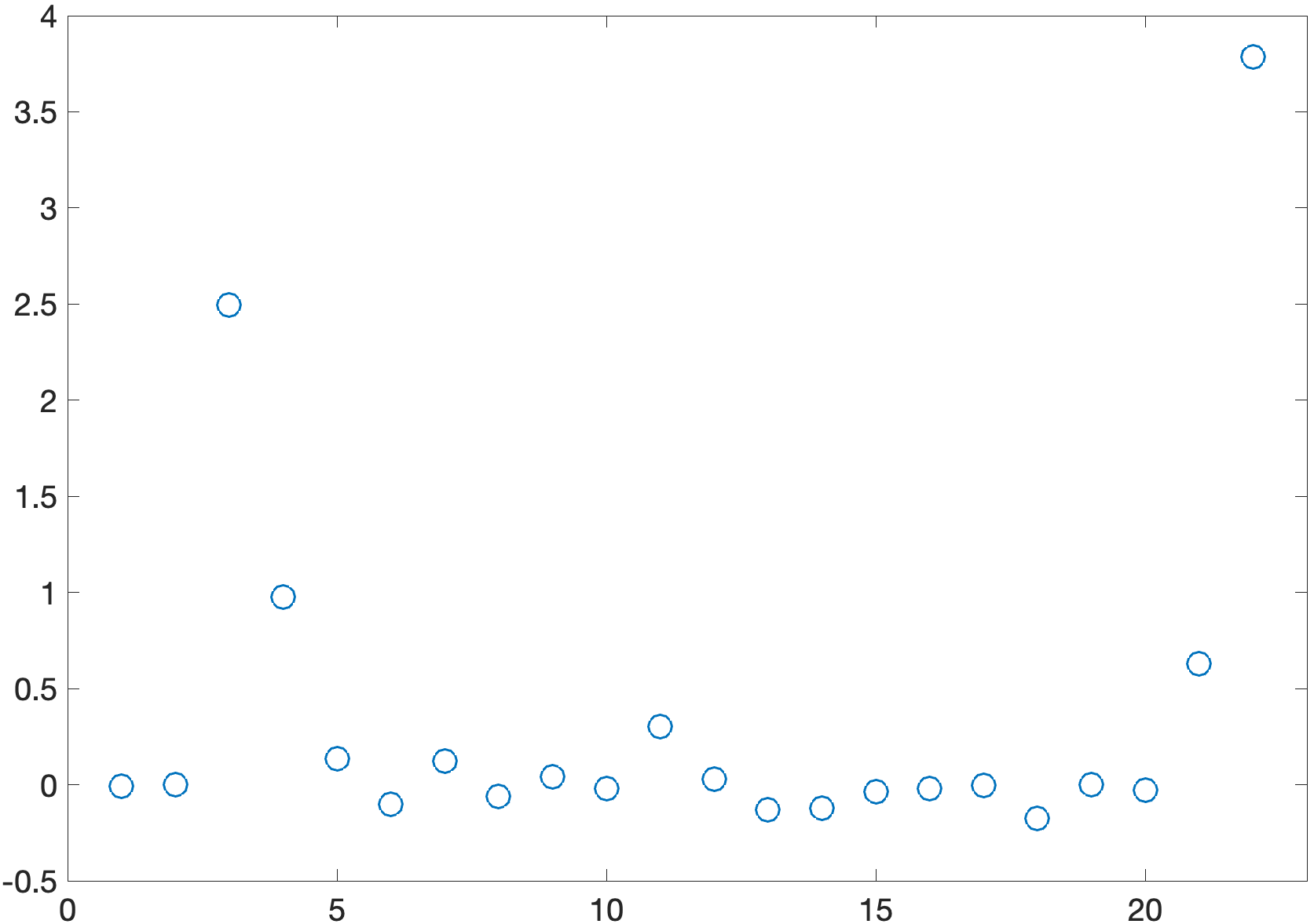}}
		\caption{Scatterplots for the chemometrics dataset.}
		\label{real_fig1}  
\end{figure}

Analogously, we investigate  a multivariate yeast cell-cycle dataset from \cite{SpellmanS98C}, which can be obtained in the R package ``spls". 
This data set contains 542 cell-cycle-related genes (i.e., $n = 542$).
Each gene contains 106  binding levels of  transcription factors (i.e., $k = 106$) and 18 time points covering two cell cycles (i.e., $p = 18$). The binding levels of the transcription factors play a role in determining which genes are expressed and help detail the process behind eukaryotic cellcycles. More explanations about the  dataset can be found in \citep{WangC07G, ChunK10S,ChenH12S,KongL17I}. As $2^{106}$ is too large to apply the classical methods based on the AIC, BIC and $C_p$, thus we only use KOO methods  and  general KOO methods for this dataset. The selection results are presented in Table \ref{tab_real2}, and the scatterplots of $\{\tilde{A}_{j}\}$ and $\{\tilde{C}_{j}\}$ can be found in Figure \ref{real_fig2}. The experimentally confirmed transcription factors that are related to the cell cycle regulation by \cite{WangC07G} are marked bold. These results also support our conclusion that the general KOO methods $\breve  \bbj_{A}$ and $\breve  \bbj_{C}$ are more reasonable than the KOO methods 		$\tilde \bbj_A$, 		$\tilde \bbj_B$ and 		$\tilde \bbj_C$.
  \begin{table}[htbp]\center
	\begin{tabular}{c|c}
		\hline
		Methods   &  Results\\ \hline
		$\tilde \bbj_A$ & {\bf SWI5},	{\bf  STE12},	{\bf  ACE2},	 {\bf NDD1},	 RME1,	HIR2,	{\bf SWI4},	 {\bf MBP1},	{\bf MET31},	HIR1,	\\
		& {\bf MCM1},	YFL044C,	RLM1,	YAP5,	IME4,	NRG1,	GAT3,	MIG1,	SIP4,	SMP1\\
			$\tilde \bbj_B$&{\bf SWI5}\\
			$\tilde \bbj_C$&{\bf SWI5},	{\bf  STE12},	{\bf  ACE2},	 {\bf NDD1},	 RME1,	HIR2,	{\bf SWI4},	 {\bf MBP1},	{\bf MET31}\\
			$\breve  \bbj_{A}$&{\bf SWI5},	{\bf  STE12},	{\bf  ACE2},	 {\bf NDD1},	 RME1,	HIR2,	{\bf SWI4},	 {\bf MBP1},	{\bf MET31},	HIR1\\
			$\breve \bbj_{C}$&
{\bf SWI5},	{\bf  STE12},	{\bf  ACE2},	 {\bf NDD1},	 RME1,	HIR2,	{\bf SWI4},	 {\bf MBP1}\\
		\hline	
	\end{tabular}
	  \caption{Selection results under the     KOO methods  and  general KOO methods for the yeast cell-cycle dataset. The experimentally confirmed transcription factors are marked bold.}\label{tab_real2}
\end{table} 
\begin{figure}[htbp!]
\subfigure[Scatterplot of $\breve{A}_{j}$]{
		\includegraphics[width=6cm,height=4.5cm]{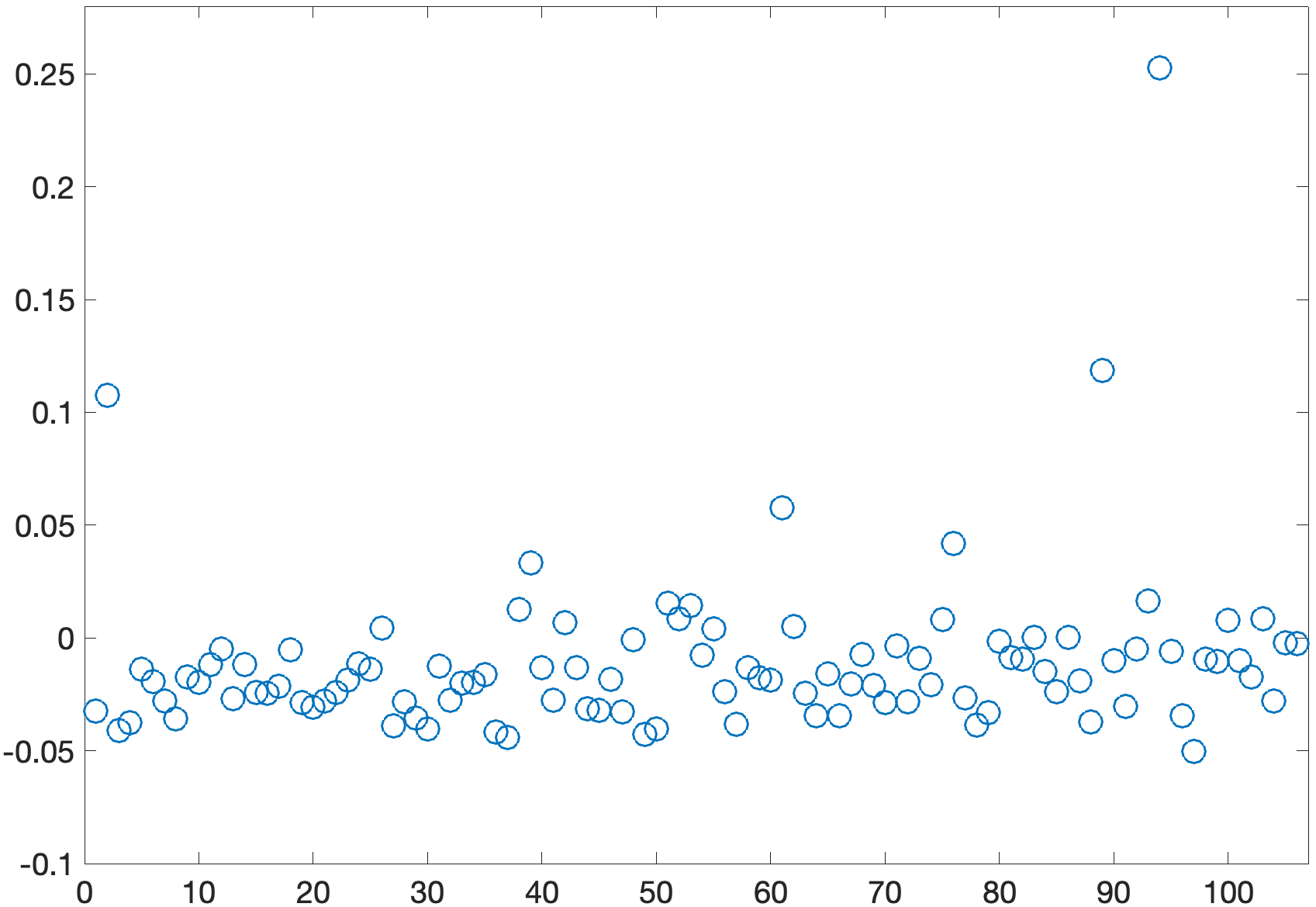}}
		\subfigure[Scatterplotof $\breve{C}_{j}$]{
		\includegraphics[width=6cm,height=4.5cm]{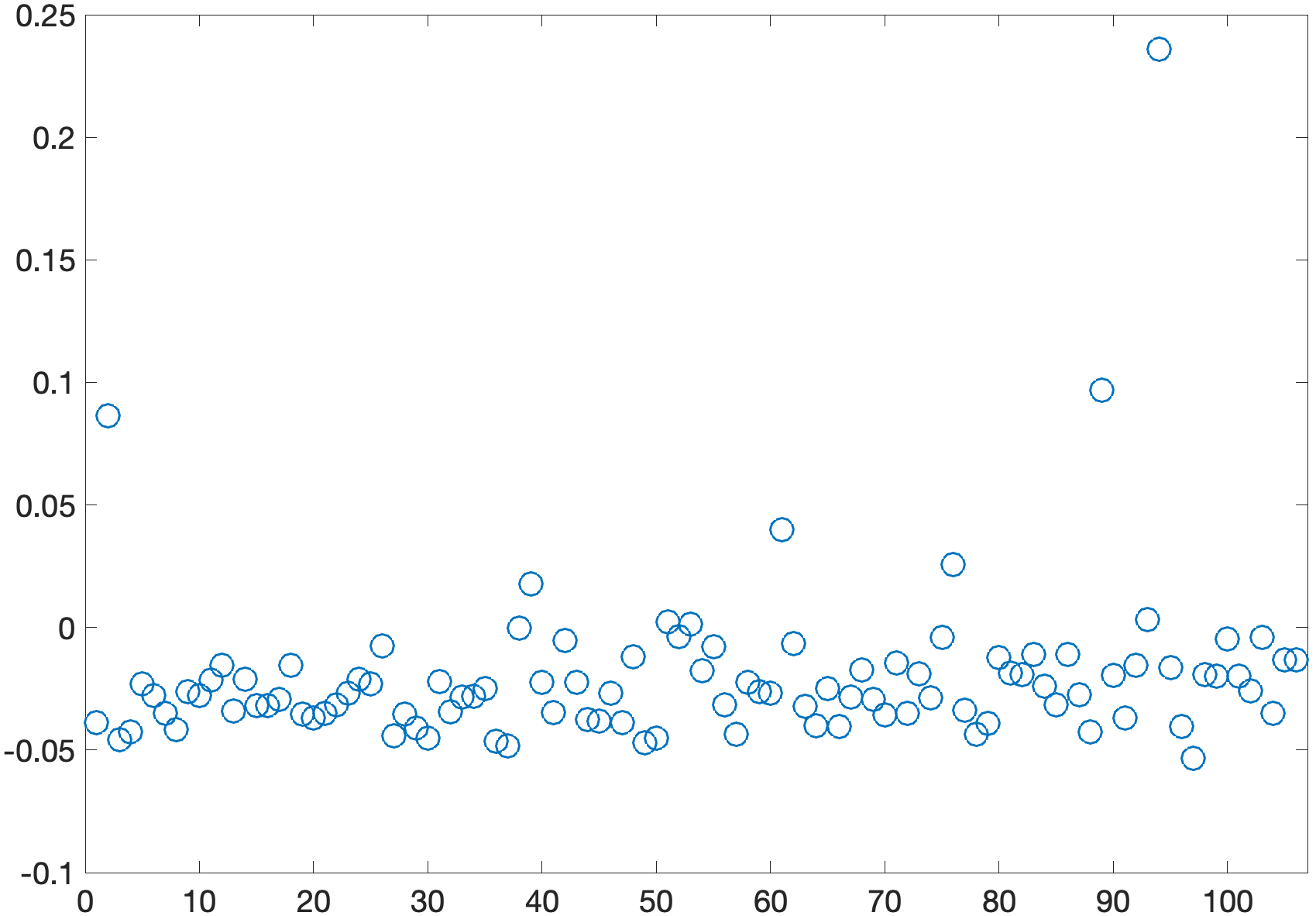}}
		\caption{Scatterplots for the yeast cell-cycle dataset.}
		\label{real_fig2}  
\end{figure}

\section{Proofs  of Theorems \ref{th1}, \ref{th2}, \ref{newthm} and \ref{gen_kooth}}\label{proofths}
In this section, we present the proofs  of Theorems \ref{th1}, \ref{th2}, \ref{newthm} and \ref{gen_kooth}. We start	with some preliminary results, which are not only the basis of the theorems but also have meaning themselves and have many potential applications in other multivariate analysis problems.  In general, a $\{c_n, \alpha_k\}$-dependent random variable $Z(c_n, \alpha_k)$ cannot serve as a limiting target of a random sequence $Z_n(c_n, \alpha_k), n \geq1$. Nevertheless, to ease the presentation, from time to time, without ambiguity we still write $Z_n(c_n, \alpha_k)\to Z(c_n, \alpha_k)$, if $Z_n(c_n, \alpha_k)-Z(c_n, \alpha_k)\to0$.

\subsection{Preliminaries}\label{Pre}
From the definition of the selection method based on the AIC, BIC, and $C_p$ in \eqref{AIC}-\eqref{hatABC}, 
the strong consistency of  the selection method based on the AIC (resp. the BIC and $C_p$) is equivalent to that for all $\bbj\in \bbJ\backslash \bbj_*$, $A_{\bbj}>A_{\bbj_*}$ (resp. $B_{\bbj}>B_{\bbj_*}$ and $C_{\bbj}>C_{\bbj_*}$) almost surely for sufficiently large $p$ and $n$. In addition, as $\bbJ=\bbJ_-\cup\bbj_*\cup\bbJ_+$, we need to consider only the following two cases, i.e., the overspecified case ($\bbj\in \bbJ_+$) and the underspecified case ($\bbj\in \bbJ_-$), and investigate for each case the conditions that guarantee that the inequality $A_{\bbj}>A_{\bbj_*}$ (resp. $B_{\bbj}>B_{\bbj_*}$ and $C_{\bbj}>C_{\bbj_*}$) holds.
We first consider the overspecified case, i.e., $\bbj\in \bbJ_+$, and assume that $k_{\bbj}-k_{\bbj_*}=m>0$; then, we have 
\begin{align*}  \frac1n(A_{\bbj}-A_{\bbj_*})=\frac1n\sum_{t=0}^{m-1}(A_{\bbj_{-t}}-A_{\bbj_{-t-1}}).
\end{align*}
Based on the definition of $A_{\bbj}$ in  \eqref{AIC}, it follows that 
\begin{align*}
	A_{\bbj_{-t}}-A_{\bbj_{-t-1}}=\log\left(\frac{|n\widehat\bgS_{\bbj_{-t}}|}{|n\widehat\bgS_{\bbj_{-t-1}}|}\right)+2c_n,
\end{align*}
which implies
\begin{align}
	\frac1n(A_{\bbj}-A_{\bbj_*})=\sum_{t=0}^{m-1}\left[\log\left(\frac{|n\widehat\bgS_{\bbj_{-t}}|}{|n\widehat\bgS_{\bbj_{-t-1}}|}\right)+2c_n\right].
\label{AIC_1}
\end{align}
Analogously, we also have 
 \begin{align}
	\frac1n(B_{\bbj}-B_{\bbj_*})=\sum_{t=0}^{m-1}\left[\log\left(\frac{|n\widehat\bgS_{\bbj_{-t}}|}{|n\widehat\bgS_{\bbj_{-t-1}}|}\right)+\log(n)c_n\right]
\label{BIC_1}
\end{align}
and
\begin{align}
 \frac1n(C_{\bbj}-C_{\bbj_*})=&(1-\alpha_k)\rtr\widehat[\bgS_{{\bm \omega}}^{-1}(\widehat\bgS_{\bbj}-\widehat\bgS_{\bbj_*})]+2mc_n\non
 =&\sum_{t=0}^{m-1}\left((1-\alpha_k)\rtr[\widehat\bgS_{{\bm \omega}}^{-1}(\widehat\bgS_{\bbj_{-t}}-\widehat\bgS_{\bbj_{-t-1}})]+2c_n\right).
 \label{Cp_1}
\end{align}
Then, we have the following lemma. 
\begin{lemma}\label{le1} Suppose that assumptions {\rm (A1) through (A4)} hold. For all overspecified models $
\bbj$ with  $k_{\bbj}-k_{\bbj_*}=m>0$, we have   
	\begin{align}\label{le1_eq1}
\frac1{n}(A_{\bbj}-A_{\bbj_*})- \sum_{t=1}^m\log\left(\frac{1-\alpha_{m-t}-c_n}{1-\alpha_{m-t}}\right)-2mc_n=o_{a.s.}(m),
	\end{align}
		\begin{align}\label{le1_eq3}
\frac1{n}(B_{\bbj}-B_{\bbj_*})-\sum_{t=1}^m\log\left(\frac{1-\alpha_{m-t}-c_n}{1-\alpha_{m-t}}\right)-
mc_n\log(n)=o_{a.s.}(m)
	\end{align}
	and 
	\begin{align}\label{le1_eq2}
\frac1{n}(C_{\bbj}-C_{\bbj_*})- \frac{mc_n(\alpha_k-1)}{1-\alpha_k-c_n}-2c_nm
=o_{a.s.}(m).
	\end{align}
	Moreover, if $\lim_{n\to\infty} 
	 \alpha_m>0$, then we have 
	\begin{align}\label{le1_eq1}
\frac1{n^2}(A_{\bbj}-A_{\bbj_*})= \phi(\alpha_m,c_n)+o_{a.s.}(1),
	\end{align}
		\begin{align}\label{le1_eq3}
\frac1{n^2}(B_{\bbj}-B_{\bbj_*})-(
\log(n)-2)c_n\alpha_m= \phi(\alpha_m,c_n)+o_{a.s.}(1),
	\end{align}
	and 
	\begin{align}\label{le1_eq2}
\frac1{n^2}(C_{\bbj}-C_{\bbj_*})=\alpha_m\psi(\alpha_k,c_n)+o_{a.s.}(1).
	\end{align}
\end{lemma}

The proof of this lemma needs  RMT, thus we present it in Section \ref{proofsle} for understandability. 

\begin{remark}Note that, taking the AIC for example, this lemma indicates that if  
$\lim \sum_{t=1}^m\log\left(\frac{1-\alpha_{m-t}-c_n}{1-\alpha_{m-t}}\right)+2mc_n>0$, for all  $\bbj\in\bbJ_+$ satisfying $k_{\bbj}-k_{\bbj_*}=m$ and  sufficiently large  $p$ and $n$, then $\bbj$ almost surely cannot be selected by the AIC. On the other hand, if  
$ \lim\sum_{t=1}^m\log\left(\frac{1-\alpha_{m-t}-c_n}{1-\alpha_{m-t}}\right)+2mc_n<0$, then for sufficiently large  $p$ and $n$, $\bbj_*$ almost surely cannot be selected by the AIC, which means that, in this case, the AIC is almost surely inconsistent. The BIC and $C_p$  are analogous.
\end{remark}

Next, we consider the underspecified case,  i.e., $\bbj\in \bbJ_-$. Analogously, 
we have 
$$
	\frac1n(A_{\bbj}-A_{\bbj_*})=\sum_{t=-s}^{m-1} {\mathcal A}_{-t},$$
	$$
	\frac1n(B_{\bbj}-B_{\bbj_*})=\sum_{t=-s}^{m-1} {\mathcal B}_{-t}
	~~\mbox{and}~~
\frac1n(C_{\bbj}-C_{\bbj_*})=\sum_{t=-s}^{m-1} {\mathcal C}_{-t},
$$
where 
\bqn
{\mathcal A}_{-t}=\begin{cases}
\log({|\widehat\bgS_{\bbj_{-t}}|})-\log({|\widehat\bgS_{\bbj_{-t-1}}|})+2c_n&\mbox{if } t\geq0,\\
	\log({|\widehat\bgS_{\bbj_{-t}}|})-\log({|\widehat\bgS_{\bbj_{-t-1}}|})-2c_n&\mbox{if } t<0,
\end{cases}
\eqn
\bqn
{\mathcal B}_{-t}=\begin{cases}\log({|\widehat\bgS_{\bbj_{-t}}|})-\log({|\widehat\bgS_{\bbj_{-t-1}}|})+\log(n)c_n&\mbox{if } t\geq0,\\
	\log({|\widehat\bgS_{\bbj_{-t}}|})-\log({|\widehat\bgS_{\bbj_{-t-1}}|})-\log(n)c_n&\mbox{if } t<0,
\end{cases}
\eqn
and 
\bqn
{\mathcal C}_{-t}=\begin{cases}
(1-\alpha_k)\rtr[\widehat\bgS_{{\bm \omega}}^{-1}(\widehat\bgS_{\bbj_{-t}}-\widehat\bgS_{\bbj_{-t-1}})]+2c_n&\mbox{if } t\geq0,\\
(1-\alpha_k)\rtr[\widehat\bgS_{{\bm \omega}}^{-1}(\widehat\bgS_{\bbj_{-t}}-\widehat\bgS_{\bbj_{-t-1}})]-2c_n, &\mbox{if } t<0.
\end{cases}
\eqn
Since $\bbj_0\supset \bbj_*$, it from Lemma \ref{le1} that, 
\bqn
\sum_{t=0}^{m-1}{\mathcal A}_{-t}- \sum_{t=1}^m\log\left(\frac{1-\alpha_{m+s-t}-c_n}{1-\alpha_{m+s-t}}\right)-2mc_n=o_{a.s.}(m),
\eqn
\bqn
\sum_{t=0}^{m-1}{\mathcal B}_{-t}-\sum_{t=1}^m\log\left(\frac{1-\alpha_{m+s-t}-c_n}{1-\alpha_{m+s-t}}\right)-
mc_n\log(n)=o_{a.s.}(m)
\eqn
and 
\bqn
\sum_{t=0}^{m-1}{\mathcal C}_{-t}-m\psi(\alpha,c)=o_{a.s.}(m).
\eqn
Therefore, we need to consider only the case in which $t<0$ and some additional notation.
Let  $\bfell_t=\{i_{1},\cdots, i_{-t}\}$,  we denote
	\begin{align*}
\gD_t :=&  (\bbX'_{\bfell_t}\bbQ_{\bbj_t}\bbX_{\bfell_t})^{1/2}\Theta_{\bfell_t}\Theta'_{\bfell_t}(\bbX'_{\bfell_t}\bbQ_{\bbj_t}\bbX_{\bfell_t})^{1/2};\\
\tilde\bba_t':=&\bba_t'\bbQ_{\bbj_t}\bbX_{\bfell_t}(\bbX_{\bfell_t}'\bbQ_{\bbj_{t}}\bbX_{\bfell_t})^{-1/2};\\
  \delta_t:=&\tilde\bba_t'\left(({1-\alpha_m})\bbI+{n^{-1}\gD_t}\right)^{-1}\tilde\bba_t;\\
  \eta_t:= &\frac1n\tilde\bba_t'\gD_t\tilde\bba_t=\frac1n\bba_t'\bbQ_{\bbj_t}\bbX_{\bfell_t}\Theta_{\bfell_t}\Theta'_{\bfell_t}\bbX_{\bfell_t}'\bbQ_{\bbj_t}\bba_t.
\end{align*}
By basic calculation, we obtain that $\delta_t$ and $\eta_t$ can both be expressed as functions of the noncentrality matrix $\Phi_{\bbj}$, as follows:
 \begin{align*}
  \prod_{l=1}^{-t}\tilde \bba_{l}'((1-\alpha_m)\bbI+\frac1n\Delta_l)^{-1}\tilde \bba_l=(1-\alpha_m)^{p+t}|(1-\alpha_m)\bbI+\Phi_{\bbj}|^{-1}
\end{align*} and 
\begin{align*}
   \sum_{l=1}^{-t}\eta_l=\frac1n\tr[\Theta'_{\bbj_*}\bbX_{\bbj_*}'(\sum_{l=1}^{-t}\bba_l\bba_l')\bbX_{\bbj_*}\Theta_{\bbj_*}]=\tr(\Phi_{\bbj}).
\end{align*}
Then, we have the following lemma.
\begin{lemma}\label{le2} 
Suppose assumptions {\rm (A1) through (A5)} hold.  If  $\bbj\in \bbJ_-$  and $t<0$,  then we have 
\begin{align}\label{le2_eq1}
{\mathcal A}_{-t}+\log({\delta_t})+\log({1-\alpha_{m}-c_n})+2c_n=o_{a.s.}(1),
\end{align}
\begin{align}\label{le2_eq3}
{\mathcal B}_{-t}+\log({\delta_t})+\log({1-\alpha_{m}-c_n})+c_n \log(n)=o_{a.s.}(1),
\end{align}
and 
\begin{align}\label{le2_eq2}
{\mathcal C}_{-t}-\frac{(1-\alpha_k)(\eta_t+c_n)}{1-c_n-\alpha_k}+2c_n=o_{a.s.}(1).
\end{align}
\end{lemma}
The proof of this lemma is  presented in Section  \ref{proofsle}. By combining Lemmas \ref{le1} and \ref{le2},  the following lemma is straightforward.
\begin{lemma}\label{le3} Suppose that assumptions {\rm (A1) through (A5)} hold. For all underspecified models $\bbj$ with $k_{\bbj_-}=s>0$ and $
k_{\bbj_+}=m\geq0$,  
  we have 
\begin{align*}
\frac1n(A_{\bbj}-A_{\bbj_*})=&\sum_{t=1}^m\log\left(\frac{1-\alpha_{m+s-t}-c}{1-\alpha_{m+s-t}}\right)\\
&+\log(\tau_{\bbj})-s\log({1-\alpha_{m}-c})+2(m-s)c+o_{a.s.}(m),
\end{align*}
\begin{align*}
\frac1n(B_{\bbj}-B_{\bbj_*})=&\sum_{t=1}^m\log\left(\frac{1-\alpha_{m+s-t}-c}{1-\alpha_{m+s-t}}\right)\\
&+\log(\tau_{\bbj})-s\log({1-\alpha_{m}-c})+(m-s)c\log(n)+o_{a.s.}(m)
\end{align*}
and
\begin{align*}
\frac1n(C_{\bbj}-C_{\bbj_*})=m\psi(\alpha,c)+\frac{(1-\alpha)(\kappa_{\bbj}+sc)}{1-c-\ga}-2sc+o_{a.s.}(m).
\end{align*}
\end{lemma}
\begin{remark}
	Lemmas \ref{le1} through \ref{le3} are the fundamental results for the model selection criteria discussed herein. These lemmas have numerous potential applications in multivariate analysis problems, such as the
growth curve model \citep{EnomotoS15C, FujikoshiE13H},
multiple discriminant analysis \citep{Fujikoshi83C,FujikoshiS16H},
principal component analysis \citep{FujikoshiS16S,BaiC18C}, and canonical correlation analysis \citep{NishiiB88S,BaoH18C}.
\end{remark}
Now, with the aid of Lemma \ref{le1} and Lemma \ref{le3},  we  prove
Theorems \ref{th1}, \ref{th2}, \ref{newthm} and \ref{gen_kooth}.

\subsection{Proof of Theorem \ref{th1}}
We first prove (1). For the strong consistency of  AIC method,  it is sufficient  to prove for all $\bbj\neq \bbj_*$, $A_{\bbj}>A_{\bbj_*}$ almost surely for large enough $p$ and $n$. 
When  $\bbj\in\bbJ_+$ with $k_{\bbj}-k_{\bbj_*}=m$,
	from the definition  of $\phi(\alpha,c)$, we know that if 
	$\phi(\alpha,c)>0$, then we have  $\log(1-c)+2c>0$,  and for any $\alpha_m$ satisfying $0<\alpha_m\leq \alpha$, we have $\phi(\alpha_m,c)>0$. In addition, if $\alpha_m\to0$, by \eqref{le1_eq1}, we have 
	\begin{align*}
  \frac1{nm}(A_{\bbj}-A_{\bbj_*})\asto\log(1-c)+2c>0.
\end{align*}
Thus, according to Lemma \ref{le1}, if $\phi(\alpha,c)>0$, AIC method cannot be asymptoticly over-specified, i.e., for  any $\bbj\in\bbJ_+$ and for sufficiently large $p$ and $n$,
\begin{align*}
  A_{\bbj}>A_{\bbj_*} ~~~a.s..
\end{align*}
For the case of $\bbj\in \bbJ_-$ with $k_{\bbj_+}=m\geq0$ and        $k_{\bbj_*\cap\bbj_-^c}=s>0$, note that $s< k_*$,  $\Phi_{\bbj}>0$ and $\log(\tau_{n\bbj})>s\log(1-\alpha_m)$ uniformly. If $m\geq s$, by Lemma \ref{le3} and when $\phi(\alpha,c)>0$,   \begin{align*}
   &\frac1n(A_{\bbj}-A_{\bbj_*})\\
   =&\sum_{t=1}^m\log\left(\frac{1-\alpha_{m-t}-c}{1-\alpha_{m-t}}\right)\\
&+\log(\tau_{\bbj})-s\log({1-\alpha_{m}-c})+2(m-s)c+o_{a.s.}(m)\\
=&\sum_{t=s+1}^m\log\left(\frac{1-\alpha_{m-t}-c}{1-\alpha_{m-t}}\right)
+\log(\tau_{\bbj}/(1-\alpha_m))+2(m-s)c+o_{a.s.}(m),
\end{align*}
which is  almost surely positive for sufficiently large $p$ and $n$. Thus, we only need to consider  the case in which $m< s$. In this case, since $k_*$ is fixed, $\alpha_m=m/n\to0$ and  $\tau_{\bbj}>1$. Then, under the condition $\log(\tau_{\bbj})>(s-m)(\log(1-c)+2c)$, we have that 
\begin{align*}
  \frac1n(A_{\bbj}-A_{\bbj_*})\asto\log(\tau_{\bbj})-(s-m)(\log(1-c)+2c)>0, 
\end{align*}
 which together with the above discussion implies the variable selection method based on the AIC is strongly consistent. 

If for some  $\bbj\in \bbJ_-$ with $m-s<0$,  $\log(\tau_{\bbj}^0)< (s-m)(\log(1-c)+2c)$, then from Lemmas \ref{le1} and \ref{le3}, we know that  for sufficiently large $p$ and $n$,  
\begin{align*}
  A_{\bbj}< A_{\bbj_*}, ~~~a.s.,
\end{align*}
which means that, in this case, the AIC asymptotically selects the under-specified model $\bbj$. On the other hand, 
 the condition $\phi(\alpha,c)>0$ guarantees  that the AIC cannot asymptotically selects over-specified models. Thus, the AIC asymptotically selects an under-specified model.

If $\phi(\alpha,c)<0$,  for any $\bbj\in  \bbj_*\bigcup\bbJ_-$ with $k_{\bbj_+}=m\geq0$ and $k_{\bbj_*\cap\bbj_-^c}=s\geq0$   we have that
\begin{align*}
  &\frac1n(A_{\bbj}-A_{{\bm \omega}})=\sum_{t=1}^m\log\left(\frac{1-\alpha_{m-t}-c_n}{1-\alpha_{m-t}}\right)-\sum_{t=1}^{k-k_*}\log\left(\frac{1-\alpha_{k-k_*-t}-c_n}{1-\alpha_{k-k_*-t}}\right)\\
&+\log(\tau_{\bbj})-s\log({1-\alpha_{m}-c_n})+2(m-s-k+k_*)c_n+o_{a.s.}(k-m)\\
=& \log(\frac{\tau_{\bbj}}{(1-\alpha_{m})^s})-s\left(\log\left(\frac{1-\alpha_{m}-c_n}{1-\alpha_{m}}\right)+2c\right)\\
&-\sum_{t=m+1}^{k-k_*}\log\left(\frac{1-\alpha_{k-k_*-t}-c_n}{1-\alpha_{k-k_*-t}}\right)-2(k-k_*-m)c_n+o_{a.s.}(k-m). 
\end{align*}
Note that $\log(\frac{\tau_{\bbj}}{(1-\alpha_{m})^s})>0$, $\log(\frac{1-\alpha-c}{1-\alpha})+2c<0$ and $\sum_{t=m+1}^{k-k_*}\log(\frac{1-\alpha_{k-k_*-t}-c_n}{1-\alpha_{k-k_*-t}})+2(k-k_*-m)c_n$ is negative with order of $k-m$, which indicate 
$\frac1n(A_{\bbj}-A_{{\bm \omega}})>0 $ almost surely for sufficiently large $p$ and $n$. Thus, the AIC asymptotically selects an over-specified model. Thus, we obtain (1).

For (2), by \eqref{le1_eq3}, it is easy to find that  BIC method cannot be asymptoticly over-specified.  In addition, under assumption (A5), 
for any $\bbj\in \bbJ_-$ with $k_{\bbj_+}=m\geq0$,       $k_{\bbj_*\cap\bbj_-^c}=s>0$, and $m< s$, by Lemma \ref{le3}, we obtain that
\begin{align}\label{Btau}
  \frac1n(B_{\bbj}-B_{\bbj_*})\asto\log(\tau_{\bbj})-(s-m)(\log(1-c)+\log (n)c)<0,
\end{align}
 which implies the conclusion (2).

The proof of  (3) is analogous with (1);  thus, the details are not presented here. Then we    complete the proof of this theorem.
 
\subsection{Proof of Theorem \ref{th2}} 
 By the same proof procedure of Theorem \ref{th1}, if the assumption (A5') holds, we know that AIC and $C_p$ methods cannot choose the under-specified models almost surely. Thus, if $\phi(\alpha,c)>0$ ($\psi(\alpha,c)>0$, resp.), AIC ($C_p$, resp.) is strongly consistent; Otherwise,   AIC ($C_p$, resp.) is almost surely over-specified. 
 In addition, from \eqref{Btau} we can easily obtain the conclusion (2) in Theorem \ref{th2}. Then, we complete the proof.

\subsection{Proof of Theorem \ref{newthm}}
We consider only the case in which, under assumptions (A1) through (A5), the other cases are analogous. 
If $j$ does not exist in the true model $\bbj_*$, then ${\bm \omega} \backslash j$ includes $\bbj_*$. By Lemma \ref{le1}, we obtain 
 \begin{gather}\label{tildeA1}
  \tilde{A}_{j}\asto \log\Big(\frac{1-\alpha}{1-\alpha-c}\Big)-2c,\\
    \tilde{B}_{j}+\log(n)c\asto \log\Big(\frac{1-\alpha}{1-\alpha-c}\Big),~~~
        \tilde{C}_{j}\asto \frac{(1-\alpha)c}{1-\alpha-c}-2c.\nonumber\end{gather}
If  $j$    lies in the true model $\bbj_*$, by Lemmas \ref{le3}, we have that 
\begin{gather}\label{tildeA2}
  \tilde{A}_{j}\asto\log(\tau_{{\bm \omega}\backslash j})-\log(1-\alpha-c)-2c,\\
    \tilde{B}_{j}\asto\log(\tau_{{\bm \omega}\backslash j})-\log(1-\alpha-c)-\log(n)c,~~
    \tilde{C}_{j}\asto\frac{(1-\alpha)(\kappa_{{\bm \omega}\backslash j}+c)}{1-\ga-c}-2c.\nonumber
\end{gather}
Thus, we complete the proof based on a discussion similar to that for the proof of Theorem \ref{th1}.

\subsection{Proof of Theorem  \ref{gen_kooth}}
Since the rank of matrix $\Phi_{_{{\bm \omega}\backslash j}}$ is one, we have 
	 \begin{align*}
  \log((1-\ga)^{1-p}|(1-\ga)\bbI+\Phi_{{{\bm \omega}\backslash j}}|)=\log(1-\ga+\tr(\Phi_{{{\bm \omega}\backslash j}})),
\end{align*}
which, together with \eqref{tildeA1} and \eqref{tildeA2}, directly implies this theorem. Thus, we complete the proof.

\section{Proofs of Lemmas \ref{le1} and  \ref{le2}}\label{proofsle}
In this section, we present the technical proofs of Lemma \ref{le1} and Lemma \ref{le2}.  We first briefly describe our proof strategy and the main tools
 of RMT. 
Recall equations \eqref{nhatS} and  \eqref{Qat}, from Sylvester's determinant theorem, we obtain that
\begin{align}\label{det_eq}
	|n\widehat\bgS_{\bbj_{-t}}|
=&|\bbY'\bbQ_{\bbj_{-t}}\bbY|=|\bbY'\bbQ_{\bbj_{-t-1}}\bbY-\bbY'\bba_{t+1}\bba_{t+1}'\bbY|\\
=&|\bbY'\bbQ_{\bbj_{-t-1}}\bbY|(1-\bba_{t+1}'\bbY(\bbY'\bbQ_{\bbj_{-t-1}}\bbY)^{-1}\bbY'\bba_{t+1})\nonumber\\
=&|n\widehat\bgS_{\bbj_{-t-1}}|(1-\bba_{t+1}'\bbY(\bbY'\bbQ_{\bbj_{-t-1}}\bbY)^{-1}\bbY'\bba_{t+1}).\nonumber
\end{align}
Thus, to prove Lemma \ref{le1} and Lemma \ref{le2},  we  need to obtain only the almost sure limits of  $\bba_{t}'\bbY(\bbY'\bbQ_{\bbj_{-t}}\bbY)^{-1}\bbY'\bba_{t}$ or similar expressions  with different $\bbj_{-t}$. The proof strategy is that,  we first define a function
\begin{align*}
\hbar_n(z):= n^{-1} \bba_{t}'\bbY(n^{-1}\bbY'\bbQ_{\bbj_{-t}}\bbY-z\bbI)^{-1}\bbY'\bba_{t}:\mathbb{C}^{+}\longmapsto \mathbb{C}^{+},
\end{align*}
where $\mathbb{C}^{+}=\{z\in\mathbb{C}^{+}:\Im z>0\}$. 
Next,  we prove that outside a null set independent of $\bbj_{-t}$,  for every $z\in\mathbb{C}^{+}$, $\hbar_n(z)$ has a limit $\hbar(z)\in \mathbb{C}^{+}$. Note that by Vitali's convergence theorem (see, e.g.,  Lemma 2.14 in \citep{BaiS10S}) it is sufficient to prove for any fixed $z\in\mathbb{C}^{+}$, $\hbar_n(z)\asto\hbar(z)$. Finally, we let $z\downarrow0+0i$ and obtain almost surely $\hbar_n(0)\to \hbar(0)$. 

We remark that this proof  approach  is common in RMT to obtain the limiting special distribution (LSD) of random matrices. Thus, the present paper can be viewed as an application of RMT in multivariate statistical analysis.  
Moreover, since the type of  matrix $\bbY(\bbY'\bbQ_{\bbj_{-t}}\bbY)^{-1}\bbY'$ is special, and to the best of our knowledge,  no  known conclusions in RMT can be applied directly to obtain the limit of $\bba_{t}'\bbY(\bbY'\bbQ_{\bbj_{-t}}\bbY)^{-1}\bbY'\bba_{t}$, we have to derive some new theoretical results for our theorems. 

\subsection{An auxiliary lemma}
 We introduce some basic   results from RMT and an auxiliary lemma before proving Lemmas \ref{le1} and  \ref{le2}.
 For  any $n\times n$ matrix $\mathbf{A}_n$ with only positive  eigenvalues,  let $F^{\mathbf{A}_n}$  be the empirical  spectral distribution function  of $\mathbf{A}_n$, that is,
 $$F^{\mathbf{A}_n}(x)=\frac{1}{n}\#\{\lambda_i^{\mathbf{A}_n}\leq x\},$$
 where  $\lambda_i^{\mathbf{A}_n}$ denotes the $i$-th largest eigenvalue of $\mathbf{A}_n$ and $\#\{\cdot\}$ denotes the cardinality of the set $\{\cdot\}$.
If $F^{\mathbf{A}_n}$ has a limit distribution $F$, then we  call it the LSD of  sequence $\{\mathbf{A}_n\}$.
For any function of bounded variation $G$ on the real line, its Stieltjes transform is defined by
 $$s_G(z)=\int\frac{1}{\lambda-z}dG(\lambda),~~z\in\mathbb{C}^{+}.$$
 If matrix $\bbA$ is invertible and for any $p\times n$ matrix $\bbC$, the following formulas will be used frequently,
 \begin{align}\label{inv1}
(\bbA-\bbC\bbC')^{-1}&=\bbA^{-1}+\bbA^{-1}\bbC(\bbI-\bbC'\bbA^{-1}\bbC)^{-1}\bbC'\bbA^{-1},
 \end{align}
 which  immediately implies 
 \begin{align}\label{inv2}
(\bbA-\bbC\bbC')^{-1}\bbC&=\bbA^{-1}\bbC(\bbI-\bbC'\bbA^{-1}\bbC)^{-1}
\end{align}
\begin{align}\label{inv3}
\bbC'(\bbA-\bbC\bbC')^{-1}&=(\bbI-\bbC'\bbA^{-1}\bbC)^{-1}\bbC'\bbA^{-1}.
 \end{align}
 For any $z\in\mathbb{C}^{+}$, we also have 
 \begin{align}\label{in-out}
 	\bbC(\bbC'\bbC-z\bbI_n)^{-1}\bbC'&=\bbI_p+z(\bbC\bbC'-z\bbI_p)^{-1},
 \end{align}
 which is called the in-out-exchange formula in the sequel.
 The above equations are straightforward to obtain by basic linear algebra theory; thus, we omit the detailed calculations. 
 
A key tool for  the  proofs of  Lemma \ref{le1} and  Lemma \ref{le2} is  the  following lemma, whose proof will be postponed to the Appendix because of the space limit. 
\begin{lemma}\label{mainle} Let $\bbM:=\bbM(z)=p^{-1}\bbE'\bbQ_{\bbj_{-t}}\bbE-z\bbI_p$,  $\bga_1$ and $\bga_2$ be  ${n\times 1}$-vectors, and $\bga_3$ be a  ${p\times 1}$-vector and  assume that $\bga_1$, $\bga_2$, $\bga_3$ are all   bounded in  Euclidean norm. Then, under assumptions {\rm (A1) through (A4)} and for any integer $t$,  we have  that for any $ z\in\mathbb{C}$,
\be\label{le3.1.1}
\bga'_1\bbM^{-1}\bga_2+\frac{\bga'_1\bga_2}{z(1+\underline{s}_t(z)-\frac{1-c_n-\alpha_{m-t}}{ c_nz})}\asto 0,
\ee
\be\label{le3.1.2}
\frac{1}{\sqrt{p}}\bga_1'\bbM^{-1}\bbE'\bga_3\asto 0,
\ee
and
\begin{align}\label{le3.1.3}
\frac{1}{p}&\bga'_1\bbE\bbM^{-1}\bbE'\bga_2\nonumber\\
&+\frac{\bga_{1}'\bga_2}{z(1+\underline{s}_t(z)+\frac{c_n-1+\alpha_{m-t}}{c_nz})}
+\frac{\frac{1}{\underline{s}_t(z)+1}\bga_{1}'\bbQ_{\bbj_t}\bga_{2}}{z^2(1+\underline{s}_t(z)+\frac{c_n-1+\alpha_{m-t}}{c_nz})^2}
\asto 0,
\end{align}
where $\underline{s}_t(z)$  is the Stieltjes transform of  the LSD of  $\frac1p\bbE'\bbQ_{\bbj_{-t}}\bbE$.  
\begin{remark}
	From (1.4) in \citep{SilversteinC95A}, we have that  the Stieltjes transform $\underline{s}_t(z)$ is the unique solution on the upper complex plane to the equation
\begin{align*}
z=-\frac1{\underline{s}_t(z)}+\frac1{c_n}\int\frac{t}{1+t\underline{s}_t(z)}dH(t),
\end{align*}
where 
$H$ is the LSD of $\bbQ_{\bbj_{-t}}$. Thus, we obtain that  $H(\{0\})=\alpha_{m-t}$,  $H(\{1\})=1-\alpha_{m-t}$ and 
\begin{align*}
z=-\frac1{\underline{s}_t}+\frac1{c_n}\frac{1-\alpha_{m-t}}{1+\underline{s}_t},	
\end{align*}
which implies
\begin{align}\label{le3.1.3_2}
\frac{z(1+\underline{s}_t(z)+\frac{c_n-1+\alpha_{m-t}}{c_nz})+\frac{1}{\underline{s}_t(z)+1}}{z^2(1+\underline{s}_t(z)+\frac{c_n-1+\alpha_{m-t}}{c_nz})^2}=\frac{1}{1+
  \underline s_t(z)}-1
\end{align}
and
\bqn
\underline{s}_t(z)=\frac{1-\alpha_{m-t}-c_n-c_nz\pm\sqrt{(1-\alpha_{m-t}+c_n-c_nz)^2-4c_n(1-\alpha_{m-t})}}{2c_nz}.
\eqn
On the basis of the fact that any Stieltjes transform tends to zero as $z\to\infty$, we have \begin{align*}
\underline{s}_t(z)=\frac{1-\alpha_{m-t}-c_n-c_nz+\sqrt{(1-\alpha_{m-t}+c_n-c_nz)^2-4c_n(1-\alpha_{m-t})}}{2c_nz}
\end{align*}
and
\begin{align*}
1-\frac{1}{1+\underline{s}_t(z)}=\frac{1-\alpha_{m-t}+c_n-c_nz+\sqrt{(1-\alpha_{m-t}+c_n-c_nz)^2-4c_n(1-\alpha_{m-t})}}{2(1-\alpha_{m-t}).}
\end{align*}
Letting $z\downarrow0+0i$ and together with \eqref{eqlim1} and $1-\alpha_{m-t}-c_n>0$, we conclude that  
\begin{align}\label{eqli2}
\underline{s}_t(z)\to \frac{c_n}{1-\alpha_{m-t}-c_n}
\end{align}	and
\begin{align}\label{le3.1.4}
z\left(1+\underline{s}_t(z)-\frac{1-c_n-\alpha_{m-t}}{ c_nz}\right) \to-\frac{1-\alpha_{m-t}-c_n}{c_n}.
\end{align}
Here, we have used the fact that when the imaginary part of the square root of a complex number is positive, then its real part has the same sign as the imaginary part; thus,   $$\lim_{z\downarrow0+i0}\sqrt{(1-\alpha_{m-t}+c_n-c_nz)^2-4c_n(1-\alpha_{m-t})}=-|1-\alpha_{m-t}-c_n|.$$ 
\end{remark}
\end{lemma}
Now, we are in position to prove Lemma \ref{le1} and Lemma \ref{le2}.
\subsection{Proof of Lemma \ref{le1}}
	We first prove \eqref{le1_eq1}.  
	By equation \eqref{det_eq} and the fact that for $\bbj\in\bbJ_+$, $
	|\bbY'\bbQ_{\bbj_{-t}}\bbY|=|\bbE'\bbQ_{\bbj_{-t}}\bbE|,
$
we have
\be
\log\left(\frac{|n\widehat\bgS_{\bbj_{-t}}|}{|n\widehat\bgS_{\bbj_{-t-1}}|}\right)=\log(1-\bba_{t+1}'\bbE(\bbE'\bbQ_{\bbj_{-t-1}}\bbE)^{-1}\bbE'\bba_{t+1})
\label{eq5}
\ee
and
\begin{align}\label{Cp1}
n\widehat\bgS_{\bbj_{-t}}-n\widehat\bgS_{\bbj_{-t-1}}
=-\bbE'\bba_{t+1}\bba_{t+1} '\bbE.
\end{align}
It follows	from \eqref{AIC_1} and \eqref{eq5} that 
	\begin{align*}
	\frac{1}{n}(A_\bbj-A_{\bbj_*})=\sum_{t=1}^m[\log(1-\bba_t'\bbE(\bbE'\bbQ_{\bbj_{-t}}\bbE)^{-1}\bbE'\bba_t)+2c_n].
	\end{align*}
	Since  $\bba_{t}$ is an eigenvector of $\bbQ_{\bbj_{-t}}$,  we have 
$
	\bba'_{t}\bbQ_{\bbj_{-t}}\bba_{t}=1,
$	
	which together with Lemma \ref{mainle} and  \eqref{le3.1.3_2} implies \begin{align}\label{eqlim1}
\frac1p\bba'_t\bbE(\frac1p\bbE'\bbQ_{\bbj_{-t}}\bbE-z\bbI_p)^{-1}\bbE'\bba_t \asto  1-\frac{1}{1+\underline{s}_t(z)}.\end{align}
Therefore,  by \eqref{eqli2} and as $n\to\infty$, we have 
\begin{align}\label{int1}
	\frac1{n^2}(A_{\bbj}-A_{\bbj_*})&=n^{-1} \sum_{t=1}^m\left(\log(1-\frac{c_n}{1-\alpha_{m-t}})+2c_n+o_{a.s.}(1)\right).
\end{align}
If $\lim \alpha_m>0$, then 
integration by parts indicates that \eqref{int1} tends to 
\begin{align*}
\int_{0}^{\alpha_m} (\log(1-\frac{c_n}{1-t})+2c_n)dt=2c_n\alpha_m+\log\left(\frac{(1-c_n)^{1-c_n}(1-\alpha_m)^{1-\alpha_m}}{(1-c_n-\alpha_m)^{1-c_n-\alpha_m}}\right),
\end{align*}
which implies \eqref{le1_eq1}.

\eqref{le1_eq3} is analogous; thus, we omit the details. Next, we prove \eqref{le1_eq2}. 
	It follows from \eqref{Cp_1} and \eqref{Cp1} that 
	\begin{align}\label{cpeq}
 \frac1n(C_{\bbj}-C_{\bbj_*})=\sum_{t=1}^m\left((\frac{k}n-1)\bba'_t\bbE(\bbE'\bbQ_{{\bm \omega}}\bbE)^{-1}\bbE'\bba_t+2\frac{p}n\right). 
	\end{align}
By \eqref{le3.1.3} and  \eqref{le3.1.4} and the fact that
\begin{align*}
\bba'_t\bbQ_{\bm \omega}\bba_t=0,
\end{align*} we have 
\begin{align*} 
&\bba'_t\bbE(\bbE'\bbQ_{{\bm \omega}}\bbE)^{-1}\bbE'\bba_t=\frac{c}{1-\ga-c}
+o_{a.s.}(1),
\end{align*}
which together  with \eqref{cpeq} implies 
\begin{align*}
 \frac1{n^2}(C_{\bbj}-C_{\bbj_*})&\asto \frac{c\alpha_m(\ga-1)}{1-\ga-c}+2c\alpha_m.
 \end{align*}
 Thus, we complete the proof of Lemma \ref{le1}.

\subsection{Proof of Lemma \ref{le2}}

We start with  ${\mathcal A}_{-t}$. When $t<0$,  
\bqn
{\mathcal A}_{-t}=\log\left(\frac{|\bbY'\bbQ_{\bbj_{-t}}\bbY|}{|\bbY'\bbQ_{\bbj_{-t-1}}\bbY|}\right)-2c_n. 
\eqn
Note that the index set $\bbj_{-t-1}$ contains one index $i_{-t}$ more than $\bbj_{-t}$; therefore, from the notation  $\bba_{t}=\bbQ_{\bbj_{-t}}\bbx_{i_{-t}}/\|\bbQ_{\bbj_{-t}}\bbx_{i_{-t}}\|$, 
we have
\be
{\mathcal A}_{-t}=-\log(1-\bba_{t}'\bbY(\bbY'\bbQ_{\bbj_{-t}}\bbY)^{-1}\bbY'\bba_{t})-2c_n.
\label{eq6}
\ee
To evaluate the limit of ${\mathcal A}_{-t}$, we consider
\bqa
m_{nt}:=m_{nt}(z)=-\log(1-\frac1p\bba_{t}'\bbY(\bbY'\bbQ_{\bbj_{-t}}\bbY/p-z\bbI_p)^{-1}\bbY'\bba_{t})-2c_n\nonumber,
\eqa
where $z\in\mathbb{C}^+$.
On the basis of the fact that  $\bba_{t}=\bbQ_{\bbj_{-t}}\bba_{t}$ and the in-out-exchange formula \eqref{in-out}, we  rewrite $m_{nt}$ as
\begin{align}
m_{nt}=-\log\left(-z\bba_{t}'\left(\bbQ_{\bbj_{-t}}\bbY\bbY'\bbQ_{\bbj_{-t}}/p-z\bbI_n\right)^{-1}\bba_{t}\right)-2c_n.
\label{eq7}
\end{align}
Substitute model (\ref{eq1}) into the above equation and denote 
\bqn
I_t:=I_t(z)&=&z\bba_{t}'\left(\bbQ_{\bbj_{-t}}\bbY\bbY'\bbQ_{\bbj_{-t}}/p-z\bbI_n\right)^{-1}\bba_{t}\\
&=&z\bba_{t}'\left(\bbQ_{\bbj_{-t}}(\bbX_{\bfell_l}\Theta_{\bfell_l}+\bbE)(\Theta_{\bfell_l}'\bbX_{\bfell_l}'+\bbE')\bbQ_{\bbj_{-t}}/p-z\bbI_n\right)^{-1}\bba_{t}
\eqn
where $\bfell_t=\{i_{1},\cdots, i_{-t}\}$. Define $\bbB_1=\bbQ_{\bbj_{-t}}\bbX_{\bfell_l}(\bbX_{\bfell_l}'\bbQ_{\bbj_{-t}}\bbX_{\bfell_l})^{-1/2}$ and select $\bbB_2$ such that $\bbB=\big(\bbB_1\vdots \bbB_2\big)$ is an $n\times n$ orthogonal matrix. Then, we have 
$$
I_t
=z\bba_{t}'\bbB\left(\bbB'\bbQ_{\bbj_{-t}}(\bbX_{\bfell_l}\Theta_{\bfell_l}+\bbE)(\Theta_{\bfell_l}'\bbX_{\bfell_l}'+\bbE')\bbQ_{\bbj_{-t}}\bbB/p-z\bbI_n\right)^{-1}\bbB'\bba_{t}.
$$
With  $\bba_{t}'\bbB_2=0$, we obtain 
\begin{align*}
	I_t
=&z\tilde\bba_{t}'\bigg(\bbB'_1\bbQ_{\bbj_{-t}}(\bbX_{\bfell_l}\Theta_{\bfell_l}+\bbE)(\Theta_{\bfell_l}'\bbX_{\bfell_l}'+\bbE')\bbQ_{\bbj_{-t}}\bbB_1/p\\
& -\bbB_1'\bbQ_{\bbj_{-t}}(\bbX_{\bfell_l}\Theta_{\bfell_l}+\bbE)\bbE'\bbQ_{\bbj_{-t}}\bbB_2(\bbB'_2\bbQ_{\bbj_{-t}}\bbE\bbE'\bbQ_{\bbj_{-t}}\bbB_2
/p-z\bbI_{n+t})^{-1}\\
& \ \cdot\bbB_2'\bbQ_{\bbj_{-t}}\bbE(\bbE'+\bbX_{\bfell_l}'\Theta_{\bfell_l})\bbQ_{\bbj_{-t}}\bbB_1/p^2-z\bbI_{-t}\bigg)^{-1}\tilde\bba_{t}
\end{align*}
where $\tilde \bba_{t}=\bbB_1'\bba_{t}$. By applying in-out-exchange formula \eqref{in-out}  to the term \begin{align*}
\bbE'\bbQ_{\bbj_{-t}}\bbB_2(\bbB'_2\bbQ_{\bbj_{-t}}\bbE\bbE'\bbQ_{\bbj_{-t}}\bbB_2
/p-z\bbI_{n-k_t})^{-1} \bbB_2'\bbQ_{\bbj_{-t}}\bbE/p,
\end{align*}
we obtain
\begin{align*}
I_t=-\tilde\bba_{t}'\bigg(\frac1p\bbB'_1\bbQ_{\bbj_{-t}}(\bbX_{\bfell_l}\Theta_{\bfell_l}+\bbE)&(\frac1p\bbE'\bbQ_{\bbj_{-t}}\bbB_2\bbB_2'\bbQ_{\bbj_{-t}}\bbE-z\bbI_p)^{-1}\\
&\times(\Theta_{\bfell_l}'\bbX_{\bfell_l}'+\bbE')\bbQ_{\bbj_{-t}}\bbB_1+\bbI_{-t}
\bigg)^{-1}\tilde\bba_{t},
\end{align*}
which together with $\bbM=\frac1p\bbE'\bbQ_{\bbj_{-t}}\bbE-z\bbI_{p}$
implies 
\begin{align*}
I_t=-\tilde\bba_{t}'\bigg(\frac1p\bbB'_1\bbQ_{\bbj_{-t}}(\bbX_{\bfell_l}\Theta_{\bfell_l}+\bbE)&(\bbM-\frac1p\bbE'\bbQ_{\bbj_{-t}}\bbB_1\bbB_1'\bbQ_{\bbj_{-t}}\bbE)^{-1}\\
&\times 
(\Theta_{\bfell_l}'\bbX_{\bfell_l}'+\bbE')\bbQ_{\bbj_{-t}}\bbB_1+\bbI_{-t}
\bigg)^{-1}\tilde\bba_{t}.
\end{align*}
Equations \eqref{inv1}-\eqref{in-out} can be used to separate $I_t$ into the following four parts,
\begin{align}\label{Itz}
	I_t=-\tilde\bba_{t}'\left(I_{1t}+I_{2t}+I_{2t}'+I_{3t}\right)^{-1}\tilde\bba_{t}
\end{align}
where
\begin{align*}
I_{1t}=&\frac1p\bbB_1'\bbQ_{\bbj_{-t}}\bbX_{\bfell_l}\Theta_{\bfell_l}\bbM^{-1}\Theta_{\bfell_l}'\bbX'_{\bfell_l}\bbQ_{\bbj_{-t}}\bbB_1+\frac1{p^2}\bbB_1'\bbQ_{\bbj_{-t}}\bbX_{\bfell_l}\Theta_{\bfell_l}\bbM^{-1}\bbE'\bbQ_{\bbj_{-t}}\bbB_1\\
&\ (\bbI_{-t}-\frac1p\bbB_1'\bbQ_{\bbj_{-t}}\bbE\bbM^{-1}\bbE'\bbQ_{\bbj_{-t}}\bbB_1)^{-1}\bbB_1'\bbQ_{\bbj_{-t}}\bbE\bbM^{-1}\Theta_{\bfell_l}'\bbX'_{\bfell_l}\bbQ_{\bbj_{-t}}\bbB_1;\\
I_{2t}=&\frac1p\bbB_1'\bbQ_{\bbj_{-t}}\bbX_{\bfell_l}\Theta_{\bfell_l}\bbM^{-1}\bbE'\bbQ_{\bbj_{-t}}\bbB_1(\bbI_{-t}-\frac1p\bbB_1'\bbQ_{\bbj_{-t}}\bbE\bbM^{-1}\bbE'\bbQ_{\bbj_{-t}}\bbB_1)^{-1};\\
I_{3t}=&(\bbI_{-t}-\frac1p\bbB_1'\bbQ_{\bbj_{-t}}\bbE\bbM^{-1}\bbE'\bbQ_{\bbj_{-t}}\bbB_1)^{-1}.
\end{align*}
It follows from  \eqref{le3.1.3} and  \eqref{le3.1.3_2} that 
\begin{align*}
  \frac1p\bbB_1'\bbQ_{\bbj_{-t}}\bbE\bbM^{-1}\bbE'\bbQ_{\bbj_{-t}}\bbB_1\stackrel{a.s.}{\to}(1-\frac{1}{1+
  \underline s_t(z)})\bbI_{-t},
\end{align*}
which implies 
\begin{align}\label{I4t}
  I_{3t}\stackrel{a.s.}{\to}({1+
  \underline s_t(z)})\bbI_{-t}.
\end{align}
Moreover, from  \eqref{le3.1.1}, \eqref{le3.1.2} and assumption (A4), we  have   
\begin{align*}
  \frac1{p}\bbB_1'\bbQ_{\bbj_{-t}}\bbX_{\bfell_l}\Theta_{\bfell_l}\bbM^{-1}\bbE'\bbQ_{\bbj_{-t}}\bbB_1\stackrel{a.s.}{\to} {\bf 0}_{-t},
\end{align*}
and 
\begin{align*}
  \frac1p\bbB_1'\bbX_{\bfell_l}\Theta_{\bfell_l}\bbM^{-1}\Theta_{\bfell_l}'\bbX'_{\bfell_l}\bbB_1+\frac{p^{-1}\gD_t}{z(1+\underline{s}_t(z)-\frac{1-c_n-\alpha_{m-t}}{ c_nz})}\stackrel{a.s.}{\to} 0,
\end{align*}
which together with  \eqref{I4t} imply
\begin{align*}
   I_{t}+\tilde\bba_{t}'\left(({1+
  \underline s(z)})\bbI_{-t}-\frac{p^{-1}\gD_t}{z(1+\underline{s}_t(z)-\frac{1-c_n-\alpha_{m-t}}{ c_nz})}\right)^{-1}\tilde\bba_{t}\stackrel{a.s.}{\to}0.\end{align*}
%
%
As $z\downarrow0+0i$ and with \eqref{eqli2} and  the notation $$ \delta_t:=\tilde\bba_{t}'\left({(1-\alpha_m)}\bbI_{-t}+n^{-1}\Delta_t\right)^{-1}\tilde\bba_{t},$$ we have  
\begin{align*}
I_t(0)+(1-\alpha_{m-t}-c_n)\delta_t\stackrel{a.s.}\to0.
\end{align*}
Here one should notice that $-t\leq k_*<\infty$, and thus $\alpha_{m-t}$  and $\alpha_{m}$ have the same limit. Therefore, we conclude that for $t<0$,
\begin{align*}
{\mathcal A}_{-t}+\log\delta_t+\log(1-\alpha_{m-t}-c_n)+2c_n\stackrel{a.s.}\to0.
\end{align*}

Next, we consider ${\mathcal C}_{-t}$ when $t<0$. 
Recall that 
\begin{align*}
{\mathcal C}_{-t}=	(1-\alpha_k)\bba_{t}'\bbY(\bbE'\bbQ_{\bm \omega}\bbE)^{-1}\bbY'\bba_{t}-2c_n,
\end{align*}
and let 
\bqn
J_t(z)=p^{-1}\bba_{t}'\bbY\left(\bbE'\bbQ_{{\bm \omega}}\bbE/p-z\bbI_p\right)^{-1}\bbY'\bba_{t}.
\eqn
Then, 
by substituting model (\ref{eq1}) into the above equation, we obtain
\begin{align*}
J_t(z)=&\frac1p\bba_{t}'( \bbX_{\bfell_l}\Theta_{\bfell_l}+\bbE)\left(\bbE'\bbQ_{{\bm \omega}}\bbE/p-z\bbI_p\right)^{-1}(\Theta_{\bfell_l}'\bbX_{\bfell_l}' +\bbE')\bba_{t}\\
=&\frac1p\bba_{t}' \bbX_{\bfell_l}\Theta_{\bfell_l}\left(\bbE'\bbQ_{{\bm \omega}}\bbE/p-z\bbI_p\right)^{-1}\Theta_{\bfell_l}'\bbX_{\bfell_l}' \bba_{t}\\&+\frac1p\bba_{t}' \bbX_{\bfell_l}\Theta_{\bfell_l}\left(\bbE'\bbQ_{{\bm \omega}}\bbE/p-z\bbI_p\right)^{-1}\bbE'\bba_{t}\\
&+\frac1p\bba_{t}'\bbE\left(\bbE'\bbQ_{{\bm \omega}}\bbE/p-z\bbI_p\right)^{-1}\Theta_{\bfell_l}'\bbX_{\bfell_l}' \bba_{t}\\
&+\frac1p\bba_{t}'\bbE\left(\bbE'\bbQ_{{\bm \omega}}\bbE/p-z\bbI_p\right)^{-1}\bbE'\bba_{t}\\
:=&J_{1t}+J_{2t}+J_{2t}'+J_{3t},
\end{align*}
where $\bfell_t=\{i_{-t+1},\cdots, i_s\}$. 
It follows from Lemma \ref{mainle} that 
\begin{gather*}
 J_{1t}+\frac{\frac1p\bba_{t}' \bbX_{\bfell_l}\Theta_{\bfell_l}\Theta_{\bfell_l}'\bbX_{\bfell_l}' \bba_{t}}{z(1+\underline{s}(z)-\frac{1-c_n-\alpha_k}{ c_nz})}\stackrel{a.s.}\to0,\\
J_{2t}\stackrel{a.s.}\to0 ~~~\mbox{and}~ ~~
J_{3t}+\frac{1}{z(1+\underline{s}(z)-\frac{1-c_n-\alpha_k}{ c_nz})}\stackrel{a.s.}\to0,
\end{gather*}
where  $\underline{s}(z)$  is the Stieltjes transform of  the LSD of  $\frac1p\bbE'\bbQ_{{\bm \omega}}\bbE$.  
Note that 
\begin{gather*}
  \frac1p\bba_{t}' \bbX_{\bfell_l}\Theta_{\bfell_l}\Theta_{\bfell_l}'\bbX_{\bfell_l}' \bba_{t}\\
  = \frac1p\tilde\bba_{t}'(\bbX'_{\bfell_l}\bbQ_{\bbj_t}\bbX_{\bfell_l})^{1/2}\Theta_{\bfell_l}\Theta_{\bfell_l}'(\bbX'_{\bfell_l}\bbQ_{\bbj_t}\bbX_{\bfell_l})^{1/2}\tilde\bba_{t}
 =c^{-1}\eta_t.
\end{gather*}
Therefore, letting $z\downarrow0+0i$,  we obtain
\begin{align*}
J_t(0)-\frac{\eta_t+c_n}{1-\alpha_k-c_n}\stackrel{a.s.}\to0,
\end{align*}
which implies 
\begin{align*}
  {\mathcal C}_{-t}=\frac{(1-\alpha_k)(\eta_t+c_n)}{1-c_n-\alpha_k}-2c_n+o_{a.s.}(1).
\end{align*}
Thus, we complete the proof of Lemma \ref{le2}.

\section{Conclusion and discussion}
In the present paper, we discussed the strong consistency of three fundamental selection criteria, i.e., the AIC, BIC, and $C_p$,  in the linear regression model under an LLL framework. We presented the sufficient and necessary conditions for  their strong consistency and determined how the dimension and size of the explanatory variables and the sample size affect the selection accuracy. Then, we proposed general  KOO criteria based on KOO methods and showed the sufficient  conditions for their strong consistency. The general KOO criteria have numerous advantages, such as simplicity of expression, ease of computation, limited restrictions, and fast convergence. 
 
The present paper considers only the case in which $\alpha+c<1$ because $\hat\bgS_{\bbj}$ may otherwise be singular. The singularity of $\hat\bgS_{\bbj}$ can be avoided by using a ridge-type estimator of the covariance matrix
 (e.g., \citep{YamamuraY10V,ChenP11R}) or by rewriting the statistics with nonzero eigenvalues (e.g., \cite{ZhangH17O}). 
 The consistency properties of the model selection criteria must also be studied when the true model size $k_*$  is large, i.e., $k_*/n$ tends to a constant as $n\to\infty$. All three topics require clarification of the theoretical results of RMT, which is left for future research.
 
In addition, we intend to obtain the asymptotic distributions of $\breve{A}_{j}$ and $\breve{C}_{j}$.
 In the present paper, we obtained only almost surely the limits of $\breve{A}_{j}$ and $\breve{C}_{j}$ and not their convergence rates. If we can determine their asymptotic distributions,  such results can be used to construct more reasonable selection criteria. According to the results of \citep{BaiM07A}, we guessed that the convergence 
rate should be $O(n^{-1/2})$. 
 
The main technical tool of the present paper is RMT. In the past two decades, the power of RMT has been partially demonstrated through high-dimensional multivariate analysis. Most subjects in classical multivariate analysis, including the model selection problems considered in this paper, can be (or have been) reexamined by RMT in high-dimensional settings. We hope that RMT will attract more attention in the future research of high-dimensional MLR.

\appendix
\section{Simulation results under non-normal distributions }
In this section, we present simulation results under a standardized $t$ distribution with three degrees of freedom and a standardized chi-square distribution with two degrees of freedom. Please see Figures \ref{fig_t3_1}-\ref{fig_chi2_2}. These results  are similar to those of the normal distribution case. Thus, we guess  that the consistency of these estimators depends on only the first two moments.

\begin{figure}[htbp!]\subfigure[Setting I]{
		\includegraphics[width=12.2cm,height=6cm]{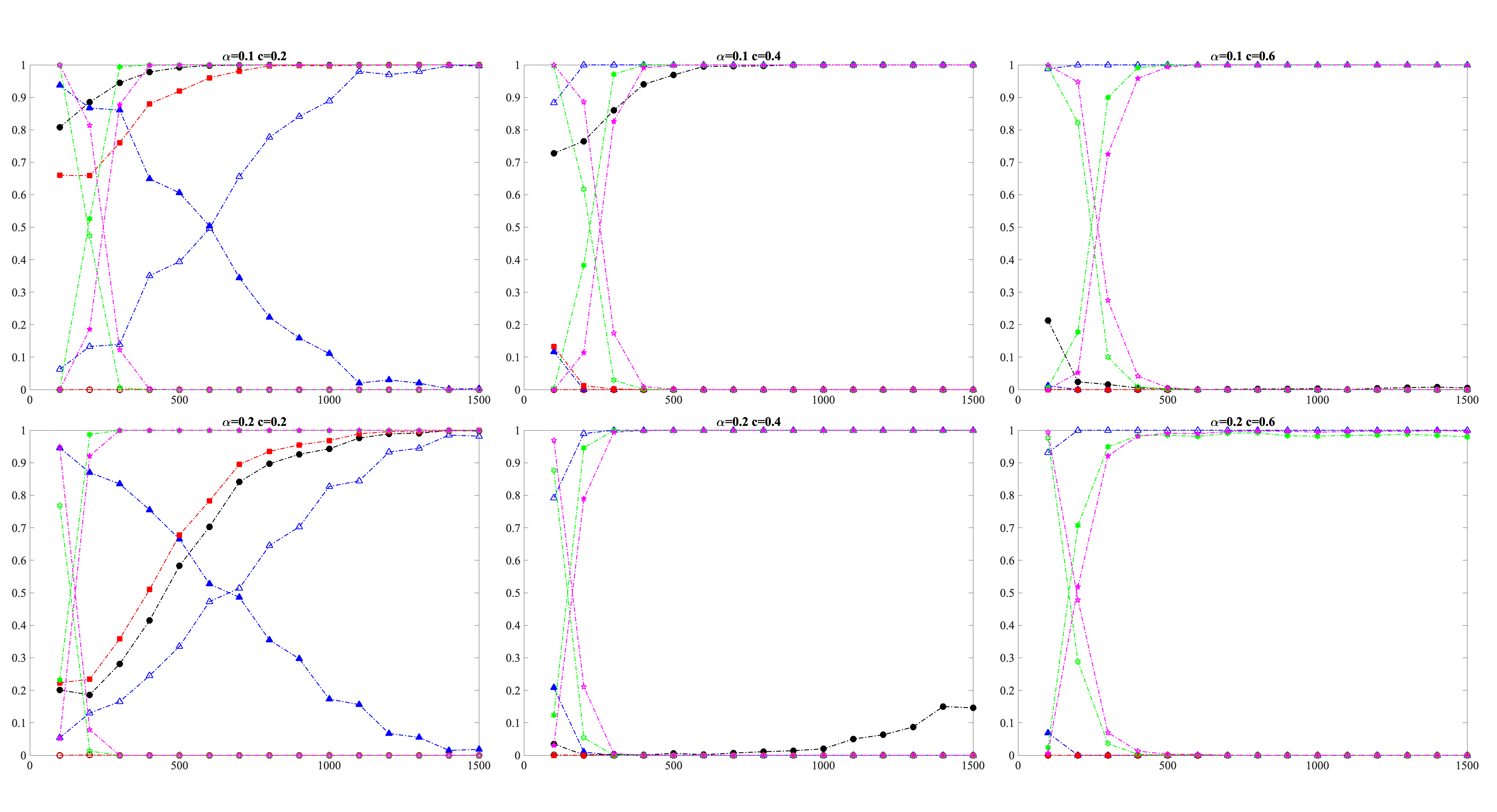}}
		\subfigure[Setting II]{
		\includegraphics[width=12.2cm,height=6cm]{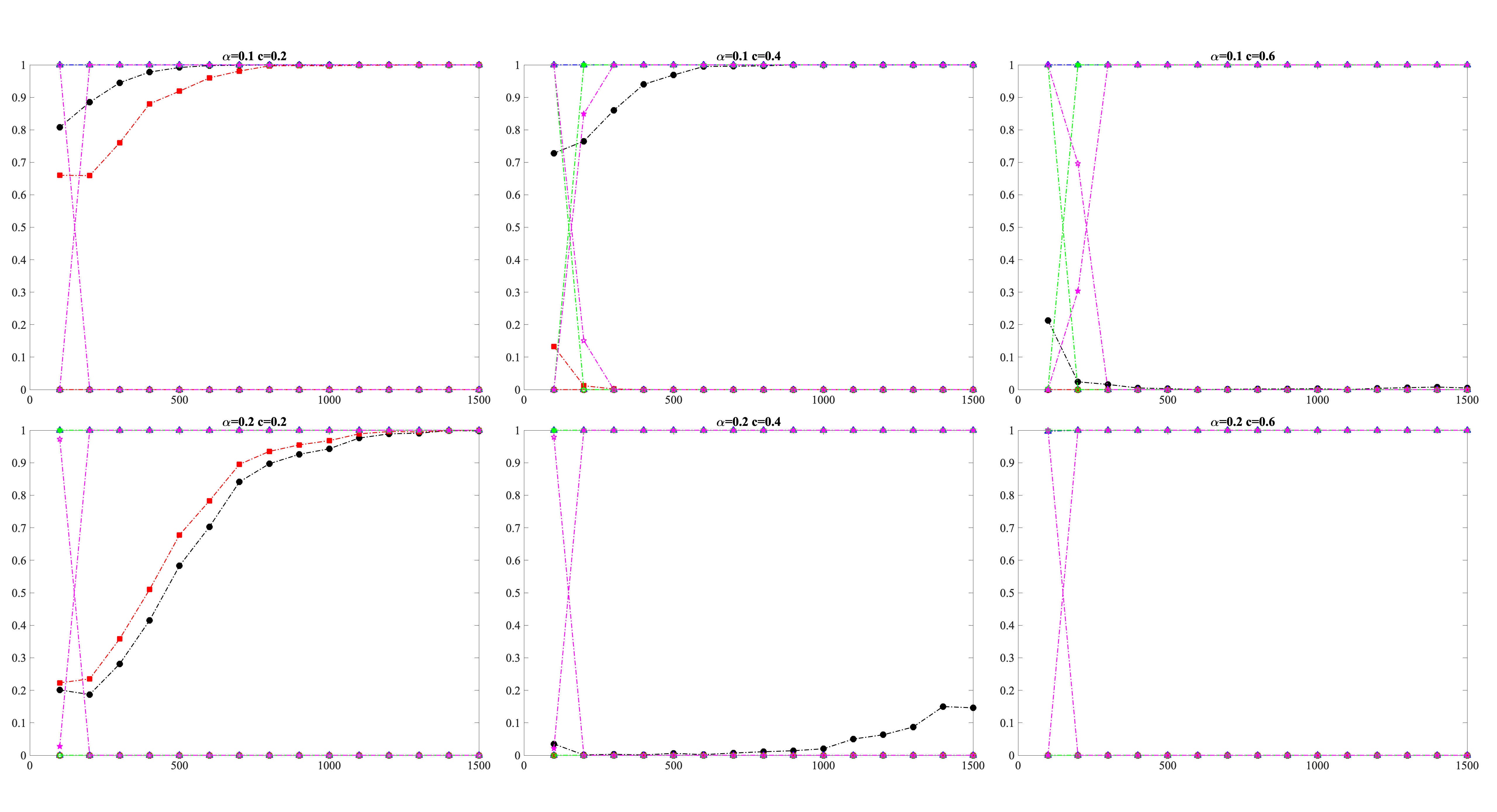}}
		\caption{Selection percentages under the AIC, BIC and $C_p$ for Settings I and II with a standard $t_3$ distribution. The horizontal axes represent the sample size $n$, and the vertical axes represent selection percentage. Black solid circles, blue solid triangles, red  solid squares, green solid hexagrams and magenta solid pentagrams
		denote the selection percentages of $\tilde\bbj_A=\bbj_*$,  $\tilde\bbj_B=\bbj_*$, $\tilde\bbj_C=\bbj_*$, $\breve\bbj_A=\bbj_*$ and $\breve\bbj_C=\bbj_*$, respectively. Correspondingly, black  circles, blue  triangles, red   squares, green  hexagrams and magenta  pentagrams
		denote the selection percentages of $\tilde\bbj_A\in\bbJ_-$,  $\tilde\bbj_B\in\bbJ_-$, $\tilde\bbj_C\in\bbJ_-$, $\breve\bbj_A\in\bbJ_-$ and $\breve\bbj_C\in\bbJ_-$, respectively. 
}
		\label{fig_t3_1} 
\end{figure}

\begin{figure}[htbp!]
\subfigure[Setting I]{
		\includegraphics[width=6.1cm,height=6cm]{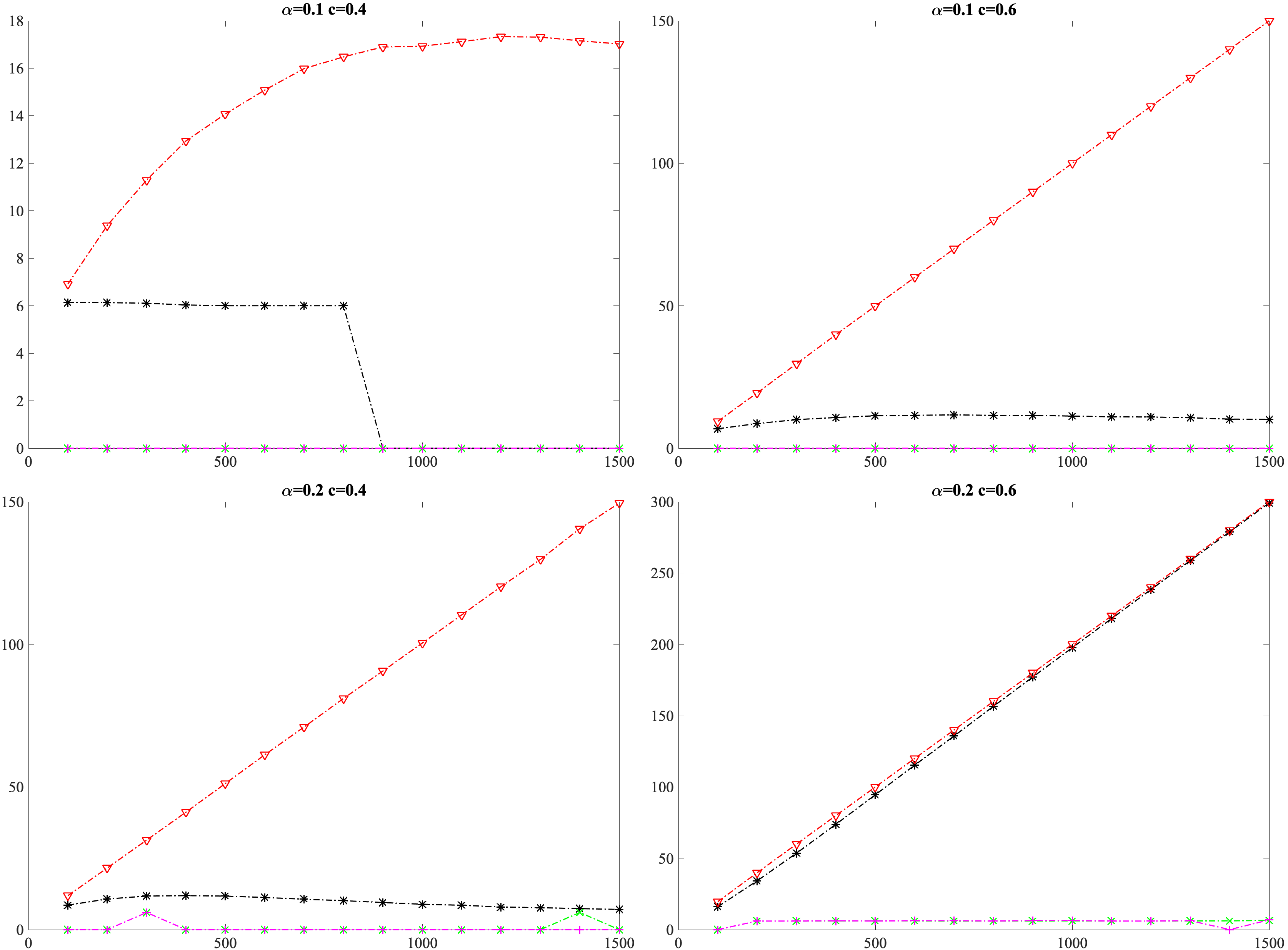}}
		\subfigure[Setting II]{
		\includegraphics[width=6.1cm,height=6cm]{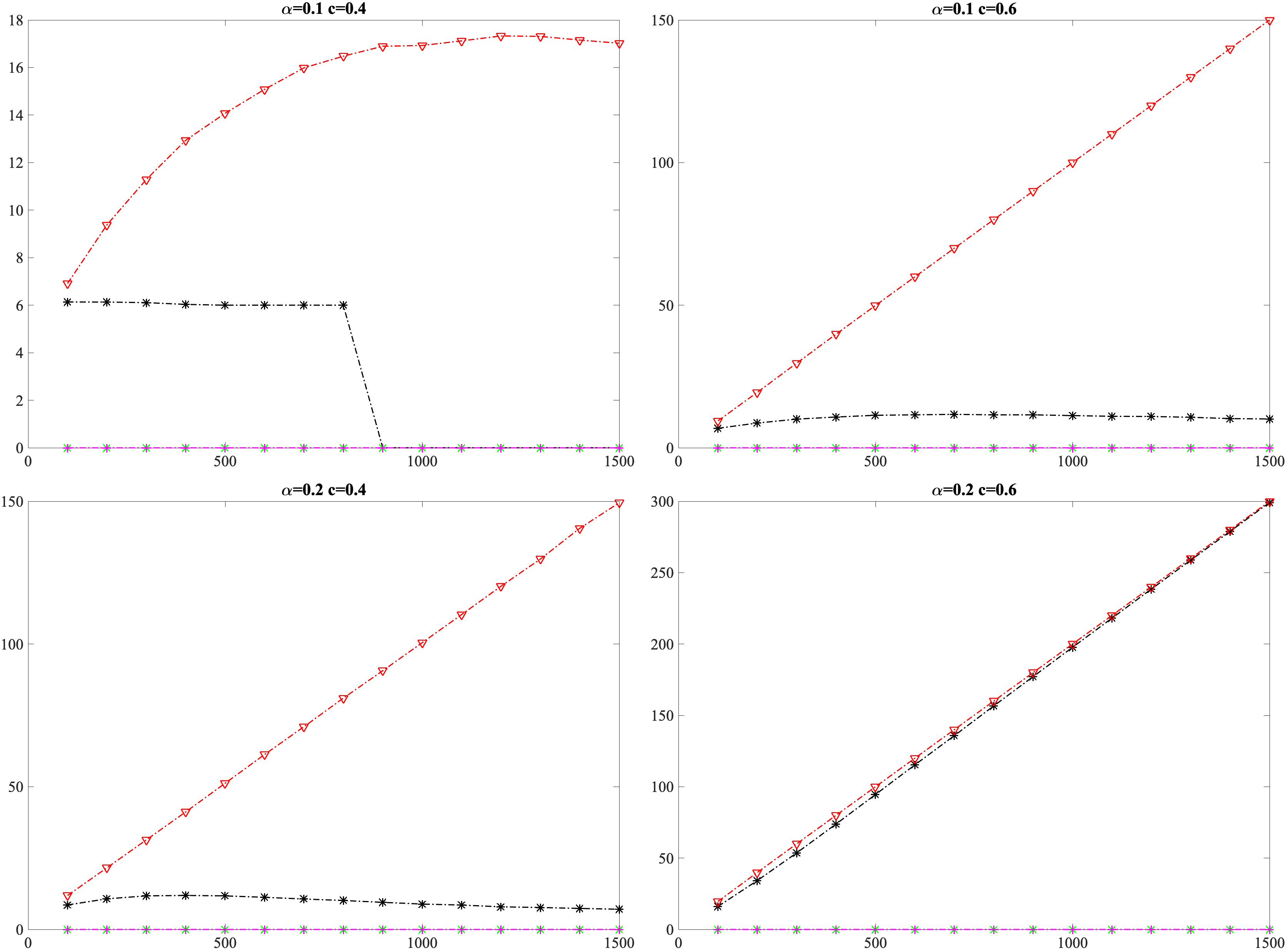}}
		\caption{Overspecified model sizes of the AIC and $C_p$ for Settings I and II with a standard $t_3$ distribution. The horizontal axes represent the sample size $n$, and the vertical axes represent the model size. 	Black  asterisks,  red   right-pointing triangles, green  crosses and magenta  plus signs
		denote the average sizes  of $\tilde\bbj_A\in\bbJ_+$,  $\tilde\bbj_C\in\bbJ_+$, $\breve\bbj_A\in\bbJ_+$ and $\breve\bbj_C\in\bbJ_+$, respectively.  In this figure, we let $0/0=0$.
		}
		\label{fig_t3_2} 
\end{figure}

\begin{figure}[htbp!]\subfigure[Setting I]{
		\includegraphics[width=12.2cm,height=6cm]{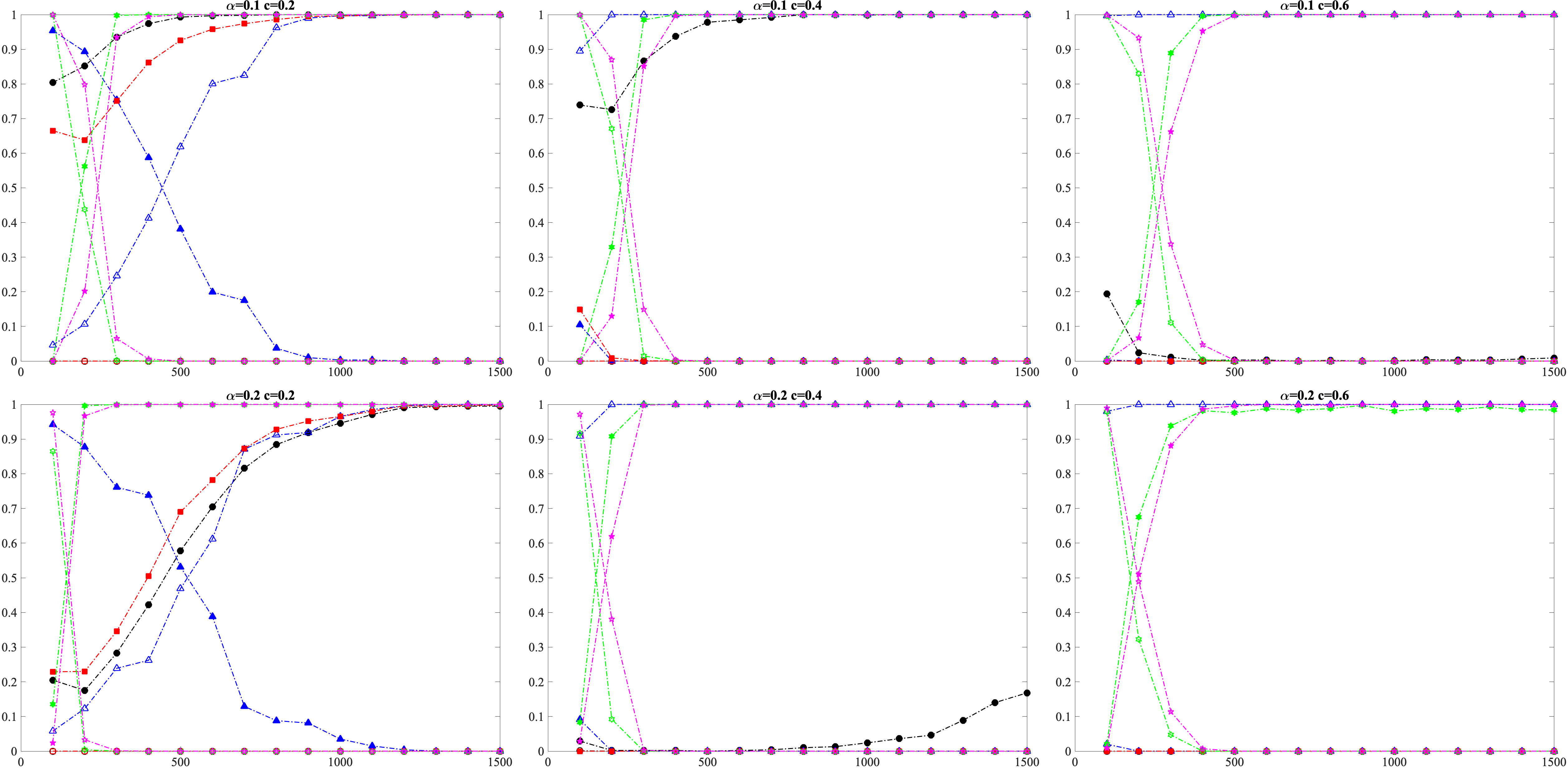}}
		\subfigure[Setting II]{
		\includegraphics[width=12.2cm,height=6cm]{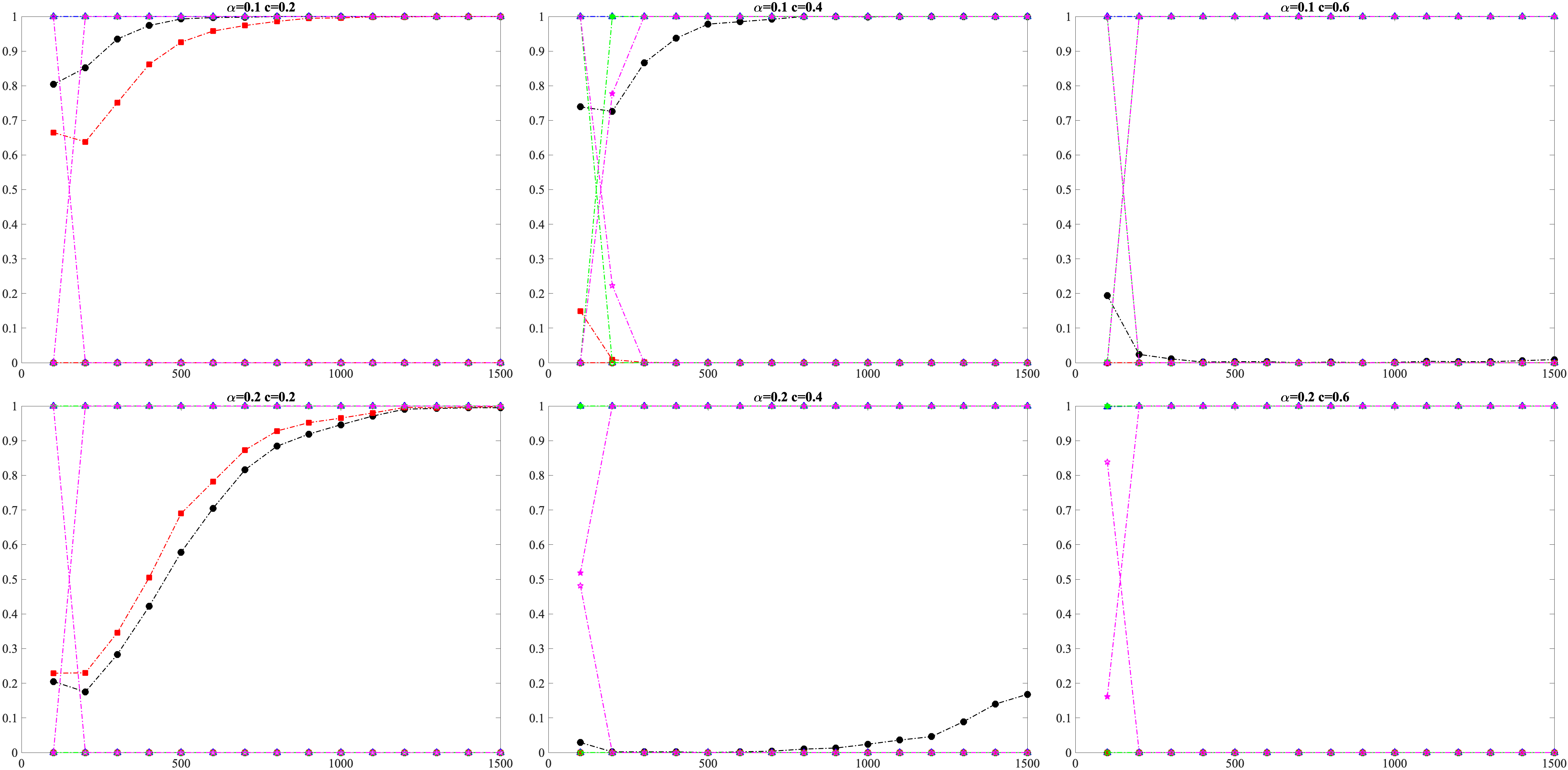}}
		\caption{Selection percentages under the AIC, BIC and $C_p$ for Settings I and II with  standard $\chi^2_2$ distribution. The horizontal axes represent the sample size $n$, and the vertical axes represent the selection percentages. Black solid circles, blue solid triangles, red  solid squares, green solid hexagrams and magenta solid pentagrams
		denote the selection percentages of $\tilde\bbj_A=\bbj_*$,  $\tilde\bbj_B=\bbj_*$, $\tilde\bbj_C=\bbj_*$, $\breve\bbj_A=\bbj_*$ and $\breve\bbj_C=\bbj_*$, respectively. Correspondingly, black  circles, blue  triangles, red   squares, green  hexagrams and magenta  pentagrams
		denote the selection percentages of $\tilde\bbj_A\in\bbJ_-$,  $\tilde\bbj_B\in\bbJ_-$, $\tilde\bbj_C\in\bbJ_-$, $\breve\bbj_A\in\bbJ_-$ and $\breve\bbj_C\in\bbJ_-$, respectively.
}
		\label{fig_chi2_1} 
\end{figure}
\begin{figure}[htbp!]
\subfigure[Setting I]{
		\includegraphics[width=6.1cm,height=6cm]{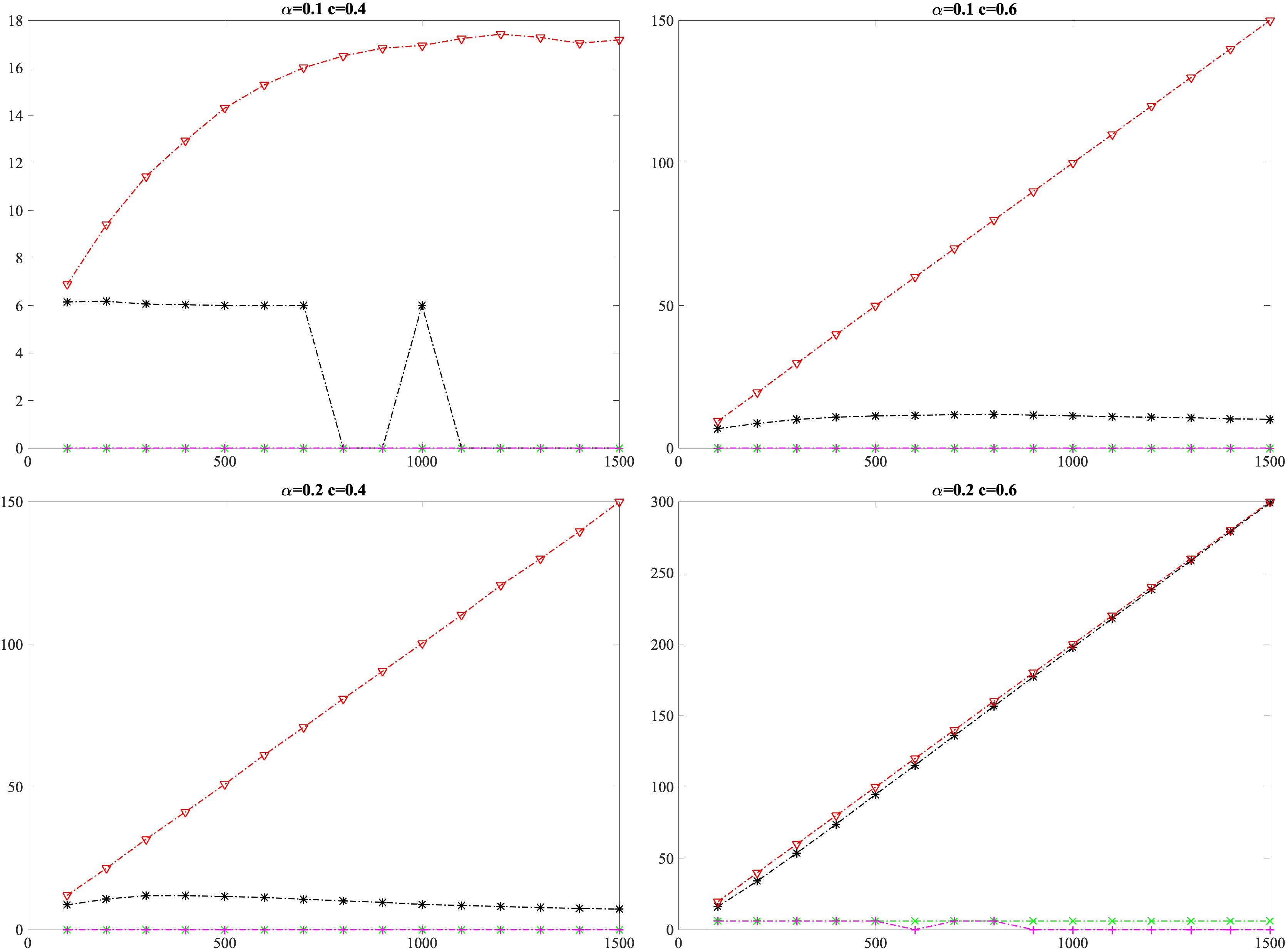}}
		\subfigure[Setting II]{
		\includegraphics[width=6.1cm,height=6cm]{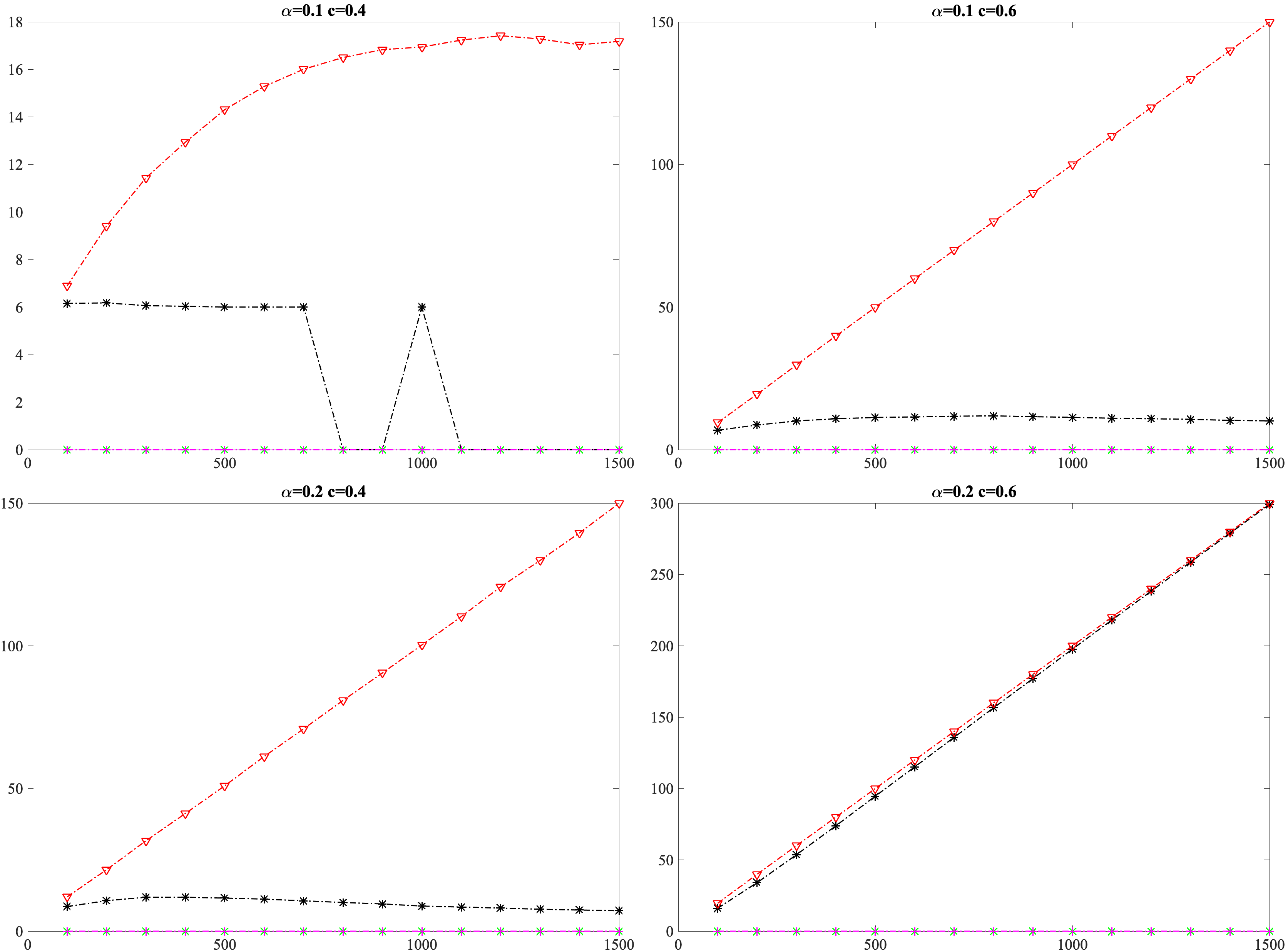}}
		\caption{Overspecified model sizes of the AIC and $C_p$ for Settings I and II with a standard $chi_2^2$ distribution. The horizontal axes represent the sample size $n$, and the vertical axes represent the model size. 	Black  asterisks,  red   right-pointing triangles, green  crosses and magenta  plus signs
		denote the average sizes  of $\tilde\bbj_A\in\bbJ_+$,  $\tilde\bbj_C\in\bbJ_+$, $\breve\bbj_A\in\bbJ_+$ and $\breve\bbj_C\in\bbJ_+$, respectively.  In this figure we let $0/0=0$.}
		\label{fig_chi2_2}  
\end{figure}
\section{Proof  of Lemma 5.1}
We now present the proof  of Lemma \ref{mainle}. In the following, $C$ represents a generic constant whose value may vary from line to line. 
%
%
%

\begin{proof}[Proof of Lemma \ref{mainle}] 
According to the truncation approach of \cite{BaiM07A}, we can assume that the  variables $\{e_{ij},i=1\dots n,j=1\dots p\}$ satisfy the following additional condition: 
\begin{align*}
	 |e_{ij}|<C, \quad  \mbox{for all $i,j$ and sufficiently large constant $C$}.
\end{align*}
 As the proof is very similar, we omit the details here.  
Let $\bga_{k1}$ be the $p-1$ subvector of $\bga_1$ with the $k$-th entry removed, and let $\alpha_{k1}$ be the $k$-th entry of $\bga_1$. Analogously, we can define $\bga_{k2} $ and $ \alpha_{k2}$. Define $\bbM_k=\frac1p\bbE_k'\bbQ_{\bbj_{-t}}\bbE_k-z\bbI_{p-1}$, where $\bbE_k$ is the $n\times (p-1)$ submatrix of $\bbE$ with the $k$-th column removed. Denote by $\E_k$ the conditional expectation given $\{\bbe_1,\cdots,\bbe_k\}$ and by $\E_0$  the unconditional expectation, where $\bbe_i $ is the $n$-vector of the $i$-th column 
of $\bbE$. 
Then, by inverting the block matrix, we obtain 
 \begin{gather}
 \bga'_1\bbM^{-1}\bga_2=\bga_{k1}'\bbM_k^{-1}\bga_{k2}+\frac{1}{\beta_kp^2}\bga_{k1}'\bbM_k^{-1}\bbE_k'\bbQ_{\bbj_{-t}}\bbe_k\bbe_k'\bbQ_{\bbj_{-t}}\bbE_k\bbM_k^{-1}\bga_{k2}\nonumber\\
 -\frac{\alpha_{k2}}{\beta_kp}\bga_{k1}'\bbM_k^{-1}\bbE_k'\bbQ_{\bbj_{-t}}\bbe_k-\frac{\alpha_{k1}}{\beta_kp}\bbe_k'\bbQ_{\bbj_{-t}}\bbE_k\bbM_k^{-1}\bga_{k2}+\frac{\alpha_{k1}\alpha_{k2}}{\beta_k},\label{beta1}
 \end{gather}
 where 
 \begin{align*}
 \beta_k&=\frac1p\bbe_k'\bbQ_{\bbj_{-t}}\bbe_k-z-\frac1{p^2}\bbe_k'\bbQ_{\bbj_{-t}}\bbE_k\bbM_k^{-1}\bbE_k'\bbQ_{\bbj_{-t}}\bbe_k\\
& =-z(1+\frac1p\bbe_k'\bbQ_{\bbj_{-t}}\widehat{\bbM}_k^{-1}\bbQ_{\bbj_{-t}}\bbe_k).
 \end{align*}
 The last equation is from the in-out-exchange formula \eqref{in-out}
and  $$\widehat{\bbM}_k:=\frac1p\bbQ_{\bbj_{-t}}\bbE_k\bbE_k'\bbQ_{\bbj_{-t}}-z\bbI_n.$$  Denote $\beta^{tr}_k=-z[1+\frac1ptr(\bbQ_{\bbj_{-t}}\widehat{\bbM}_k^{-1})]$. It  follows that 
 \begin{align}\label{betr}
	\frac1{\beta_k}=\frac1{\beta^{tr}_k}+\frac{\beta_k-\beta^{tr}_k}{p\beta_k\beta^{tr}_k}=\frac1{\beta^{tr}_k}+\frac{\xi_k}{\beta_k\beta^{tr}_k},
\end{align}
where $\xi_k=p^{-1}\bbe_k'\bbQ_{\bbj_{-t}}\widehat{\bbM}_k^{-1}\bbQ_{\bbj_{-t}}\bbe_k-p^{-1}tr\bbQ_{\bbj_{-t}}\widehat{\bbM}_k^{-1}\bbQ_{\bbj_{-t}}$.
 It follows from \eqref{beta1} that
\bqn
&&\bga'_1\bbM^{-1}\bga_2-\E  \bga'_1\bbM^{-1}\bga_2
=\sum_{k=1}^p(\E _k-\E _{k-1})\bga'_1\bbM^{-1}\bga_2\\
&=&\sum_{k=1}^p(\E _k-\E _{k-1})(\bga'_1\bbM^{-1}\bga_2-\bga_{k1}'\bbM_k^{-1}\bga_{k2})\\
&=&\sum_{k=1}^p(\E _k-\E _{k-1})({\cal M}_1-{\cal M}_2-{\cal M}_3+{\cal M}_4),
\eqn
where 
\begin{gather*}
  {\cal M}_1=\frac{1}{\beta_kp^2}\bga_{k1}'\bbM_k^{-1}\bbE_k'\bbQ_{\bbj_{-t}}\bbe_k\bbe_k'\bbQ_{\bbj_{-t}}\bbE_k\bbM_k^{-1}\bga_{k2},~
  {\cal M}_2=\frac{\alpha_{k2}}{\beta_kp}\bga_{k1}'\bbM_k^{-1}\bbE_k'\bbQ_{\bbj_{-t}}\bbe_k\\
  {\cal M}_3=\frac{\alpha_{k1}}{\beta_kp}\bbe_k'\bbQ_{\bbj_{-t}}\bbE_k\bbM_k^{-1}\bga_{k2},~
  {\cal M}_4=\frac{\alpha_{k1}\alpha_{k2}}{\beta_k}.
\end{gather*} 
Note that for any fixed $z\in \mathbb{C}^+$,  we have  
 $$\min\{|\beta_k|, ~|\beta^{tr}_k|,~\|\bbM_k
\|,~\|\widehat\bbM_k
\|\}\geq \Im z= v>0,$$ and  $\|p^{-1/2}\bbE_k\|$ is almost surely bounded by a constant under   assumption (A2) (see \citep{BaiS98N,BaiS99E}, for example). Thus,  together with  the condition that  $\bga_1$ and $\bga_2$ are both   bounded in Euclidean norm, 
we conclude   that  for $m\geq 1$
\begin{gather}
\E|\xi_k|^{2m}\leq Cp^{-m}v^{-2m}\label{xik},\\
\E|\bga_{k1}'\bbM_k^{-1}\bbE_k'\bbQ_{\bbj_{-t}},\bbe_k\bbe_k'\bbQ_{\bbj_{-t}}\bbE_k\bbM_k^{-1}\bga_{k2}|^m\leq Cp^{m}v^{-2m}\\
\max\{\E|\bga_{k1}'\bbM_k^{-1}\bbE_k'\bbQ_{\bbj_{-t}}\bbe_k|^{2m},\E|\bbe_k'\bbQ_{\bbj_{-t}}\bbE_k\bbM_k^{-1}\bga_{k2}|^{2m}\}\leq Cp^{m}v^{-2m}.
\end{gather}
Here, we use the quadric	form inequality shown in Lemma 2.7 of \cite{BaiS98N}.  Thus, 
by the Burkholder inequality (see Lemma 2.1 in \citep{BaiS98N}),   we obtain 
\begin{align*}
 \E\left|\sum_{k=1}^p (\E _k-\E _{k-1}){\cal M}_1\right|^{4}&\leq C
\E \left(\sum_{k=1}^p\E _{k-1}|{\cal M}_{1}|^2\right)^2+ C
\E \sum_{k=1}^p|{\cal M}_{1}|^{4}\\
&=O(p^{-2}).
\end{align*}
Analogously, we also have 
\begin{align*}
  \E\left|\sum_{k=1}^p (\E _k-\E _{k-1}){\cal M}_2\right|^{4}=O(p^{-2})\mbox{~and~}  \E\left|\sum_{k=1}^p (\E _k-\E _{k-1}){\cal M}_3\right|^{4}=O(p^{-2}).
\end{align*}
For ${\cal M}_4$, by \eqref{xik}, \eqref{betr} and the Burkholder inequality, we obtain
\begin{align*}
   \E\left|\sum_{k=1}^p (\E _k-\E _{k-1}){\cal M}_4\right|^{4}=   \E\left|\sum_{k=1}^p (\E _k-\E _{k-1})\frac{\xi_k}{\beta_k\beta^{tr}_k}\right|^{4}=O(p^{-2}), 
\end{align*}
which finally implies 
$$\E|\bga'_1\bbM^{-1}\bga_2-\E  \bga'_1\bbM^{-1}\bga_2|^{4}=O(p^{-2}).$$
Therefore, by the Borel-Cantelli lemma,  we have that 
\begin{align}\label{A2}
	\bga'_1\bbM^{-1}\bga_2-\E  \bga'_1\bbM^{-1}\bga_2\asto0.
\end{align}
%
%
%

Because the entry distributions of $\bbE$ are identical, we know that the diagonal elements of $\E \bbM^{-1}$ are the same, denoted as $a_n(z)$ in the following. The off-diagonal elements are also the same, denoted as $b_n(z)$. From (\ref{beta1}), (\ref{betr}) and (\ref{xik}), we have
\begin{align}\label{an}
	a_n(z)&=\E \frac{1}{\beta_1}=\E\frac{1}{\beta^{tr}_1}+o(1).
\end{align}
Let $\gl_1\ge \cdots\ge\gl_n\ge 0$ be the eigenvalues of $p^{-1}\bbQ_{\bbj_{-t}}\bbE_1\bbE_1'\bbQ_{\bbj_{-t}}$ whose corresponding eigenvectors are denoted by $\bbv_1,\cdots,\bbv_n$. Note that the rank of $\bbQ_{\bbj_{-t}}$ is $n-k_{\bbj_{-t}}$, the dimension of the null space of $\bbQ_{\bbj_{-t}}$ is $k_{\bbj_{-t}}$,  
and  the null space of $p^{-1}\bbQ_{\bbj_{-t}}\bbE_1\bbE_1'\bbQ_{\bbj_{-t}}$ is not smaller than that of $\bbQ_{\bbj_{-t}}$. Thus, we may select the last $k_{\bbj_{-t}}$ eigenvectors from the null space of $\bbQ_{\bbj_{-t}}$, that is, we may assume that 
$\bbQ_{\bbj_{-t}}\bbv_i=0$, $i=n-k_{\bbj_{-t}}+1,\cdots,n$.  Therefore, the matrix $\widehat{\bbM}_1^{-1}\bbQ_{\bbj_{-t}}$ has $k_{\bbj_{-t}}$-fold eigenvalue $0$. By contrast, for all $i=1,\cdots,n-k_{\bbj_{-t}}$, $\bbQ_{\bbj_{-t}}\bbv_i\ne 0$ and
$\bbQ_{\bbj_{-t}}\bbv_i$ is also an eigenvector of $p^{-1}\bbQ_{\bbj_{-t}}\bbE_1\bbE_1'\bbQ_{\bbj_{-t}}$ corresponding to $\gl_i$, that is, 
$$
p^{-1}\bbQ_{\bbj_{-t}}\bbE_1\bbE_1'\bbQ_{\bbj_{-t}}\bbv_i=\gl_i\bbQ_{\bbj_{-t}}\bbv_i,
$$
which is equivalent to
$$
(p^{-1}\bbQ_{\bbj_{-t}}\bbE_1\bbE_1'\bbQ_{\bbj_{-t}}-z\bbI_n)^{-1}\bbQ_{\bbj_{-t}}\bbv_i=(\gl_i-z)^{-1}\bbQ_{\bbj_{-t}}\bbv_i.
$$
Thus, $(\gl_i-z)^{-1}$ is a nonzero eigenvalue of 
$\widehat{\bbM}_1^{-1}\bbQ_{\bbj_{-t}}$. Therefore, 
\bqn
\frac1p\rtr\widehat{\bbM}_1^{-1}\bbQ_{\bbj_{-t}}
=\frac1p\sum_{i=1}^{n-k_{\bbj_{-t}}}\frac{1}{\gl_i-z}
=\frac1p\sum_{i=1}^n\frac1{\gl_i-z}+\frac{k_{\bbj_{-t}}}{pz}=\frac1ptr\widehat{\bbM}_1^{-1}+\frac{k_{\bbj_{-t}}}{pz}.
\eqn
Since the spectra of $p^{-1}\bbQ_{\bbj_{-t}}\bbE_1\bbE_1'\bbQ_{\bbj_{-t}}$ and $p^{-1}\bbE_1'\bbQ_{\bbj_{-t}}\bbE_1$ differ by $|n-p-1|$ zero eigenvalues, it follows that
\begin{align*}
\E\frac1ptr\widehat{\bbM}_1^{-1}=\E{\frac1ptr{\bbM}_1^{-1}+\frac{1-n/p}{z}}
\to\underline{s}_t(z)+\frac{c-1}{cz}.
\end{align*}
Together with \eqref{an} and the last equation, we obtain
\begin{align}
\label{eq3.1.2}
a_n(z)\to-\frac1{z(1+\underline{s}_t(z)+\frac{c-1+\alpha_{m-t}}{cz})}.
\end{align}
Here,  $\alpha_{m-t}=\lim k_{\bbj_{-t}}/n.$

 Furthermore, from the inverse matrix formula, we obtain
 \bqn
 b_n(z)&=&\E  \left(\frac{\bbu_1'\bbM_1^{-1}\bbE_1'\bbQ_{\bbj_{-t}}\bbe_1}{p\beta_1}\right),
 \eqn
 where $\bbu_1=(1,0,\cdots,0)'$ is a $(p-1)$-dimensional vector, and $\bbe_1$ is the first column of $\bbE$.
 Because $\bbe_1$ is independent of $\bbE_1$, by the Cauchy-Schwarz inequality and equation 
\eqref{betr},
 we have
\begin{align}
\label{eq3.1.3}
 |b_n(z)|&=\left|\E   \left(\frac{\bbu_1'\bbM_1^{-1}\bbE_1'\bbQ_{\bbj_{-t}}\bbe_1\xi_1}{p\beta_1 \beta_1^{tr}}\right)\right|=O(p^{-1}).
\end{align} 
 Therefore, by combining \eqref{A2}-\eqref{eq3.1.3} and the $C_r$ inequality, the proof of  \eqref{le3.1.1} is complete.  

For the proof of (\ref{le3.1.2}), 
 using the inverse of block matrix again, we obtain 
 \begin{gather*}
\bga'_1\bbM^{-1}\bbE'\bga_2=\bga_{k1}'\bbM_k^{-1}\bbE_k'\bga_{2}+\frac{1}{\beta_kp^2}\bga_{k1}'\bbM_k^{-1}\bbE_k'\bbQ_{\bbj_{-t}}\bbe_k\bbe_k'\bbQ_{\bbj_{-t}}\bbE_k\bbM_k^{-1}\bbE_k'\bga_{2}\nonumber\\
 -\frac{\alpha_{k2}}{\beta_kp}\bga_{k1}'\bbM_k^{-1}\bbE_k'\bbQ_{\bbj_{-t}}\bbe_k-\frac{\alpha_{k1}}{\beta_kp}\bbe_k'\bbQ_{\bbj_{-t}}\bbE_k\bbM_k^{-1}\bbE_k'\bga_{2}+\frac{\alpha_{k1}\bbe_k'\bga_{2}}{\beta_k}.
 \end{gather*}
By means of the same procedure used in the proof of  \eqref{le3.1.1}, we can obtain that, almost surely,
$$\frac{1}{\sqrt{p}}\bga_1'\bbM^{-1}\bbE'\bga_2- \frac{1}{\sqrt{p}}\E\bga_1'\bbM^{-1}\bbE'\bga_2\to 0.$$

Next, we show that 
$
\frac{1}{\sqrt{p}}\E\bga_1'\bbM^{-1}\bbE'\bga_2=o(1).
$
Write $\bbM^{-1}=(M^{ij})$. Then, we have
\begin{align*}
  &\frac{1}{\sqrt{p}}\E \bga_1'\bbM^{-1}\bbE'\bga_2=\frac{1}{\sqrt{p}}\sum_{ij}\alpha_{i1}\E  M^{ij}\bbe_j'\bga_2\\
=&\frac{1}{\sqrt{p}}\sum_{i=1}^p\alpha_{i1}\left(\E  M^{11}\bbe_1'\bga_2+(p-1)\E  M^{12}\bbe_2'\bga_2\right)\\
=&\frac{1}{\sqrt{p}}\sum_{i=1}^p\alpha_{i1}\left(\E \frac{\bbe_1'\bga_2}{\beta_1}-(p-1)\E \frac{\bbu_1'\bbM_1^{-1}\bbE_1'\bbQ_{\bbj_{-t}}\bbe_1\bbe_2'\bga_2}{p\beta_1}\right)\\
=&\frac{1}{\sqrt{p}}\sum_{i=1}^p\alpha_{i1}\left(\E \frac{\xi_1\bbe_1'\bga_2}{\beta_1\beta_1^{tr}}
-(p-1)\E \frac{\bbu_1'\bbM_1^{-1}\bbE_1'\bbQ_{\bbj_{-t}}\bbe_1\bbe_2'\bga_2\xi_1}{p\beta_1\beta_1^{tr}}\right)\\
=&\frac{1}{p^{3/2}}\sum_{i=1}^p\alpha_{i1}\sum_{j=1}^n\E \frac{e_{j1}^3(\bbQ_{\bbj_{-t}}\widehat{\bbM}_1^{-1}\bbQ_{\bbj_{-t}})_{jj}\alpha_{j1}}{(\beta_1^{tr})^2}\\
&-\frac{1}{\sqrt{p}}\sum_{i=1}^p\alpha_{i1}\E \frac{\bbu_1'\bbM_1^{-1}\bbE_1'\bbQ_{\bbj_{-t}}\bbe_1\bbe_2'\bga_2\xi_1}{(\beta_1^{tr})^2}+o(1)\\
=&o(1),
\end{align*}
which completes the proof of (\ref{le3.1.2}).
In the above equation, we use \eqref{xik}, the Cauchy-Schwarz inequality, the $C_r$ inequality and the   facts that 
\bqn
\sum_{i=1}^p|\alpha_{i1}|\le\left(p\sum_{i=1}^p|\alpha_{i1}^2|\right)^{1/2} =O(\sqrt{p})
\eqn
and
\bqn
&&\left|\E \frac{\bbu_2'\bbM_1^{-1}\bbE_1'\bbQ_{\bbj_{-t}}\bbe_1\bbe_2'\bga_2\xi_1}{\sqrt{p}\beta_1\beta_1^{tr}}\right|\\
&\le &\left|\E \frac{\bbu_2'\bbM_1^{-1}\bbE_1'\bbQ_{\bbj_{-t}}\bbe_1\bbe_2'\bga_2\xi_1}{\sqrt{p}(\beta_1^{tr})^2}\right|
+
\E \left|\frac{(\bbu_2'\bbM_1^{-1}\bbE_1'\bbQ_{\bbj_{-t}}\bbe_1\bbe_2'\bga_2)\xi_1^2}{\sqrt{p}(\beta_1^{tr})^2\beta_1}\right|
\\
&=&O(p^{-1/2}).
\eqn

Next, we give  the proof of \eqref{le3.1.3}. The random part is analogous  to the proof of  \eqref{le3.1.1}, which is, almost surely
$$\frac{1}{{p}}\bga_1'\bbE\bbM^{-1}\bbE'\bga_2- \frac{1}{{p}}\E\bga_1'\bbE\bbM^{-1}\bbE'\bga_2\to 0,$$
and we omit the details here. Next, we focus on the nonrandom part. 
By means of the inverse of the block matrix, we have 
\begin{align*}
  &\frac{1}{{p}}\E \bga_1'\bbE\bbM^{-1}\bbE'\bga_2=\frac{1}{{p}}\sum_{ij}\E\bga_{1}'\bbe_i  M^{ij}\bbe_j'\bga_2\\
=&\E\bga_{1}'\bbe_1M^{11}\bbe_1'\bga_2+(p-1)\E\bga_{1}'\bbe_1 M^{12}\bbe_2'\bga_2)\\
=&\E \frac{\bga_{1}'\bbe_1\bbe_1'\bga_2}{\beta_1}-(p-1)\E \frac{\bga_{1}'\bbe_1\bbu_1'\bbM_1^{-1}\bbE_1'\bbQ_{\bbj_{-t}}\bbe_1\bbe_2'\bga_2}{p\beta_1}\\
=&\E \frac{\bga_{1}'\bga_2}{\beta_1^{tr}}
-(p-1)\E \frac{\bbu_1'\bbM_1^{-1}\bbE_1'\bbQ_{\bbj_{-t}}\bga_{1}\bbe_2'\bga_2}{p\beta_1^{tr}}\\
&+\E \frac{\xi_1\bga_{1}'\bbe_1\bbe_1'\bga_2}{\beta_1\beta_1^{tr}}
-(p-1)\E \frac{\bga_{1}'\bbe_1\bbu_1'\bbM_1^{-1}\bbE_1'\bbQ_{\bbj_{-t}}\bbe_1\bbe_2'\bga_2\xi_1}{p\beta_1\beta_1^{tr}}.
\end{align*}
It follows from \eqref{eq3.1.2} that 
\begin{align}\label{beta1tra12}
\E \frac{\bga_{1}'\bga_2}{\beta_1^{tr}}\to-\frac{\bga_{1}'\bga_2}{z(1+\underline{s}_t(z)+\frac{c-1+\alpha_{m-t}}{cz})}.
\end{align}
Using the  inverse of block matrix again, we obtain 
 \begin{gather*}
\E\bbu_1'\bbM_1^{-1}\bbE_1'\bbQ_{\bbj_{-t}}\bga_{1}\bbe_2'\bga_2\\
=\E\frac{\bbe_2'\bbQ_{\bbj_{-t}}\bga_{1}\bbe_2'\bga_2}{\beta_{12}}-\E\frac{1}{\beta_{12}p}\bbe_2'\bbQ_{\bbj_{-t}}\bbE_{12}\bbM_{12}^{-1}\bbE_{12}'\bbQ_{\bbj_{-t}}\bga_{1}\bbe_2'\bga_2,
 \end{gather*}
which together with \eqref{betr}, \eqref{xik}  and  in-out-exchange foumula \eqref{in-out} implies 
\begin{align*}
  &\E \frac{\bbu_1'\bbM_1^{-1}\bbE_1'\bbQ_{\bbj_{-t}}\bga_{1}\bbe_2'\bga_2}{\beta_1^{tr}}\\=&\E\frac{\bga_{1}'\bbQ_{\bbj_{-t}}\bga_2}{(\beta_{12}^{tr})^2}-\E\frac{\bga_{1}\bbQ_{\bbj_{-t}}\bbE_{12}\bbM_{12}^{-1}\bbE_{12}'\bbQ_{\bbj_{-t}}\bga_2
}{p(\beta_{12}^{tr})^2}+o(1)\\
=&-\E\frac{z\bga_{1}\bbQ_{\bbj_{-t}}\widehat\bbM_{12}^{-1}\bbQ_{\bbj_{-t}}\bga_2
}{(\beta_{12}^{tr})^2}+o(1).
\end{align*}
Here,  $\bbE_{12}$ is the $n\times (p-2)$ submatrix of $\bbE$ with the first  and second  columns removed. 
$\beta_{12}$, $\beta_{12}^{tr}$, $\bbM_{12}^{-1}$  and $\widehat\bbM_{12}^{-1}$ are denoted analogously. Note that 
$\E \frac{\xi_1\bga_{1}'\bbe_1\bbe_1'\bga_2}{\beta_1\beta_1^{tr}}
$ and $ \E \frac{\bga_{1}'\bbe_1\bbu_1'\bbM_1^{-1}\bbE_1'\bbQ_{\bbj_{-t}}\bbe_1\bbe_2'\bga_2\xi_1}{\beta_1\beta_1^{tr}}$ are both $o(1)$ because of \eqref{betr} and \eqref{xik}.  Similar to (5.9) in \cite{BaiM07A},  we have 
\begin{align*}
  \E \bga_{1}'\bbQ_{\bbj_{-t}}\widehat\bbM_{12}^{-1}\bbQ_{\bbj_{-t}}\bga_2-  \bga_{1}\bbQ_{\bbj_{-t}}(-z \underline{s}_t(z) \bbQ_{\bbj_{-t}}-z\bbI)^{-1}\bbQ_{\bbj_{-t}}\bga_2\to0.
\end{align*}
As $\bga_{1}'\bbQ_{\bbj_{-t}}$ and $\bbQ_{\bbj_{-t}}\bga_2$ are both eigenvectors of $\bbQ_{\bbj_{-t}}$, we obtain
\begin{align*}
  \bga_{1}\bbQ_{\bbj_{-t}}(z \underline{s}_t(z) \bbQ_{\bbj_{-t}}+z\bbI)^{-1}\bbQ_{\bbj_{-t}}\bga_2=\frac{\bga_{1}\bbQ_{\bbj_{-t}}\bga_2}{z \underline{s}_t(z)+z},
\end{align*}
which together with \eqref{beta1tra12}
 and the fact that 
 \begin{align*}
  \E \beta_1^{-1}-\E \beta_{12}^{-1}\to0,
\end{align*}
 completes the proof of  \eqref{le3.1.3}. Thus, the proof of this lemma is complete. 
%
\end{proof}

\end{document}